\documentclass{article}
\usepackage{amssymb}
\usepackage{amsmath}

\usepackage{psfrag}
\usepackage[dvips]{graphicx}
\usepackage{epsfig}
\usepackage{wrapfig}

\newtheorem{theorem}{Theorem}

\newtheorem{corollary}[theorem]{Corollary}

\newtheorem{definition}{Definition}

\newtheorem{lemma}{Lemma}

\newenvironment{proof}[1][Proof]{\noindent\textbf{#1.} }{\ \rule{0.5em}{0.5em}}

\begin{document}
\title{\textbf{Formation and Propagation of Discontinuity for\\ Boltzmann Equation in Non-Convex Domains}}
\author{\large{\textbf{Chanwoo Kim}}\\ \small{Department of Mathematics, Brown University, Providence, RI 02917, USA.
E-mail : ckim@math.brown.edu}}
\maketitle

\begin{abstract}
The formation and propagation of singularities for Boltzmann equation in bounded domains has been an important question in numerical studies as well as in theoretical studies. Consider the nonlinear Boltzmann solution near Maxwellians under in-flow, diffuse, or bounce-back boundary conditions. We demonstrate that discontinuity is created at the non-convex part of the grazing boundary, then propagates only along the forward characteristics inside the domain before it hits on the boundary again.
\end{abstract}

\section{Introduction}
A density of a dilute gas is governed by the Boltzmann equation
\begin{equation}
\partial _{t}F + v\cdot \nabla _{x}F = Q(F,F),  \label{boltzmann}
\end{equation}%
where $F(t,x,v)$ is a distribution function for the gas particles at a time $t\geq 0,$ a position $x\in \Omega\subset \mathbb{R}^3 \ $and a velocity $v\in \mathbb{R}^{3}.$ Throughout
this paper, the collision operator takes the form%
\begin{eqnarray}
Q(F_{1},F_{2})
&=&\int_{\mathbb{R}^{3}}\int_{\mathbb{S}^{2}} B(v-u,\omega)
F_{1}(u^{\prime })F_{2}(v^{\prime })d\omega du
-\int_{\mathbb{R}^{3}}\int_{\mathbb{S}^{2}} B(v-u,\omega) F_{1}(u)F_{2}(v)d\omega du  \notag \\
&\equiv &Q_{+}(F_{1},F_{2})-Q_{-}(F_{1},F_{2}),
\label{qgl}
\end{eqnarray}%
where $u^{\prime }=u+[(v-u)\cdot \omega ]\omega ,$ $v^{\prime
}=v-[(v-u)\cdot \omega ]\omega $ and
\begin{equation*}
B(v-u,\omega) = |v-u|^{\gamma} q_0 (\frac{v-u}{|v-u|}\cdot \omega),
\end{equation*}
with $0< \gamma \leq 1$ (hard potential) and
\begin{equation}
\int_{\mathbb{S}^2} q_0 (\hat{u}\cdot \omega) d\omega < +\infty, \ \ \ \text{(angular cutoff)}\label{Gradcutoff}
\end{equation}
for all $\hat{u}\in\mathbb{S}^2$.

If the gas is contained in a bounded region or flows past a solid bodies, the Boltzmann equation must be accompanied by boundary conditions, which describe the interaction of the gas molecules with the solid walls. Let the domain $\Omega$ be a smooth bounded domain. We consider three basic types of boundary conditions(\cite{Guo08},\cite{V}) for $F(t,x,v)$ at $(x,v)\in\partial\Omega\times\mathbb{R}^3$ with $v\cdot n(x)<0$, where $n(x)$ is an outward unit normal vector at $x$ :
\\
\newline 1. \textbf{In-flow injection boundary condition} : incoming particles are prescribed ;
\begin{equation}
F(t,x,v) = G(t,x,v).\label{Finflow}
\end{equation}
\newline 2. \textbf{Diffuse reflection boundary condition} : incoming particles are a probability average of the outgoing particles ;
\begin{equation}
F(t,x,v) = c_{\mu}\mu(v) \int_{v^{\prime}\cdot n(x)>0} F(t,x,v^{\prime})\{n(x)\cdot v^{\prime}\}dv^{\prime},
\label{Fdiffuse}
\end{equation}
with a normalized Maxwellian $\mu = e^{-\frac{|v|^2}{2}}$, a normalized constant $c_{\mu}>0$ such that
\begin{equation}
c_{\mu}\int_{v^{\prime}\cdot n(x)>0} \mu(v^{\prime})|n(x)\cdot v^{\prime}|dv^{\prime} =1.\label{normalizedconstant}
\end{equation}
\newline 3. \textbf{Bounce-back reflection boundary condition} : incoming particles bounce back at the reverse the velocity ;
\begin{equation}
F(t,x,v) = F(t,x,-v).\label{Fbounceback}
\end{equation}

The purpose of this paper is to investigate possible formation and propagation of discontinuity for the nonlinear Boltzmann equation under those boundary conditions. In order to state our results, we need following definitions.

\subsection{Domain}
Throughout this paper, we assume the domain $\Omega\subset\mathbb{R}^3$ is open and bounded and connected. For simplicity, We assume that the boundary $\partial\Omega$ is smooth, i.e. for each point $x_0 \in \partial\Omega$, there exists $r>0$ and a smooth function $\Phi_{x_0} : \mathbb{R}^2 \rightarrow \mathbb{R}$ such that - upon relabeling and reorienting the coordinates axes if necessary - we have
\begin{equation}
\Omega \cap B(x_0,r) = \{ x\in B(x_0 ,r) : x_3 > \Phi_{x_0}(x_1 ,x_2)
\}.\label{boundaryfunction}
\end{equation}
The outward normal vector at $
\partial \Omega $ is given by
\begin{equation*}
n(x_1 ,x_2)=\frac{1}{\sqrt{1+|\nabla_x \Phi (x_1,x_2)|^2}} ( \ \partial_{x_1}\Phi_{x_0}(x_1,x_2) , \ \partial_{x_2}\Phi_{x_0}(x_1,x_2) , \ -1 \ ).  \label{outwardnormal}
\end{equation*}
Given $(t,x,v)$, let $[X(s),V(s)]=[X(s;t,x,v),V(s;t,x,v)]=[x-(t-s)v,v]$ be
a trajectory (or a characteristics) for the Boltzmann equation (\ref%
{boltzmann}) :
\begin{equation*}
\frac{dX(s)}{ds}=V(s),\text{ \ \ \ \ \ \ \ \ \ \ \ \ \ \ \ \ \ \ }\frac{dV(s)%
}{ds}=0,  \label{ode}
\end{equation*}%
with the initial condition : $[X(t;t,x,v),V(t;t,x,v)]=[x,v].$
\begin{definition}
For $(x,v)\in\bar{\Omega}\times\mathbb{R}^3$, we
define the \textbf{backward exit time},   $t_{\mathbf{b}}(x,v)\geq 0$ to be the last moment at which the
back-time straight line $[X(s;0,x,v),V(s;0,x,v)]$ remains in the interior of
$ \ \Omega :$
\begin{equation*}
t_{\mathbf{b}}(x,v)=\sup(\{ 0 \}\cup \{\tau>0 : x-sv\in \Omega \ \text{for all } \ 0 < s < \tau \}).
\end{equation*}
We also define the \textbf{backward exit position} in $\partial\Omega$
\begin{equation*}
x_{\mathbf{b}}(x,v)=x-t_{\mathbf{b}}(x,v)v\in \partial \Omega,
\end{equation*}
and  we always have $v\cdot n(x_{\mathbf{b}}(x,v))\leq 0.$
\end{definition}
\subsection{Discontinuity Set and Discontinuity Jump}
We denote the phase boundary in the phase space $\Omega \times \mathbb{R}^{3}$ as $%
\gamma =\partial \Omega \times \mathbb{R}^{3},$ and split it into outgoing
boundary $\gamma _{+},$ the incoming boundary $\gamma _{-},$ and the
grazing boundary $\gamma _{0}$ $:$
\begin{eqnarray*}
\gamma _{+} &=&\{(x,v)\in \partial \Omega \times \mathbb{R}^{3}:\text{ \ }%
n(x)\cdot v>0\}, \\
\gamma _{-} &=&\{(x,v)\in \partial \Omega \times \mathbb{R}^{3}:\text{ \ }%
n(x)\cdot v<0\}, \\
\gamma _{0} &=&\{(x,v)\in \partial \Omega \times \mathbb{R}^{3}:\text{ \ }%
n(x)\cdot v=0\}.
\end{eqnarray*}
We need to study the grazing boundary $\gamma_0$ more carefully.
\begin{figure}[h]
\begin{center}
\epsfig{file=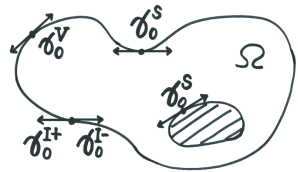,angle=90, height=3cm}
\end{center}
\caption{Grazing Boundary $\gamma_0$}
\label{fig1}
\end{figure}
\begin{definition}\label{grazingboundary}
We define the \textbf{concave(singular) grazing boundary} which is a subset of the grazing boundary $\gamma_0$ :
\begin{equation*}
\gamma_0^{\mathbf{S}} = \{(x,v)\in\gamma_0 : t_{\mathbf{b}}(x,v)\neq 0 \ \text{and} \ t_{\mathbf{b}}(x,-v)\neq 0\},
\end{equation*}
and the \textbf{outward inflection grazing boundary} in the grazing boundary $\gamma_0$ :
\begin{eqnarray*}
\gamma_0^{I+} =\{ (x,v) \in \gamma_0 : t_{\mathbf{b}}(x,v) \neq 0 \ \text{and} \ t_{\mathbf{b}}(x,-v)=0 \ \text{and there is \ } \delta>0 \ \text{such that } \ x+\tau v\in\bar{\Omega}^{c} \ \text{ for } \ \tau\in (0,\delta)
\},
\end{eqnarray*}
and the \textbf{inward inflection grazing boundary} in the grazing boundary $\gamma_0$ :
\begin{eqnarray*}
\gamma_0^{I-}=\{ (x,v) \in \gamma_0 : t_{\mathbf{b}}(x,v)=0 \ \text{and} \ t_{\mathbf{b}}(x,-v) \neq 0  \ \text{and there is \ } \delta>0 \ \text{such that } \ x-\tau v\in\bar{\Omega}^{c} \ \text{ for } \ \tau\in (0,\delta)
\},
\end{eqnarray*}
and the \textbf{convex grazing boundary} in the grazing boundary $\gamma_0$ :
\begin{eqnarray*}
\gamma_0^V = \{ (x,v)\in\gamma_0 : t_{\mathbf{b}}(x,v)=0 \ \text{and} \ \ t_{\mathbf{b}}(x,-v)=0
\}.
\end{eqnarray*}
\end{definition}
It turns out that the concave (singular) grazing boundary $\gamma_0^{\mathbf{S}} \ $ is the only part at which discontinuity can be created and propagates into the interior of the phase space $\Omega\times\mathbb{R}^3$.
\begin{definition}
Define the \textbf{discontinuity set} in $[0,\infty) \times\bar\Omega\times\mathbb{R}^3$ as
\begin{eqnarray}
\mathfrak{D} =
\Big\{(0,\infty)\times [ \ \gamma_0^{\mathbf{S}}\cup\gamma_0^V \cup\gamma_0^{I+} \ ]\Big\} \cup\Big\{ \ (t,x,v) \in (0,\infty)\times\{\Omega\times\mathbb{R}^3 \cup \gamma_+\} :  \ t\geq t_{\mathbf{b}}(x,v) \text{ and } \ (x_{\mathbf{b}}(x,v),v)\in\gamma_0^{\mathbf{S}} \ \Big\}
,\label{D}
\end{eqnarray}
and the \textbf{continuity set} in $[0,\infty)\times\bar\Omega\times\mathbb{R}^3$ as
\begin{eqnarray}
&&\mathfrak{C} \ = \ \Big\{ \{0\}\times \bar{\Omega}\times\mathbb{R}^3 \Big\} \cup
\Big\{ (0,\infty)\times [ \ \gamma_- \cup \gamma_0^{I-}\ ] \Big\}\nonumber\\
&& \ \ \ \ \ \ \cup \ \Big\{ \ (t,x,v) \in (0,\infty)\times\{\Omega\times\mathbb{R}^3 \cup \gamma_+\}: \ t<t_{\mathbf{b}}(x,v) \ \ \text{or} \ \ (x_{\mathbf{b}}(x,v),v) \in \gamma_- \cup \gamma_0^{I-} \
\Big\}.\label{C}
\end{eqnarray}
For bounce-back reflection boundary condition case (\ref{Fbounceback}) \  we need slightly different definitions : the \textbf{bounce-back discontinuity set} and the \textbf{bounce-back continuity set} are
\begin{eqnarray*}
&&\mathfrak{D}_{bb} \ = \
\mathfrak{D}  \ \cup \ \Big\{ (t,x,v) \in (0,\infty)\times\Omega\times\mathbb{R}^3 :  t\geq 2t_{\mathbf{b}}(x,v)+t_{\mathbf{b}}(x,-v) \text{ and } (x_{\mathbf{b}}(x,-v),-v)\in\gamma_0^{\mathbf{S}} \Big\},\\
&&\mathfrak{C}_{bb} \ = \ \Big\{ \{0\}\times \bar{\Omega}\times\mathbb{R}^3 \Big\} \cup
\Big\{ (0,\infty)\times [ \ \gamma_- \cup \gamma_0^{I-}\ ] \Big\}\\
&& \ \ \ \ \  \cup \ \Big\{ (t,x,v) \in [0,\infty)\times\{\Omega\times\mathbb{R}^3 \cup\gamma_+ \}: \ t<t_{\mathbf{b}}(x,v)   \ \text{or} \ \ \left[ \ (x_{\mathbf{b}}(x,v),v)\in\gamma_- \cup \gamma_0^{I-} \ \ \text{and} \ \ t<2t_{\mathbf{b}}(x,v)+t_{\mathbf{b}}(x,-v) \ \right] \ \ \ \ \ \ \ \ \\
&& \ \ \ \ \ \ \ \ \ \ \ \ \ \ \ \ \ \ \ \ \ \ \ \ \ \ \ \ \ \ \ \ \ \ \ \ \ \ \ \ \ \ \ \ \  \text{or} \ \ \left[ \ (x_{\mathbf{b}}(x,-v),-v)\in \gamma_- \cup \gamma_0^{I-} \ \ \text{and} \ \
(x_{\mathbf{b}}(x,v),v)\in\gamma_- \cup \gamma_0^{I-}
\ \right]
\Big\},
\label{Dbb}
\end{eqnarray*}
respectively.
\end{definition}

The discontinuity set $\mathfrak{D}$ consists of two sets : The first set of (\ref{D}) is the grazing boundary part $\gamma_0$ of $\mathfrak{D}$. This set mainly consists of the phase boundary where the backward exit time $t_{\mathbf{b}}(x,v)$ is not continuous (Lemma \ref{tbcon}). The second set of (\ref{D}) is mainly the interior phase space part of $\mathfrak{D}$, i.e. $\mathfrak{D}\cap \{[0,\infty)\times \Omega\times \mathbb{R}^3\}$, which is a subset of a union of all forward trajectories in the phase space emanating from $\gamma_0^{\mathbf{S}}$. Notice that $\mathfrak{D}$ does not include the forward trajectories emanating from $\gamma_0^{V} \cup \gamma_0^{I+}$ because those forward trajectories are not in the interior phase space $[0,\infty)\times\Omega\times\mathbb{R}^3$. We also exclude the case $t<t_{\mathbf{b}}(x,v)$ from $\mathfrak{D}$. In fact, considering the pure transport equation, $t<t_{\mathbf{b}}(x,v)$ implies the transport solution at $(t,x,v)$ equals the initial data at $(x-tv,v)$ and if the initial data is continuous, we expect the transport solution is continuous around $(t,x,v)$. Notice that we exclude the initial plane $\{0\}\times\bar{\Omega}\times\mathbb{R}^3$ from $\mathfrak{D}$ because we assume that the Boltzmann solution is continuous at $t=0$ .

The continuity set $\mathfrak{C}$ consists of points either emanating from the initial plane or from $\gamma_- \cup \gamma_0^{I-}$, but not $\gamma_0^{\mathbf{S}}$.\\
\newline Furthermore we define a set including the grazing boundary $\gamma_0$ and all forward trajectories emanating from the whole grazing boundary $\gamma_0$.
\begin{definition}\label{grazingset}
The \textbf{grazing set} is defined as
\begin{eqnarray}
\mathfrak{G} &=&\{(x,v)\in\bar{\Omega}\times\mathbb{R}^3 : n(x_{\mathbf{b}}(x,v))\cdot v=n(x-t_{\mathbf{b}}(x,v)v)\cdot v =0
\},\label{G2}
\end{eqnarray}
and the \textbf{grazing section} is
\begin{eqnarray*}
\mathfrak{G}_x = \{v\in\mathbb{R}^3 : (x,v)\in \mathfrak{G}\}
=\{v\in\mathbb{R}^3 : n(x_{\mathbf{b}}(x,v))\cdot v =0\}.
\end{eqnarray*}
\end{definition}
Obviously the grazing set $\mathfrak{G}$ includes the discontinuity set $\mathfrak{D}$. In order to study the continuity property of the Boltzmann solution we define :
\begin{definition}
For a function $\phi(t,x,v)$ defined on $[0,\infty)\times\{\bar{\Omega}\times\mathbb{R}^3\backslash \mathfrak{G}\}$ we define the \textbf{discontinuity jump in space and velocity}
\begin{equation*}
[\phi(t)]_{x,v} \ = \ \lim_{\delta\downarrow 0}\sup_{\begin{scriptsize}(x^{\prime},v^{\prime}), (x^{\prime\prime},v^{\prime\prime}) \in  \{\bar{\Omega}\times\mathbb{R}^3 \backslash \mathfrak{G}\} \cap \{B((x,v);\delta) \backslash (x,v)\}\end{scriptsize}} |\phi(t,x^{\prime},v^{\prime})-\phi(t,x^{\prime\prime},v^{\prime\prime})|,
\end{equation*}
and the \textbf{discontinuity jump in time and space and velocity}
\begin{equation*}
[\phi]_{t,x,v}  \ = \ \lim_{\delta\downarrow 0}\sup_{\begin{scriptsize}\begin{array}{ccc} t^{\prime},t^{\prime\prime}\in B(t;\delta)\\(x^{\prime},v^{\prime}), (x^{\prime\prime},v^{\prime\prime})\in \{\bar{\Omega}\times\mathbb{R}^3 \backslash \mathfrak{G}\} \cap \{B((x,v);\delta) \backslash (x,v)\}  \end{array}\end{scriptsize}} |\phi(t^{\prime},x^{\prime},v^{\prime})-\phi(t^{\prime\prime},x^{\prime\prime},v^{\prime\prime})|,
\end{equation*}
where $\mathfrak{G}$ is defined in Definition \ref{grazingset}.
We say a function $\phi$ is discontinuous in space and velocity (in time and space and velocity) at $(t,x,v)$ if and only if $[\phi(t)]_{x,v} \neq 0 \ ([\phi]_{t,x,v} \neq 0)$ and continuous in space and velocity (in time and space and velocity) at $(t,x,v)$ if and only if $[\phi(t)]_{x,v} = 0 \ ([\phi]_{t,x,v} = 0)$.
\end{definition}

Notice that the function $\phi$ is only defined away from the grazing set $\mathfrak{G}$. Once the discontinuity jump of given function $\phi$ is zero at $(t,x,v)$ then the function $\phi$ can be extended to $[0,\infty)\times\bar{\Omega}\times\mathbb{R}^3$ near $(t,x,v)$. Because of those definitions we can consider a function which has a removable discontinuity as a continuous function. And a non-zero discontinuity jump $[\phi]_{t,x,v}\neq 0$ means $\phi$ has a "real" discontinuity which is not removable.

\subsection{Main Result}

Now we are ready to state the main theorems of this paper. In order to state theorems in the unified way we use a weight function
\begin{equation}
w(v) = \{1+\rho^2 |v|^2 \}^{\beta},\label{weight}
\end{equation}
such that $w^{-2}(v)\{1+|v|\}^3 \in L^1$.
\begin{theorem}[Formation of Discontinuity]\label{formationofsingularity}
Let $\Omega$ be an open subset of $\mathbb{R}^3$ with a smooth boundary $\partial\Omega$. Assume $\Omega$ is non-convex, $\gamma_0^{\mathbf{S}} \neq \emptyset.$ Choose any $(x_0,v_0)\in\gamma_0^{\mathbf{S}}$ with $v_0 \neq 0$. For any small $\delta>0$,
\newline 1. There exist $t_0\in (0,\min\{\delta,t_{\mathbf{b}}(x_0,-v_0)\})$ and an initial datum $F_0(x,v)$ which is continuous on $\Omega\times\mathbb{R}^3 \cup \{\gamma_- \cup \gamma_0^{\mathbf{S}}\}$, and an in-flow boundary datum $G(t,x,v)$ which is continuous on $[0,\infty)\times\{\gamma_-\cup\gamma_0^{\mathbf{S}}\}$, satisfying
\begin{equation}
F_0 (x,v) = G(0,x,v) \ \ \text{for} \ \ (x,v)\in\gamma_-\cup \gamma_0^{\mathbf{S}},\label{Cinflow}
\end{equation}
and
\begin{equation}
\Big{|}\Big{|}
\ w \ \frac{F_0-\mu}{\sqrt{\mu}}
\Big{|}\Big{|}_{L^{\infty}(\bar{\Omega}\times\mathbb{R}^3)} + \sup_{t\in[0,\infty)} \Big|\Big| \ w \ \frac{G(t)-\mu}{\sqrt{\mu}} \Big|\Big|_{L^{\infty}(\gamma_- )} < \delta,\label{gsmallness}
\end{equation}
such that if $F$ on $[0,\infty)\times\bar{\Omega}\times\mathbb{R}^3$ is Boltzmann solution of (\ref{boltzmann}) with the in-flow boundary condition (\ref{Finflow}) then $F$ is discontinuous in space and velocity at $(t_0,x_0,v_0)$, i.e. $[F(t_0)]_{x_0,v_0}\neq 0.$
\newline 2. There exist $t_0\in (0,\min\{\delta,t_{\mathbf{b}}(x_0,-v_0)\})$ and an initial datum $F_0(x,v)$ which is continuous on $\Omega\times\mathbb{R}^3 \cup \{\gamma_- \cup \gamma_0^{\mathbf{S}}\}$, satisfying
\begin{equation}
F_0 (x,v) = c_{\mu}\mu(v) \int_{v^{\prime}\cdot n(x)>0} F_0(x,v^{\prime})\{n(x)\cdot v^{\prime}\}dv^{\prime} \ \ \text{for} \ \ (x,v)\in\gamma_-\cup \gamma_0^{\mathbf{S}},\label{Cdiffuse}
\end{equation}
and
\begin{equation}
\Big{|}\Big{|}
\ w \ \frac{F_0-\mu}{\sqrt{\mu}}
\Big{|}\Big{|}_{L^{\infty}(\bar{\Omega}\times\mathbb{R}^3)} < \delta,\label{smallness}
\end{equation}
such that if $F$ on $[0,\infty)\times\bar{\Omega}\times\mathbb{R}^3$ is Boltzmann solution of (\ref{boltzmann}) with the diffuse boundary condition (\ref{Fdiffuse}) with the compatibility condition (\ref{Cdiffuse}) then $F$ is discontinuous in space and velocity at $(t_0,x_0,v_0)$, i.e. $[F(t_0)]_{x_0,v_0}\neq 0.$
\newline 3. There exist $t_0\in (0,\min\{\delta, t_{\mathbf{b}}(x_0,-v_0), t_{\mathbf{b}}(x_0,v_0)\})$ and an initial datum $F_0(x,v)$ which is continuous on $\Omega\times\mathbb{R}^3 \cup \{\gamma_- \cup \gamma_0^{\mathbf{S}}\}$, satisfying (\ref{smallness}) and
\begin{equation}
F_0(x,v)=F_0(x,-v) \ \ \text{for} \ \ (x,v)\in\gamma_-\cup \gamma_0^{\mathbf{S}},\label{Cbounceback}
\end{equation}
such that if $F$ on $[0,\infty)\times\bar{\Omega}\times\mathbb{R}^3$ is Boltzmann solution of (\ref{boltzmann}) with the bounce-back boundary condition (\ref{Fbounceback}) then $F$ is discontinuous in space and velocity at $(t_0,x_0,v_0)$, i.e. $[F(t_0)]_{x_0,v_0}\neq 0.$
\end{theorem}
The smallness of given data (\ref{gsmallness}), (\ref{smallness}) ensures the global existence of Boltmzann solution for all boundary conditions \cite{Guo08}.  Notice that we can observe the formation of discontinuity for any point of the concave(singular) grazing boundary $\gamma_0^{\mathbf{S}}$ of any generic non-convex domain $\Omega$. If we assume that $F_0(x,v)$ is continuous on $\Omega\times\mathbb{R}^3 \cup \{\gamma_- \cup \gamma_0^{\mathbf{S}}\}$ and $G(t,x,v)$ is continuous on $[0,\infty)\times\{\gamma_-\cup\gamma_0^{\mathbf{S}}\}$, and that the compatibility conditions (\ref{Cinflow}), (\ref{Cdiffuse}) and (\ref{Cbounceback}) are valid up to $\gamma_- \cup\gamma_0^{\mathbf{S}}$, then Theorem 1 implies the continuity breaks down at the concave(singular) grazing boundary $\gamma_0^{\mathbf{S}}$ after a short time $t_0 \in (0,\min\{\delta,t_{\mathbf{b}}(x_0,-v_0)\})$ for in-flow (\ref{Finflow}), diffuse (\ref{Fdiffuse}) boundary condition and $t_0 \in (0,\min\{\delta, t_{\mathbf{b}}(x_0,-v_0),t_{\mathbf{b}}(x_0,v_0)\})$ for bounce-back (\ref{Fbounceback}) boundary condition. For this generic cases, we said the Boltzmann solution $F$ has a \textbf{local-in-time formation of discontinuity} at $(t_0,x_0,v_0)$.\\
\\
Once we have the formation of discontinuity at $(t_0,x_0,v_0)\in\gamma_0^{\mathbf{S}}$, we further establish that the discontinuity propagates along the forward characteristics.
\begin{theorem}[Propagation of Discontinuity]
\label{propagation} Let $\Omega$ be an open bounded subset of $\mathbb{R}^3$ with a smooth boundary $\partial\Omega$.
Let $F(t,x,v)$ be the Boltzmann solution of (\ref{boltzmann}) with the initial datum $F_0$ which is continuous on $\Omega\times\mathbb{R}^3 \cup \{\gamma_- \cup \gamma_0^{\mathbf{S}}\}$, and with one of the following boundary conditions :
\newline 1. For in-flow boundary condition (\ref{Finflow}), let (\ref{Cinflow}) and (\ref{gsmallness}) be valid and $G(t,x,v)$ be continuous on $[0,\infty)\times\{\gamma_-\cup\gamma_0^{\mathbf{S}}\}$.
\newline 2. For diffuse boundary condition (\ref{Fdiffuse}), assume (\ref{smallness}) and (\ref{Cdiffuse}).
\newline 3. For bounce-back boundary condition (\ref{Fbounceback}), assume (\ref{smallness}) and (\ref{Cbounceback}).
\newline Then for all $t\in [t_0, t_0 + t_{\mathbf{b}}(x_0,-v_0))$ we have
\begin{equation}
[F]_{t,x_0 +(t-t_0)v_0,v_0}\leq  e^{-C_1(1+|v_0|)^{\gamma}(t-t_0)}[F(t_0)]_{x_0,v_0},\label{1sidedinequality}
\end{equation}
where $C_1>0$ only depends on $\big{|}\big{|}
 w  \frac{F-\mu}{\sqrt{\mu}}
\big{|}\big{|}_{L^{\infty}([0,\infty)\times\bar{\Omega}\times\mathbb{R}^3)}$.

On the other hand, assume $[F(t_0)]_{x_0,v_0}\neq 0$, and $t_0 \in (0,t_{\mathbf{b}}(x_0,-v_0))$ for in-flow and diffuse boundary conditions and $t_0 \in (0,\min\{t_{\mathbf{b}}(x_0,-v_0),t_{\mathbf{b}}(x_0,v_0)\})$ for bounce-back boundary condition, and a strict concavity of $ \partial\Omega$ at $x_0$ along $v_0$, i.e.
\begin{equation}
\sum_{i,j} (v_0)_{i} \partial_{x_i} \partial_{x_j} \Phi(x_0) (v_0)_j < -C_{x_0,v_0}.\label{convexity}
\end{equation}
Then for all $t\in [t_0, t_0 + t_{\mathbf{b}}(x_0,-v_0))$, the Boltzmann solution $F$ is discontinuous in time and space and velocity at $(t,x_0+ (t-t_0)v_0,v_0)$, i.e. $[F]_{t,x_0 +(t-t_0)v_0,v_0}\neq 0$ and
\begin{equation}
C   e^{-C_2 (1+|v_0|)^{\gamma}(t-t_0)}[F(t_0)]_{x_0,v_0}\leq [F]_{t,x_0 +(t-t_0)v_0,v_0},\label{decaydiscontinuity}
\end{equation}
where $0<C<1$, and $C_2 = C_2 (\big{|}\big{|}
 w  \frac{F-\mu}{\sqrt{\mu}}
\big{|}\big{|}_{L^{\infty}})\in\mathbb{R}$ which is positive for sufficiently small $\big{|}\big{|}
 w  \frac{F-\mu}{\sqrt{\mu}}
\big{|}\big{|}_{L^{\infty}([0,\infty)\times\bar{\Omega}\times\mathbb{R}^3)}$.
\end{theorem}
The strict concavity condition (\ref{convexity}) rules out some technical issue of the backward exit time $t_{\mathbf{b}}$. Our theorem characterize the propagation of discontinuity before the forward trajectory reaches the boundary. In the case that the forward trajectory reaches the boundary, i.e. $t\geq t_0+t_{\mathbf{b}}(x_0,-v_0)$, the situation is much more complicated. Denote $x_1 = x_0+t_{\mathbf{b}}(x_0,-v_0)v_0, t_1 =t_0 +t_{\mathbf{b}}(x_0,-v_0)$. If the trajectory hits on the boundary non-tangentially, i.e. $(x_1,v_0)\in\gamma_+$, for in-flow and diffuse boundary cases, the discontinuity disappears because of the continuity of the in-flow datum and the average property of diffuse boundary operator. For bounce-back case the discontinuity is reflected and continues to propagate along the trajectory. If the trajectory hits on the boundary tangentially, i.e. $(x_1,v_0)\in\gamma_0$, there are three possibilities. Firstly if $(x_1,v_0)\in\gamma_0^{I+}$ then the situation is same as the case $(x_1,v_0)\in\gamma_+$ above. Secondly, if the trajectory is contained in the boundary for awhile, i.e. there exists $\delta>0$ so that $x_1 +s v_0\in\partial\Omega$ for $s\in (0,\delta)$ then it is hard to predict the propagation of discontinuity in general. Assuming certain condition on $\Omega$ for example Definition \ref{linesegment}, we can rule such a unlikely case.

The last case is that $(x_1,v_0)\in\gamma_0^{\mathbf{S}}$. Assume we have a sequence of $\{t_n=t_{n-1}+t_{\mathbf{b}}(x_{n-1},-v_0)\}$ and $\{x_n=x_{n-1}+t_{\mathbf{b}}(x_{n-1},-v_0)v_0\}\in\partial\Omega$ so that $(x_n,v_0)\in\gamma_0^{\mathbf{S}}$, and a directional strict concavity (\ref{convexity}) is valid for each $(x_n,v_0)$. We can show the propagation of discontinuity also between the first and the second intersections, i.e. $[F]_{t,x_0 (t-t_0)v_0,v_0} \neq 0$ for $t\in[t_1,t_2)$ in general. For $t\geq t_2$, if we have very simple geometry, for example the first picture of Figure 2, we can show the propagation of discontinuity, i.e. $[F]_{t,x_0 (t-t_0)v_0,v_0} \neq 0$ for $t\in[t_n,t_{n+1})$ even for $n=2,3$. But in general, for example the second picture of Figure 2, we cannot show $[F]_{t,x_0 (t-t_0)v_0,v_0} \neq 0$ for $t\in[t_n,t_{n+1})$ for $n\geq 2$.
\begin{figure}[h]
\begin{center}
\epsfig{file=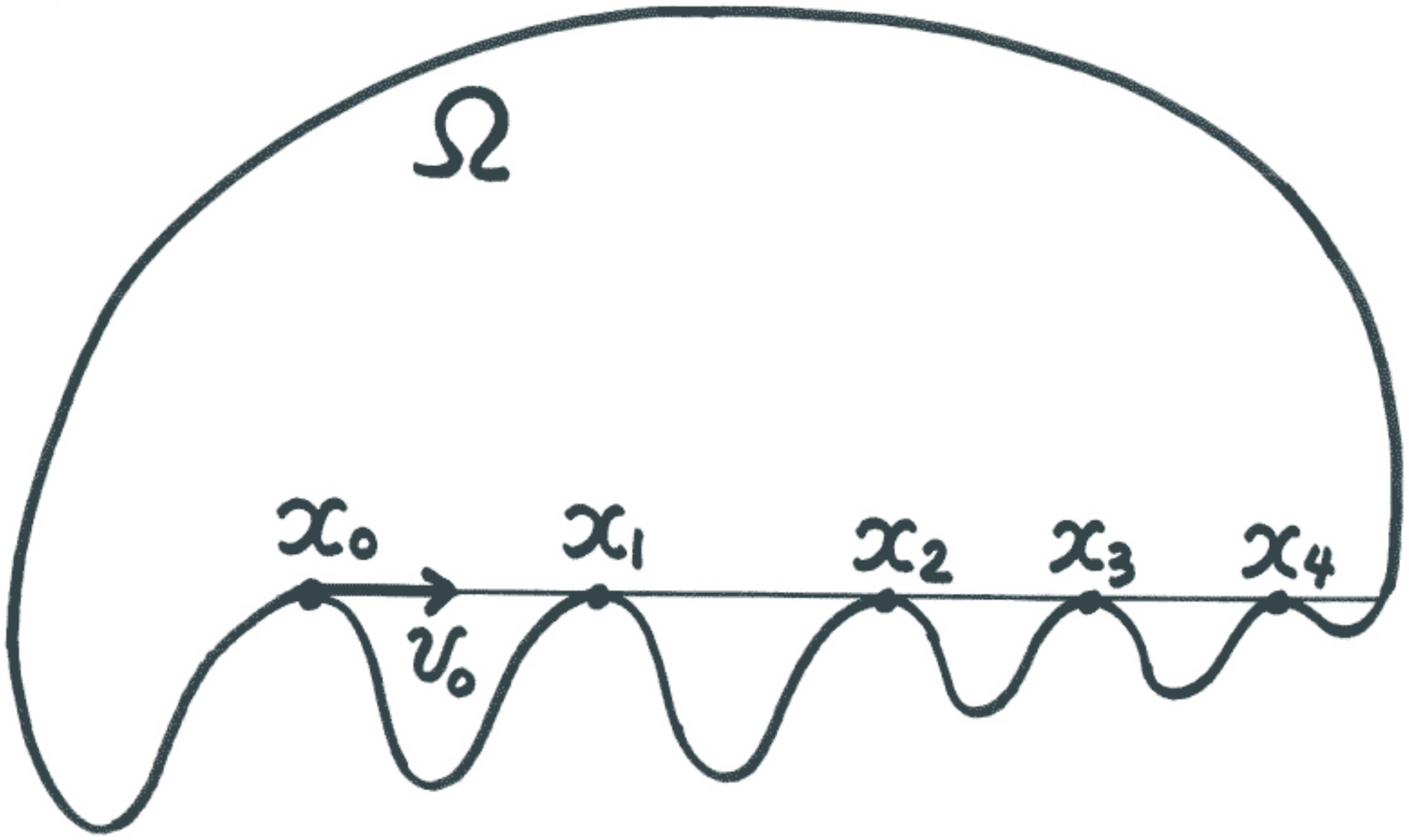,angle=90, height=3cm}\label{fig2} \ \ \ \ \ \ \ \ \ \ \
\epsfig{file=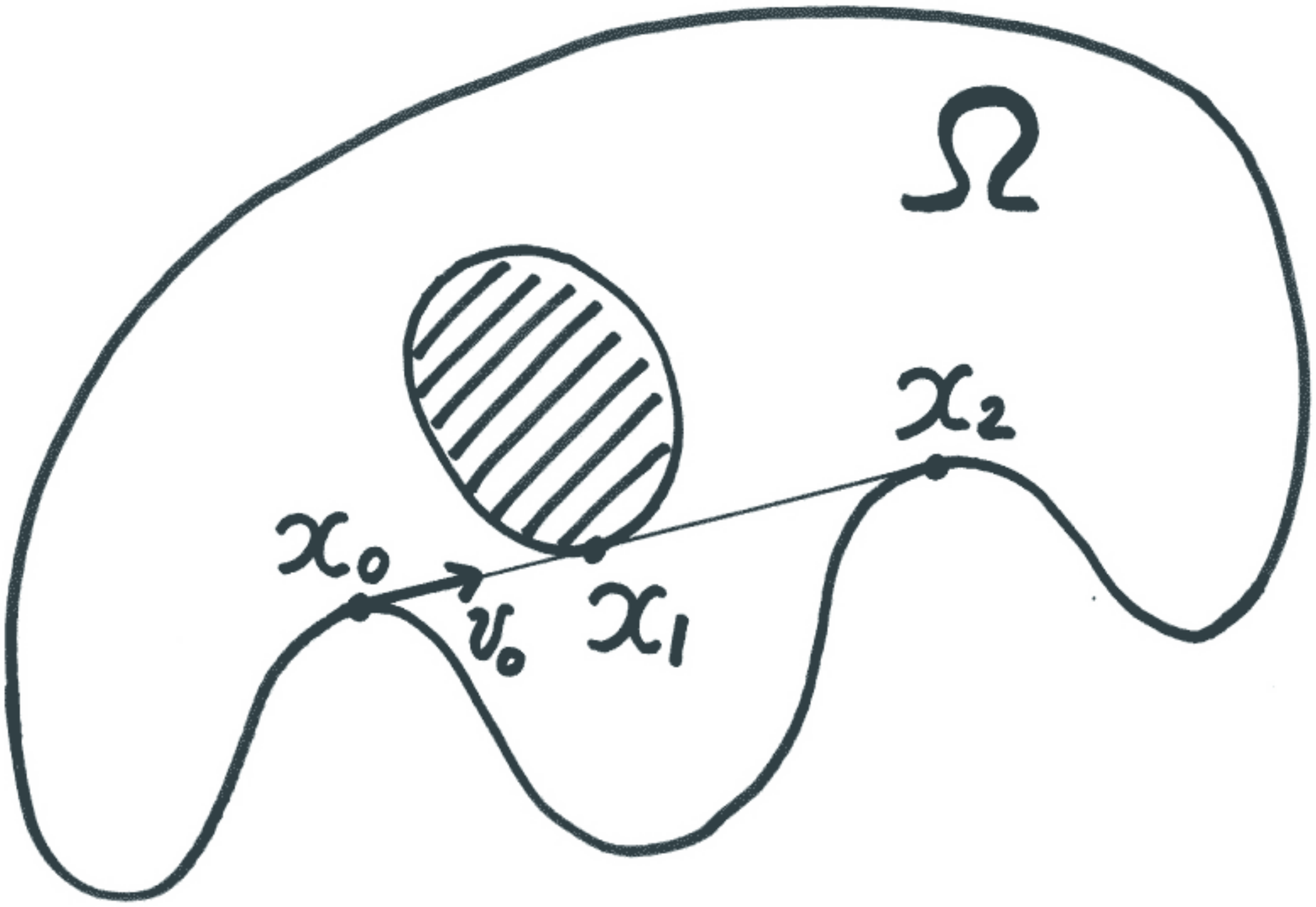,angle=90, height=3cm}\label{fig3}
\end{center}
\caption{Grazing Again}
\end{figure}
\\
The next result states that Theorem \ref{formationofsingularity} and Theorem \ref{propagation} capture all possible singularities (discontinuities), despite nonlinearity in the Boltzmann equation. In other words, the singularity of the Boltzmann solution is propagating as the linear Boltzmann equation and no new singularities created from the nonlinearity of Boltzmann equation.
\begin{theorem}[Continuity away from $\mathfrak{D}$]\label{continuityawayfromD}
Let $\Omega$ be an open bounded subset of $\mathbb{R}^3$ with a smooth boundary $\partial\Omega$. Let $F(t,x,v)$ be a Boltzmann solution of (\ref{boltzmann}) with the initial datum $F_0$ which is continuous on $\Omega\times\mathbb{R}^3 \cup \{\gamma_- \cup \gamma_+ \cup \gamma_0^{I-}\}$ and with one of
\newline 1. In-flow boundary condition (\ref{Finflow}). Assume (\ref{gsmallness}) is valid and the compatibility condition
\begin{equation}
F_0 (x,v) = G(0,x,v)
\ \ \text{for} \ \ (x,v)\in\gamma_- \cup \gamma_0^{I-},\label{CCinflow}
\end{equation}
and $G(t,x,v)$ is continuous on $[0,\infty)\times\{\gamma_-\cup\gamma_0^{I-}\}$.
\newline 2. Diffuse boundary condition (\ref{Fdiffuse}). Assume (\ref{smallness}) is valid and the compatibility condition
\begin{equation}
F_0 (x,v) = c_{\mu}\mu(v) \int_{v^{\prime}\cdot n(x)>0} F_0(x,v^{\prime})\{n(x)\cdot v^{\prime}\}dv^{\prime} \ \ \text{for} \ \ (x,v)\in\gamma_- \cup \gamma_0^{I-}. \label{CCdiffuse}
\end{equation}
\newline 3. Bounce-back boundary condition (\ref{Fbounceback}). Assume (\ref{smallness}) is valid and the compatibility condition
\begin{equation}
F_0(x,v)=F_0(x,-v) \ \ \text{for} \ \ (x,v)\in\gamma_- \cup \gamma_0^{I-}.\label{CCbounceback}
\end{equation}
Then $F(t,x,v)$ is a continuous function on $\mathfrak{C}$ for 1,2 and a continuous function on $\mathfrak{C}_{bb}$ for 3. If the domain $\Omega$ does not include a line segment (Definition \ref{linesegment}) then the continuity set $\mathfrak{C}$ and $\mathfrak{C}_{bb}$ are the complementary of $\mathfrak{D}$ and $\mathfrak{D}_{bb}$ respectively. Therefore $F(t,x,v)$ is continuous on $(\mathfrak{D})^c$ for 1,2 and continuous on $(\mathfrak{D}_{bb})^c$ for 3.
\end{theorem}
\begin{definition}\label{linesegment}
Assume $\Omega\in\mathbb{R}^3$ be open and the boundary $\partial\Omega$ be smooth. We say the boundary $\partial\Omega$ \textbf{does not include a line segment} if and only if
for each $x_0 \in\partial\Omega$ and for all $(u_1,u_2)\in\mathbb{S}^1$ there is no $\delta>0$ such that
\begin{equation*}
\Phi_{x_0}(\tau u_1,\tau u_2)
\end{equation*}
is a linear function for $\tau\in (-\delta,\delta)$ where $\Phi_{x_0}$ from (\ref{boundaryfunction}).
\end{definition}
\subsection{Previous Works and Significance of This Work}
There are many references for the mathematical study of different aspects of the boundary value problem of the Boltzmann equation, for example \cite{Guiraud}\cite{Guo08}\cite{Mischler} and the references therein. In \cite{Guo08}, an unified $L^2 - L^{\infty}$ theory in the near Maxwellian regime is developed to establish the existence, uniqueness and exponential decay toward a Maxwellian, for all four basic types of the boundary conditions and rather general domains.

The qualitative study of the particle-boundary interaction in a bounded domain and its effects on the global dynamics is a fundamental problem in the Boltzmann theory. One of challenging questions is the regularity theory of kinetic equations in bounded domain. This problem is hard because even for simplest kinetic equations with the differential operator $v\cdot \nabla_x$, the phase boundary $\partial\Omega\times\mathbb{R}^3$ is always characteristic but not uniformly characteristic at the grazing set $\gamma_0 = \{ (x,v) : x\in\partial\Omega, \ \text{and} \ v\cdot n(x) = 0 \}$. In a convex domain a continuity of the Boltzmann solution away from $\gamma_0$ is established in \cite{Guo08} for all four basic boundary conditions. In a convex domains, backward trajectories starting at interior points of the phase space cannot reach points of the grazing boundary $\gamma_0$, due to Velocity Lemma(\cite{Guo95}\cite{Hwang04}), where possible singularities may exist.

In general, on the other hand, in a non-convex domain, backward trajectories starting at interior points of the phase space can reach the grazing boundary. Therefore we expect singularities will be created at some part of grazing boundary $\gamma_0$ and propagate inside of the phase space. This question has been attracting a lot of attentions from early '90s, see references in pp.91--92 in Sone's book \cite{Sone}. For Boltzmann equation, most of works are numerical studies \cite{Sone}\cite{S-T}\cite{T-S-A} and few mathematical studies.

Once we enlarge our survey to propagation of singularities which already exist on initial data or boundary data, there are some mathematical works \cite{ABDG}\cite{B-D}\cite{B-D1}\cite{Cercignani}\cite{DLY} as well as numerical works \cite{ATAG}\cite{Sone}. In \cite{ABDG}, for linear BGK model, a propagation of discontinuity ,which exists already in the boundary data, is studied mathematically and also numerically. In \cite{B-D}, for the full Boltzmann equation in the near vacuum regime, a propagation of Sobolev $H^{1/25}$ singularity, which exists already in the initial data, is studied and same effect has been recently shown in the near Maxwellian regime \cite{B-D1}\cite{DLY}.

In Vlasov theory, we refer to \cite{ABL}\cite{G-M-P}\cite{voigt} for the boundary value problem. Singular solutions were studied in \cite{Guo95} extensively. In \cite{Guo95}, the non-convexity condition of boundary is replaced by the inward electric field which has a similar effect with non-convexity of the boundary. In convex domains, H\"{o}lder estimates of Vlasov solution with specular reflection boundary is solved recently \cite{Hwang04}\cite{H-V}, but Sovolev-type estimate is still widely open.
\\
\\
Our results give a rather complete characterization of formation and propagation of singularity for the nonlinear Boltzmann equation near Maxwellian in general domain under in-flow, diffuse, bounce-back boundary conditions. There is no restriction of the time interval. More precisely we show that for any non-convex point $x$ of the boundary and velocity tangent to $\partial\Omega$ at $x$, there exists an initial datum (and in-flow datum, for in-flow boundary condition case) such that the Boltzmann solution has a jump discontinuity at $(x,v)$. Once the discontinuity occurs at the grazing boundary, this discontinuity propagates inside along the forward trajectory until it hits the boundary again. And except those points, the grazing boundary and forward trajectories emanating from the grazing boundary, we can show that the Boltzmann solution is continuous.(Continuity away from $\mathfrak{D}$)
\subsection{Main Ingredients of the Proofs}
\textbf{1. The Equality induced by Non-Convex Domain} \\
We consider near Maxwellian regime and linearized Boltzmann equation (\ref{LinearBoltzmannEquation}). The formation of discontinuity is a consequence of following estimate. Assume $(x,v)\in\gamma_0^{\mathbf{S}}$ as below picture so that for sufficiently small $t>0$ the backward trajectory $x-tv$ is in an interior of the phase space. For simplicity we impose the trivial in-flow boundary condition $G(t,x,v)\equiv\mu(v)$ which corresponds $g(t,x,v)\equiv 0$ (\ref{inflow}). Consider points $(x^{\prime}_n,v^{\prime}_n)$ in $\gamma_-$ and $(x^{\prime\prime}_n,v^{\prime\prime}_n)$ missing the non-convex part near $(x,v)$ and both sequences converge $(x,v)$ as $n\rightarrow \infty$.
\begin{figure}[h]
\begin{center}
\epsfig{file=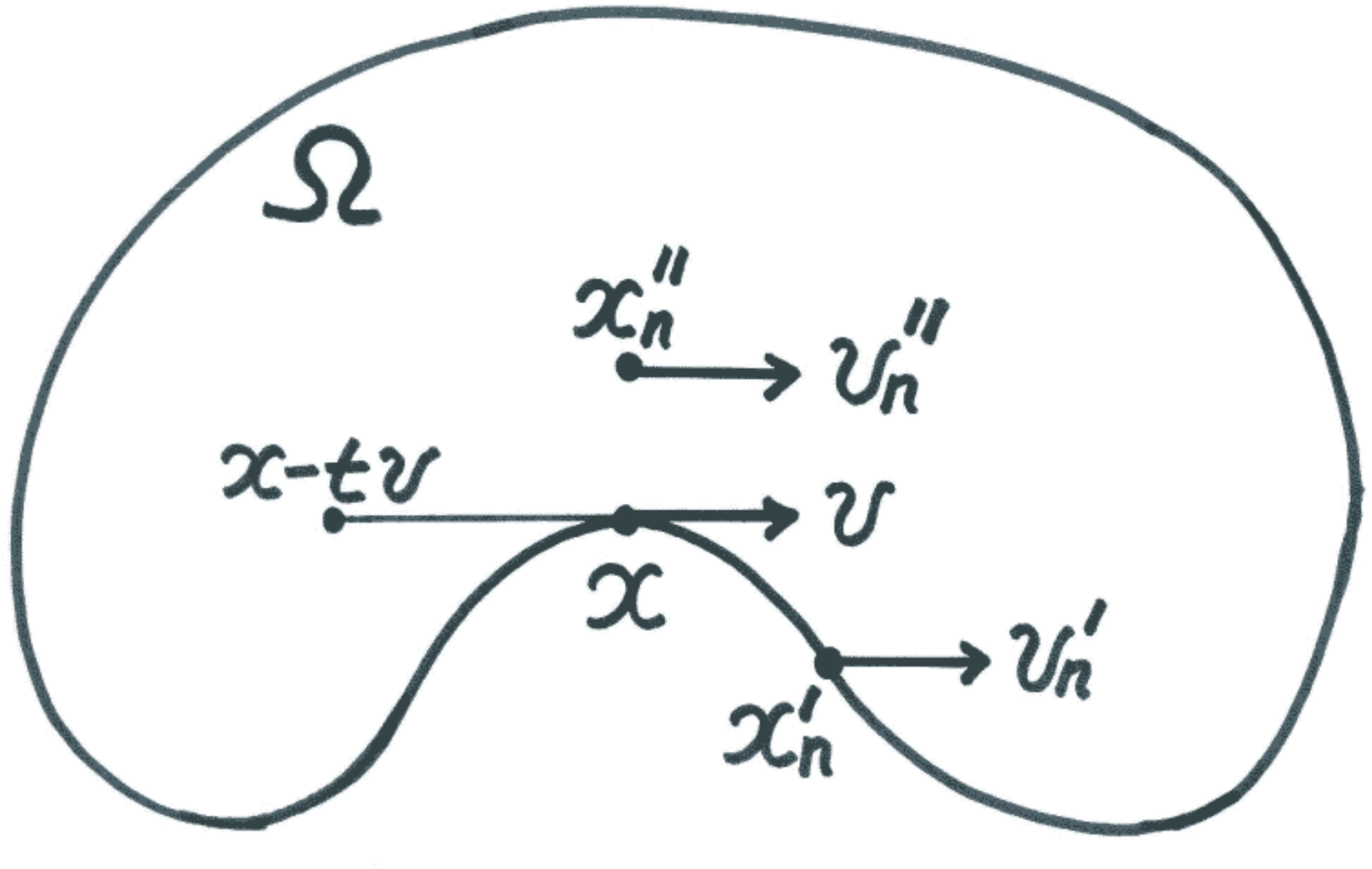,angle=90, height=3cm}
\end{center}
\label{fig4}
\caption{}
\end{figure}\\
Now suppose the solution $f$ of the linearized Boltzmann equation be continuous around $(x,v)$. Then the Boltzmann solution $f$ at $(x^{\prime}_n,v^{\prime}_n)$
\begin{equation*}
f(t,x^{\prime}_n,v^{\prime}_n) = g(t,x^{\prime}_n,v^{\prime}_n) =0,
\end{equation*}
and at $(x^{\prime\prime}_n,v^{\prime\prime}_n)$,
\begin{equation*}
f(t,x^{\prime\prime}_n,v^{\prime\prime}_n) = e^{-\nu(v^{\prime\prime}_n)t}f_0 (x^{\prime\prime}_n-tv^{\prime\prime}_n,v^{\prime\prime}_n) + \int_0^t e^{-\nu(v^{\prime\prime}_n)(t-s)}\{Kf+\Gamma(f,f)\}(s,x_n^{\prime\prime}-(t-s)v_n^{\prime\prime},v_n^{\prime\prime})ds
\end{equation*}
converges each other as $n\rightarrow\infty$. Then we have the following equality
\begin{equation}
f_0 (x-tv,v) = -\int_0^t e^{\nu(v)s}\{Kf + \Gamma(f,f)\}(s,x-(t-s)v,v)ds.
\label{inequalityinducedbynonconvexbody}
\end{equation}
Thanks to \cite{Guo08}, the pointwise estimate of $f$, with some standard estimates of $K,\Gamma$, the right hand side of above equality has magnitude $O(t)||f_0||_{\infty}(1+||f_0||_{\infty})$. If you choose $f_0(x-tv,v)=||f_0||_{\infty}$ then the above equality (\ref{inequalityinducedbynonconvexbody}) cannot be true for sufficiently small $t$ unless the trivial case $f_0\equiv 0(F\equiv \mu)$. Therefore the Boltzmann solution $f$ cannot be continuous at $(x,v)$. For diffuse (\ref{Fdiffuse}), bounce-back (\ref{Fbounceback}) boundary conditions we also obtain the equality induced by non-convex domain similar as (\ref{inequalityinducedbynonconvexbody}).
\\
\\
This argument bases on the idea that free transport effect is dominant to collision effect if time $t>0$ and the perturbation $\frac{F-\mu}{\sqrt{\mu}}$ is small.
\\
\newline\textbf{2. Continuity of the Gain Term $Q_+$}\\
The smoothing effect of the gain term $Q_+$ is one of the fundamental features of the Boltzmann theory. There are lots of results about the smoothing effect in Sobolev regularity, for example
\begin{equation*}
||Q_+(\phi,\psi)||_{H^{\frac{N-1}{2}}} \leq C ||\phi||_{L^1} ||\psi||_{L^2},
\end{equation*}
with some assumption on various collision kernels \cite{Lions}\cite{Wennberg94}\cite{Wennberg97}. To study the propagation of singularity and regularity, in the case of angular cutoff kernel (\ref{Gradcutoff}), it is standard to use Duhamel formulas and combine the Velocity average lemma and the regularity of $Q_+$ \cite{B-D}. For detail, see Villani's note \cite{V} especially pp. 77--79.
\\
\\
In order to study the propagation of discontinuity and continuity we need a totally different smooth effect of $Q_+$. For the discontinuity induced by the non-convex domain, we need following : Recall the grazing set $\mathfrak{G}$ in Definition \ref{grazingset}. A test function $\phi(t,x,v)$ is continuous on $[0,T]\times(\Omega\times\mathbb{R}^3) \backslash \mathfrak{G}$ and bounded on $[0,T]\times\Omega\times\mathbb{R}^3$. Then
\begin{equation}
Q_{+}(\phi,\phi)(t,x,v) \ \in \ C^0( [0,T]\times \Omega \times \mathbb{R}^3 ).\label{continuityQ+}
\end{equation}
Recall that the grazing set $\mathfrak{G}= \{(x,v)\in\bar{\Omega}\times\mathbb{R} : v\in \mathfrak{G}_x \}$. The grazing section $\mathfrak{G}_x = \{\tau u \in\mathbb{R}^3 : t\geq 0, \ u\in\mathfrak{G}_x \cap \mathbb{S}^2\}$ is a union of straight lines in velocity space $\mathbb{R}^3$ and two dimensional Lebesque measure of $\mathfrak{G}_x \cap \mathbb{S}^2$ is zero (Hongjie Dong's Lemma, Lemma 17 of \cite{Guo08}). Moreover, using continuous behavior of $\mathfrak{G}_x$ in $x$, one can invent a very effective covering of $\mathfrak{G}_x$ (Guo's covering, Lemma 18 of \cite{Guo08}). Because of those geometric and size restriction on $\mathfrak{G}$, even the gain term $Q_+$ is an integration operator in $v$ alone, we can prove the smoothing effect of $Q_+$ in $C^0( [0,T]\times \Omega \times \mathbb{R}^3 )$ for $t,x$ and $v$, see Theorem \ref{continuity}. Notice that those smoothing effect on $C^0_{t,x,v}$ has been believed to be true for long time without a mathematical proof in numerical communities \cite{Aoki}, p1587 of \cite{ABDG}, p502 of \cite{S-T}.
\\
\\
The main idea to prove the smoothing effect in $C^0_{t,x,v}$ is to use the Carleman's representation for $Q_+ (\phi,\phi)(t,x,v)$ which has been a very effective tool \cite{Guo03}\cite{Wennberg94}\cite{Wennberg97}.
\begin{equation}
\int_{\mathbb{R}^3} \phi(t,x,v^{\prime})\frac{1}{|v-v^{\prime}|^2} \int_{E_{vv^{\prime}}} \phi(t,x,v_1^{\prime}) B(2v-v^{\prime}-v_1^{\prime},\frac{v^{\prime}-v_1^{\prime}}{|v^{\prime}-v_1^{\prime}|}) dv_1^{\prime} dv^{\prime},\label{carleman1}
\end{equation}
with the hyperplane $E_{vv^{\prime}}=\{v_1^{\prime}\in\mathbb{R}^3 : (v_1^{\prime}-v)\cdot (v^{\prime}-v)=0 \}$. We will show the smallness of
\begin{equation}
|Q_+(\phi,\phi)(\bar{t},\bar{x},\bar{v})-Q_+(\phi,\phi)(t,x,v)|,\nonumber
\end{equation}
for $|(t,x,v)-(\bar{t},\bar{x},\bar{v})|< \delta$.
Assume we have sufficient decay of $\phi$ for large $v$. Replace the integrable kernel $\frac{1}{|v-v^{\prime}|^2}$ by smooth compactly supported function and cut off the singular part of $B(2v-v^{\prime}-v_1^{\prime},\frac{v^{\prime}-v_1^{\prime}}{|v^{\prime}-v_1^{\prime}|})$ to control the above difference as
\begin{eqnarray*}
&& \ O(\delta)||\phi||_{\infty}^2 \ + \ C\int_{|v^{\prime}|<N} |\phi(t,x,v^{\prime})-\phi(\bar{t},\bar{x},v^{\prime\prime})| \int_{E_{\bar{v}v^{\prime\prime}}\cap\{|v_1^{\prime\prime}|< N\}}|\phi(\bar{t},\bar{x},v_1^{\prime\prime})| dv_1^{\prime\prime}dv^{\prime} \\
&&+ \ C \int_{|v^{\prime}|<N} |\phi(t,x,v^{\prime})|
\left\{\int_{E_{vv^{\prime}} \cap \{|v_1^{\prime}|<N\}}\phi(t,x,v_1^{\prime}) dv_1^{\prime}
-\int_{E_{\bar{v}v^{\prime\prime}} \cap \{|v_1^{\prime\prime}|<N\}}\phi(\bar{t},\bar{x},v_1^{\prime\prime}) dv_1^{\prime\prime}
\right\}dv^{\prime},
\end{eqnarray*}
where $v^{\prime\prime}(v^{\prime})$ is chosen to be $v^{\prime}-(v-\bar{v})$ for convenience.

One can easily control the integration at the first line. Because for the first term, integrating over $v^{\prime}$, we can cut off a small neighborhood of $\mathfrak{G}_x$ from $|v^{\prime}|<N$. Away from that neighborhood, using the continuity of $\phi$ away from $\mathfrak{G}_x$ we can control the integrand pointwisely.

In order to control the second line integration we have to control the difference in big braces.
To do that we choose a special change of variables for $v_1^{\prime\prime}$ ,(\ref{changeofvariables2}). Under this change of variables the second line is bounded by
\begin{equation*}
C \int_{|v^{\prime}|<N} |\phi(t,x,v^{\prime})|
\int_{E_{vv^{\prime}} \cap \{|v_1^{\prime}|<N\}} | \ \phi(t,x,v_1^{\prime})-
\phi(\bar{t},\bar{x},v_1^{\prime\prime}) | \ dv_1^{\prime} dv^{\prime}.
\end{equation*}
The second integration term above is a function of $t,x,\bar{t},\bar{x}$ and $v$. Unfortunately one cannot expect a pointwise control (smallness) of the second integration for all $v$ : even the grazing section $\mathfrak{G}_x$, where $\phi(t,x,v_1^{\prime})$ might have discontinuity, is small, i.e. 2-dimensional Lebesque measure of $\mathfrak{G}_x \cap \mathbb{S}_x$ is zero, the measure on the plane $E_{vv^{\prime}}$ could be large(even infinite). However, in Section 3.3, we can show that that bad situation happens for very rare $v^{\prime}$ in $\{v^{\prime}\in\mathbb{R}^3 : |v^{\prime}| < N\}$ and use the integration over $v^{\prime}$ to control the above integration.\\
\newline \textbf{3. New Proof of Continuity of Boltzmann solution with Diffuse Boundary Condition}\\
In Section 5.2 we prove a continuity away from $\mathfrak{D}$ of Boltzmann solution with diffuse boundary condition using simple iteration scheme (\ref{hm}) with iteration diffuse boundary condition (\ref{iterationdiffuse}). This iteration scheme has several advantages. First it preserves a continuity away from $\mathfrak{D}$ as $m$ increasing, that is, if $h^m$ is continuous away from $\mathfrak{D}$ then $h^{m+1}$ is also continuous away from $\mathfrak{D}$. Second, the sequence $\{h^m\}$ has uniform $L^{\infty}$ bound and moreover it is Cauchy in $L^{\infty}$ for in-flow boundary condition $h^{m}|_{\gamma_-}= wg$. Therefore $h=\lim h^m$, a solution of the linear Boltzmann equation is continuous local in time. Combining with uniform-in-time boundedness of Boltzmann solution (\cite{Guo08}), we achieve the continuity for all time. In order to apply this idea to diffuse boundary condition, we use Guo's idea \cite{Guo08} : A norm of the diffuse boundary operator is less than 1 effectively, if we trace back several bounces. This approach gives simpler proof for the continuity of Boltzmann equation with diffuse boundary condition with convex domain (see Lemma $23~25$ of \cite{Guo08}).
\subsection{Structure of Paper}
In Section 2, we state some preliminary facts which are useful tools for this paper. In Section 3, we state and prove the continuity of $Q+$ (Theorem 4). In Section $4~6$, we deal with in-flow boundary, diffuse boundary and bounce-back boundary, respectively. For each section, first we prove the formation of discontinuity (Theorem 1). Then we show the continuity away from $\mathfrak{D}$ (Theorem 3). Using this continuity, combining with continuity of $Q_+$, we show the propagation of discontinuity (Theorem 2).

\section{Preliminary}
In this section we study continuity properties of the backward exit time $t_{\mathbf{b}}(x,v)$ and, a measure theoretic property and geometric covering of the grazing set $\mathfrak{G}$, and estimates of Boltzmann operators and the Carleman's representation.\\
\\
We use Lemma 1 of \cite{Guo08}, basic properties of the backward exist time $t_{\mathbf{b}}(x,v)$ :
\begin{lemma}
\label{huang} \cite{Guo08} Let $\Omega$ be an open bounded subset of $\mathbb{R}^3$ with a smooth boundary $\partial\Omega$. Let $(t,x,v)$ be connected with $(t-t_{\mathbf{b}}(x,v),x_{\mathbf{b%
}}(x,v),v)$ backward in time through a trajectory of (\ref{ode}).
\newline 1. The backward exit time $t_{\mathbf{b}}(x,v)$ is lower semicontinuous.
\newline 2 If
\begin{equation}
v\cdot n(x_{\mathbf{b}}(x,v))<0,  \label{negative}
\end{equation}%
then $(t_{\mathbf{b}}(x,v),x_{\mathbf{b}}(x,v))$ are smooth functions of $%
(x,v)$ so that%
\begin{eqnarray*}
\text{ \ \ \ }\nabla _{x}t_{\mathbf{b}} &=&\frac{n(x_{\mathbf{b}})}{v\cdot
n(x_{\mathbf{b}})},\text{ \ \ \ }\nabla _{v}t_{\mathbf{b}}=\frac{t_{\mathbf{b%
}}n(x_{\mathbf{b}})}{v\cdot n(x_{\mathbf{b}})},  \notag \\
\text{ \ \ \ }\nabla _{x}x_{\mathbf{b}} &=&I+\nabla _{x}t_{\mathbf{b}%
}\otimes v,\text{ \ \ }\nabla _{v}x_{\mathbf{b}}=t_{\mathbf{b}}I+\nabla
_{v}t_{\mathbf{b}}\otimes v.  \label{tbderivative}
\end{eqnarray*}%
\end{lemma}
For a convex domain, if a point $(x,v)$ is in the interior of the phase space, i.e. $(x,v)\in\Omega\times\mathbb{R}^3$ then the condition (\ref{negative}) is always satisfied and hence $t_{\mathbf{b}}(x,v)$ is smooth due to Lemma \ref{huang}. However for a non-convex domain, there is a point $(x,v)$ in $\Omega\times\mathbb{R}^3$ but $(x_{\mathbf{b}}(x,v,),v)\in\gamma_0$, i.e. $v\cdot n(x_{\mathbf{b}}(x,v))=0$. We further investigate a continuity property of $t_{\mathbf{b}}$ for that case. Indeed, discontinuity behavior of $t_{\mathbf{b}}(x,v)$ for $(x_{\mathbf{b}}(x,v),v)\in\gamma_0^{\mathbf{S}}$ is a main ingredient of the formation of discontinuity.
\begin{lemma}\label{tbcon}
Let $\Omega\in\mathbb{R}^3$ be an open set with a smooth boundary $\partial\Omega$. Assume $(x_0,v_0)\in\Omega\times\mathbb{R}^3$ with $v_0\neq0$ and $t_{\mathbf{b}}(x_0,v_0)<\infty$. Consider $(x_0,v_0)\in\mathfrak{G}$, i.e. $(x_{\mathbf{b}}(x_0,v_0),v_0)\in\gamma_0$.
\begin{eqnarray*}
&&\text{If }(x_{\mathbf{b}}(x_0,v_0),v_0)\in\gamma_0^{I-} \ \ \text{then} \ t_{\mathbf{b}}(x,v) \ \ \text{is continuous around } (x_0,v_0).\\
&&\text{If }(x_{\mathbf{b}}(x_0,v_0),v_0)\in\gamma_0^{\mathbf{S}} \ \ \text{then} \ t_{\mathbf{b}}(x,v) \ \ \text{is not continuous around } (x_0,v_0).
\end{eqnarray*}
Recall $\gamma_0^{I-}$ and $\gamma_0^{\mathbf{S}}$ in Definition \ref{grazingboundary}.
\end{lemma}
\begin{proof}
Throughout this proof, without loss of generality we assume that $\partial\Omega$ is a graph of $\Phi$ locally and $\Phi(0,0)=0$ and $(\partial_{x_1}\Phi, \partial_{x_2}\Phi)(0,0)=(0,0)$. Moreover assume $x_0 =(|x_0|,0,0), v_0 =(|v_0|,0,0)$ and $t_{\mathbf{b}}(x_0,v_0)=\frac{|x_0|}{|v_0|}$ so that $x_{\mathbf{b}}(x_0,v_0)=(0,0,0)=(0,0,\Phi(0,0))$.

First, let $(x_{\mathbf{b}}(x_0,v_0),v_0)\in\gamma_0^{I-}$. By the definition of $\gamma_0^{I-}$, we have $\Phi(-\tau,0)>0$ and $\Phi(\tau,0)<0$ for $0<\tau<<1$. Using the continuity of $\Phi$, choose sufficiently small $\varepsilon>0, \ \delta>0$ such that $\Phi(-\delta,y)>\frac{\varepsilon}{2}$ and $\Phi(\delta,y)<-\frac{\varepsilon}{2}$ for $0<|y|<\delta$. Fix $x=(x_1,x_2,x_3)\sim x_0$ and $v=(v_1,v_2,v_3)\sim v_0$. We define
\begin{equation}
\Psi(x,v,t) = x_3 -tv_3 -\Phi(x_1-tv_1, x_2-tv_2).\nonumber
\end{equation}
For $t^{\prime}\equiv \frac{x_1 -\delta}{v_1}$, $\Psi(x,v,t_0^{\prime})= -\Phi(\delta, x_2 -\frac{x_1 -\delta}{v_1} v_2 )+x_3 -\frac{x_1-\delta}{v_1} v_3 >\frac{\varepsilon}{4}$ for $(x_1,x_2,x_3)\sim (|x_0|,0,0), \ (v_1,v_2,v_3)\sim (|v_0|,0,0)$. For $t^{\prime\prime} = \frac{x_1 + \delta}{v_1}$, $\Psi(x,v,t^{\prime\prime})=-\Phi(-\delta,x_2 -\frac{x_1+\delta}{v_1}v_2)+ x_3 -\frac{x_1 +\delta}{v_1}v_3 < -\frac{\varepsilon}{4}$  for $(x_1,x_2,x_3)\sim (|x_0|,0,0), \ (v_1,v_2,v_3)\sim (|v_0|,0,0)$. Using the continuity of $\Phi$ and $\Psi$, there exists $t_* \in (\frac{x_1}{v_1}-\frac{\delta}{v_1},\frac{x_1}{v_1}+\frac{\delta}{v_1})$ so that $\Psi(x,v,t_*)=0$, i.e. $t_{\mathbf{b}}(x,v)=t_*$. If $x\sim x_0$ and $v\sim v_0$ then $\frac{x_1}{v_1}-\frac{\delta}{v_1}\sim \frac{|x_0|}{|v_0|}-\frac{\delta}{|v_0|} = t_{\mathbf{b}}(x_0,v_0)-\frac{\delta}{|v_0|}$ and $\frac{x_1}{v_1}+\frac{\delta}{v_1}\sim \frac{|x_0|}{|v_0|}+ \frac{\delta}{|v_0|} \sim t_{\mathbf{b}}(x_0,v_0)+\frac{\delta}{|v_0|}$ so that $t_* \in (t_{\mathbf{b}}(x_0,v_0)-\frac{\delta}{|v_0|},t_{\mathbf{b}}(x_0,v_0)+\frac{\delta}{|v_0|})$.

Next, let $(x_{\mathbf{b}}(x,v),v)\in\gamma_0^{\mathbf{S}}$. By the definition of the concave grazing boundary $\gamma_0^{\mathbf{S}}$, we have $\Phi(-\tau,0)>0$ and $\Phi(\tau,0)<0$ for $0<\tau<<1$. Choose a sequence $x_n = (|x_0|,0,\frac{1}{n})$. There exists $\varepsilon>0$ such that $t_{\mathbf{b}}(x_n,v_0)> t_{\mathbf{b}}(x_0,v_0)+\varepsilon$ for sufficiently large $n$. This implies that $(x_n,v_0)\rightarrow (x_0,v_0)$ but $t_{\mathbf{b}}(x_n,v_0)\nrightarrow t_{\mathbf{b}}(x_0,v_0)$ as $n\rightarrow \infty$.
\end{proof}
\\
\\
In the next two lemmas, we consider the grazing set $\mathfrak{G}$ (Definition \ref{grazingset}) including the discontinuity set $\mathfrak{D}$. The next lemma, Lemma 17 of \cite{Guo08} due to Hongjie Dong, is important to control a size of $\mathfrak{G}$. We denote $m_2$ as a standard 2-dimensional Lebesque measure and $m_3$ as a standard 3-dimensional Lebesque measure. Recall that the grazing section $\mathfrak{G}_x$ in Definition \ref{grazingset}.
\begin{lemma}
\label{hongjie} \cite{Guo08}  If $\partial\Omega$ is $C^1$ then the grazing section $\mathfrak{G}_x$ restricted to $\mathbb{S}^2$ has zero 2-dimensional Lebesque measure, i.e.
\begin{equation*}
m_2(\mathfrak{G}_x \cap \mathbb{S}^2)=0,
\end{equation*}
for all $x\in\bar{\Omega}$.
\end{lemma}
With condition $m_2(\mathfrak{G}_x \cap \mathbb{S}^2)=0$, we can construct the Guo's covering which is little bit stronger than the original one in Lemma 18 in \cite{Guo08}.
\begin{lemma}[Guo's covering]
\label{bouncebacklower}
 \cite{Guo08} Assume $m_2 (\mathfrak{G}_x \cap \mathbb{S}^2)=0$ is valid for all $x\in\bar{\Omega}$. Let $B_N =\{v\in\mathbb{R}^3 : |v|\leq N \}$. Then for any $\ \varepsilon >0$ and $N_*>0$
there exist $\delta _{\varepsilon ,N,N_*}>0,$ and $l_{\varepsilon,N,N_*,\Omega}$
balls $B(x_{1};r_{1}),B(x_{2};r_{2})...,B(x_{l};r_{l})\subset \bar{\Omega}$,
as well as open sets $O_{x_{1}},O_{x_{2},}...O_{x_{l}}$ of $B_{N}$ which are radial symmetric, i.e.
\begin{equation*}
O_{x_i} = \{ t\hat{v}\in\mathbb{R}^3 : t\geq 0, \ \hat{v}\in O_{x_i}\cap\mathbb{S}^2
\},
\end{equation*}
with $m_3(O_{x_{i}})<\frac{\varepsilon}{N_*} $ and $m_2(O_{x_{i}} \cap \mathbb{S}^2)\leq \frac{\varepsilon}{N^2 N_*}$ for all $1\leq i\leq l_{\varepsilon,N,N_*,\Omega},$ such that for any $x\in
\bar{\Omega},$ there exists $x_{i}$ so that $x\in B(x_{i};r_{i})\,\ $and for
$v\notin O_{x_{i}},$
\begin{equation*}
|v\cdot n(x_{\mathbf{b}}(x,v))|>\delta _{\varepsilon ,N, N_*}>0,
\end{equation*}
or equivalently
\begin{eqnarray*}
O_{x_i} \supset
\bigcup_{x\in B(x_i ;r_i)}\{ v\in B_{N} : |v\cdot n(x_{\mathbf{b}}(x,v))| \leq \delta_{\varepsilon ,N, N_*}
\}
\supset \bigcup_{x\in B(x_i ;r_i)} \mathfrak{G}_{x} \cap B_{N}.
\end{eqnarray*}
\end{lemma}
Combining Lemma \ref{hongjie} and Lemma \ref{bouncebacklower}, we have following lemma, which is useful to prove Theorem \ref{continuity}. Namely, a function which is continuous away from the grazing set $\mathfrak{G}$ is uniformly continuous except arbitrary small open set containing $\mathfrak{G}$.
\begin{lemma}\label{uniformlycontinuous}
Assume $\phi$ is continuous on $[0,T]\times (\Omega\times \{v\in\mathbb{R}^3 : \frac{1}{M}\leq |v|\leq N \})\backslash \mathfrak{G}$. For fixed $x\in \Omega$ and $\varepsilon>0$ and $N_*>0$, there exist
\begin{equation}
\delta=\delta(\phi,\Omega,\varepsilon,N_* ,x,\frac{1}{M},N)>0, \label{delta}
\end{equation}
and an open set $U_{x}\subset \{v\in\mathbb{R}^3 : \frac{1}{M}\leq |v|\leq N \}$ which is radial symmetric, i.e.
$U_x = \{ t\hat{v}\in\mathbb{R}^3 : t\geq 0 \ , \ \hat{v}\in U_x \cap \mathbb{S}^2 \}$
with $m_3(U_{x})<\frac{\varepsilon}{N_*}$ and $m_2(U_x \cap \mathbb{S}^2)< \frac{\varepsilon}{N_* N^2}$ such that
\begin{equation*}
|\phi(t,x,v)-\phi(\bar{t},\bar{x},\bar{v})| < \frac{\varepsilon}{N_*},
\end{equation*}
for $v\in \{v\in\mathbb{R}^3 : \frac{1}{M}\leq |v|\leq N \}\backslash U_{x}$ and $|(t,x,v)-(\bar{t},\bar{x},\bar{v})|< \delta.$
\end{lemma}
\begin{proof}
Let $x\sim\bar{x}$. Due to Guo's covering \cite{Guo08}, Lemma \ref{bouncebacklower}, we can choose $B(x_i ;r_i)$ including $x$ and $\bar{x}$, as well as $O_{x_i}\subset\mathbb{R}^3$ so that
\begin{equation*}
O_{x_i} \supset \bigcup_{y\in B(x_i;r_i)}\mathfrak{G}_y \cap B_{N} \supset \bigcup_{y\in B(x;\delta)}\mathfrak{G}_y \cap B_{N},
\end{equation*}
with $m_3(O_{x_i})<\frac{\varepsilon}{N_*}$. Notice that $m_3(\bar{O}_{x_i}) = m_3(O_{x_i})$. We can choose an open set $U_{x_i}$ so that $m_3(U_{x_i})\leq 2 m_3(O_{x_i})$ and $\bar{O}_{x_i} \subset U_{x_i}$. Since both of $\bar{O}_{x_i}$ and $B_{N}\backslash U_{x_i}$ are compact subsets of $B_N$, we have a positive distance between two sets, i.e.
\begin{equation*}
0< \mathfrak{d} =\inf\{ |\zeta-\xi| : \zeta \in \bar{O}_{x_i} \ \ \text{and} \ \ \xi\in B_{N}\backslash U_{x_i}
\}.
\end{equation*}
Assume $\delta < \mathfrak{d}/2$. Fix $x\in\bar{\Omega}$ and $v\in\{v\in\mathbb{R}^3 : \frac{1}{M}\leq |v|\leq N \}\backslash U_x$. Then $|(\bar{x},\bar{v})-(x,v)|<\delta$ implies that $\bar{v}\in \{v\in\mathbb{R}^3 : \frac{1}{M}\leq |v|\leq N \}\backslash O_{x_i}$. For such $x,v,\bar{x}$ and $\bar{v}$ consider the function $\phi$ as it's restriction on a compact set $[0,T]\times\bar{B}(x;\delta)\times B_N \backslash O_{x_i}$. Therefore $\phi_{[0,T]\times\bar{B}(x;\delta)\times B_N \backslash O_{x_i}}$ is uniformly continuous. Hence $|\phi(t,x,v)-\phi(\bar{t},\bar{x},\bar{v})|$ can be controlled small uniformly if $\delta>0$ is sufficiently small.
\end{proof}
\\
\\
We will use the Carleman's representation \cite{Guo03}\cite{Wennberg94} in the proof of Theorem \ref{continuity} crucially. Let $Q_+ (\phi,\psi)$ be defined by (\ref{qgl}) and let $\psi=\psi(v)$ and $\phi=\phi(v)$, $v\in\mathbb{R}^3$ make $Q_+ (\psi,\phi) <\infty$ almost everywhere. Then the \textbf{Carleman's representation} is
\begin{equation}
Q_+ (\psi,\phi)(v) = 2 \int_{\mathbb{R}^3}  \psi(v^{\prime}) \frac{1}{|v-v^{\prime}|^2} \int_{E_{vv^{\prime}}}  \phi(v_1^{\prime}) B(2v-v^{\prime}-v_1^{\prime},\frac{v^{\prime}-v_1^{\prime}}{|v^{\prime}-v_1^{\prime}|}) dv_1^{\prime} dv^{\prime},\label{carleman}
\end{equation}
where $E_{vv^{\prime}}$ is a hyperplane containing $v\in\mathbb{R}^3$ and perpendicular to $\frac{v^{\prime}-v}{|v^{\prime}-v|}\in\mathbb{S}^2$, i.e.
\begin{equation}
E_{vv^{\prime}}=\{v_1^{\prime}\in\mathbb{R}^3 : (v_1^{\prime} -v)\cdot (v^{\prime}-v)=0\}.\label{hyperplane}
\end{equation}
\\
In the proof of Theorem \ref{continuity} we need to control the integration over $E_{vv^{\prime}}$ in (\ref{carleman}) frequently :
\begin{lemma}\label{estimwhb}
For a rapidly decreasing function $\phi : \mathbb{R}_+ \rightarrow \mathbb{R}_+$, we have
\begin{equation}
\int_{E_{v v'}}\phi(|{{v_1}'}|) B\big(2v-{v'}-{{v_1}'},\frac{{v'}-{{v_1}'}}{|{v'}-{{v_1}'}|}\big)d{{v_1}'}
\leq C_{\phi} (1+|v-v'|^{\gamma}),\label{difference1}
\end{equation}
where $C_{\phi}$ only depends on $\phi$.
\end{lemma}
\begin{proof}
For fixed $v^{\prime}$ and $v$, let us denote $\{\tilde{\mathbf{e}}_1 , \tilde{\mathbf{e}}_2 , \tilde{\mathbf{e}}_3 \}$, with $\tilde{\mathbf{e}}_3 = \frac{v^{\prime}-v}{|v^{\prime}-v|}$, be the orthonormal basis
of $\mathbb{R}^3$ such that any $v_1^{\prime}\in E_{vv^{\prime}}$ can be written as $v_1^{\prime} = v + \eta_1
\tilde{\mathbf{e}}_1 +\eta_2 \tilde{\mathbf{e}}_2$. Since $v^{\prime}-v \perp E_{v v'}$ from (\ref{hyperplane}), there is $\eta_3$ such that
$v^{\prime}-v=\eta_3\tilde{\mathbf{e}}_3$ where $|\eta_3|=|v-v^{\prime}|$. Then we can write
$2v-v^{\prime}-v_1^{\prime}=v-v^{\prime} +v-v^{\prime}_1=-\eta_1 \tilde{\mathbf{e}}_1 -\eta_2 \tilde{\mathbf{e}}_2
-\eta_3 \tilde{\mathbf{e}}_3$ and $|2v-v^{\prime}-v^{\prime}_1|^2=\eta_1^2 +\eta_2^2 +|v^{\prime}-v|^2$. Moreover
$v^{\prime}-v^{\prime}_1 = -\eta_1 \tilde{\mathbf{e}}_1 -\eta_2 \tilde{\mathbf{e}}_2 +\eta_3 \tilde{\mathbf{e}}_3$. We can write the left hand side of (\ref{difference1}) as
\begin{eqnarray*}
&&\int_{-\infty}^{\infty}  \int_{-\infty}^{\infty}  \phi(\eta_1^2 + \eta_2^2 + |v|^2)
\Bigg{|}
\left( \begin{array}{ccc}
-\eta_1\\
-\eta_2\\
-\eta_3
\end{array} \right)
\Bigg{|}^{\gamma}
\times
\frac{1}{\eta_1^2+\eta_2^2+|{v}-{v^{\prime}}|^2}
\left( \begin{array}{ccc}
-\eta_1\\
-\eta_2\\
-\eta_3
\end{array} \right)
\cdot
\left( \begin{array}{ccc}
-\eta_1\\
-\eta_2\\\
\eta_3
\end{array} \right) d\eta_1 d\eta_2 \nonumber\\
&\leq&
\int_{-\infty}^{\infty}  \int_{-\infty}^{\infty}  \phi(\eta_1^2 + \eta_2^2)
(\eta_1^2+\eta_2^2+|v'-v|^2)^{\frac{\gamma}{2}-1} (\eta_1^{2}+\eta_2^{2}-|v'-v|^2)d\eta_1 d\eta_2 \nonumber\\
&\leq&
\int_{-\infty}^{\infty}  \int_{-\infty}^{\infty}  \phi(\eta_1^2 + \eta_2^2)\big(\eta_1^2+\eta_2^2+|v'-v|^2\big)^{\frac{\gamma}{2}}d\eta_1 d\eta_2 \nonumber\\
&\leq&
C_{\phi}(1+|v'-v|^{\gamma}) .\label{estimE}
\end{eqnarray*}
\end{proof}
In terms of the standard perturbation $f$ such that $F=\mu +\sqrt{\mu }f,$
the Boltzmann equation can be rewritten as
\begin{equation}
\left\{ \partial _{t}+v\cdot \nabla +L\right\} f=\Gamma (f,f),\text{ \ \ \ \
}f(0,x,v)=f_{0}(x,v), \label{LinearBoltzmannEquation}
\end{equation}%
where the standard linear Boltzmann operator, see \cite{Guo03}, is given by
\begin{equation*}
Lf\equiv \nu f-Kf,
\end{equation*}%
with the collision frequency $\nu (v)\equiv \int |v-u|^{\gamma }\mu
(u)q_{0}(\frac{v-u}{|v-u|}\cdot \omega)d\omega du $ for $0< \gamma
\leq 1$ and
\begin{eqnarray*}
&& \ \ \ \ \frac{1}{C_{\nu}} (1+|v|)^{\gamma}\leq \nu(v) \leq C_{\nu}(1+|v|)^{\gamma}, \label{collisionfrequency}\\
Kf &\equiv& \int_{\mathbb{R}^3}\mathbf{k}(v,v^{\prime})f(v^{\prime})dv^{\prime} \equiv \frac{1}{\sqrt{\mu}}Q_+(\mu , \sqrt{\mu}f) + \frac{1}{\sqrt{\mu}}Q_+(\sqrt{\mu}f,\mu) - \frac{1}{\sqrt{\mu}}Q_-(\sqrt{\mu}f,\mu)\label{K},\\
\Gamma (f,f)&\equiv&\frac{1}{\sqrt{\mu }}Q_+(\sqrt{\mu }f,\sqrt{\mu }%
f) - \frac{1}{\sqrt{\mu }}Q_-(\sqrt{\mu }f,\sqrt{\mu }%
f)\equiv \Gamma _{+}(f,f)-\Gamma_{-}%
(f,f).  \label{gamma}
\end{eqnarray*}%
We recall two estimates of operators $K$ and $\Gamma$ from \cite{Guo08}.
The Grad estimate for hard
potentials :%
\begin{equation*}
|\mathbf{k}(v,v^{\prime })|\leq C_{\mathbf{k}}\{|v-v^{\prime }|+|v-v^{\prime }|^{-1}\}e^{-%
\frac{1}{8}|v-v^{\prime }|^{2}-\frac{1}{8}\frac{||v|^{2}-|v^{\prime
}|^{2}|^{2}}{|v-v^{\prime }|^{2}}}.  \label{grad}
\end{equation*}%
Recall $w$ in (\ref{weight})$.$ Let $0\leq \theta<\frac{1}{4}$. Then there exists $0\leq \varepsilon (\theta
)<1$ and $C_{\theta }>0$ such that for $0\leq \varepsilon <\varepsilon
(\theta ),$
\begin{equation}
\int \{|v-v^{\prime }|+|v-v^{\prime }|^{-1}\}e^{-\frac{1-\varepsilon }{8}%
|v-v^{\prime }|^{2}-\frac{1-\varepsilon }{8}\frac{||v|^{2}-|v^{\prime
}|^{2}|^{2}}{|v-v^{\prime }|^{2}}}\frac{w(v)e^{\theta |v|^2}}{w(v^{\prime }) e^{\theta|v^{\prime}|^2}}dv^{\prime
}\leq \frac{C_{\mathbf{k}}}{1+|v|}.  \label{wk}
\end{equation}
For the nonlinear collision operator
\begin{equation}
|w\Gamma(g_1,g_2)(v)|\leq C_{\Gamma}(1+|v|)^{\gamma}||wg_1||_{\infty}||wg_2||_{\infty}.\label{Gamma1}
\end{equation}
Also we recall a standard estimate
\begin{equation}
\int_{\mathbb{R}^3} \phi(v^{\prime})|v-v^{\prime}|^{\gamma} dv^{\prime} \sim (1+|v|)^{\gamma},\label{standardbound}
\end{equation}
for $\phi\in L^1(\mathbb{R}^3)$.
\section{Continuity of the Collision Operators}

In this section we mainly prove the following Theorem :\\
\begin{theorem}[Continuity of $Q_+$]\label{continuity}
Assume $F(t,x,v)$ is continuous on $[0,T]\times(\Omega\times\mathbb{R}^3) \backslash \mathfrak{G}$ and
\begin{equation*}
||\bar{w}^{-1}F ||_{L^{\infty}([0,T]\times\bar{\Omega}\times\mathbf{R}^3)} < +\infty,
\end{equation*}
where $\bar{w}= \frac{e^{-\frac{|v|^2}{4}}}{(1+\rho^2 |v|^2)^{\beta}}\label{barw}$ with $\rho\in\mathbb{R}$ and $\beta>0$.
Then $Q_{+}(F,F)(t,x,v)$ is continuous in  $ \ [0,T]\times \Omega \times \mathbb{R}^3$ and
\begin{equation}
\sup_{[0,T]\times \bar{\Omega} \times\mathbb{R}^3} |\nu^{-1} \bar{w}^{-1} Q_{+}(F,F)(t,x,v)|<\infty.\label{bdQ}
\end{equation}
\end{theorem}
Theorem \ref{continuity}, a smooth effect in $C^0_{t,x,v}$, is the crucial ingredient to prove Theorem \ref{propagation} and Theorem \ref{continuityawayfromD}. This smooth effect of the gain term ensures that there is no singularity created by the nonlinearity of Botlzmann equation.
\\
\\
\begin{proof}[Proof of (\ref{bdQ})]
It is easy to show the boundedness (\ref{bdQ}) from
\begin{eqnarray*}
\nu^{-1} \bar{w}^{-1} Q_{+}(F,F)(t,x,v)&\leq&\frac{1}{\nu(v)\bar{w}(v)} \int_{\mathbb{R}^3}\int_{\mathbb{S}^2} B(v-u,\omega) \bar{w}(u^{\prime})\bar{w}(v^{\prime}) d\omega du \times || \bar{w}^{-1} F||_{\infty}^2\\
&\leq&
\nu(v)^{-1} \int_{\mathbb{R}^3} \int_{\mathbb{S}^2} B(v-u,\omega) \frac{e^{-\frac{|u|^2}{4}}}{(1+\rho^2 |u|^2)^{\beta}} d\omega du \times ||\bar{w}^{-1}F||_{\infty}^2\\
&\leq& C \ \nu(v)^{-1} \nu(v) ||\bar{w}^{-1}F||_{\infty}^2 \leq C||\bar{w}^{-1}F||_{L^{\infty}([0,T]\times(\bar{\Omega}\times\mathbf{R}^3))}^2
\end{eqnarray*}
where we used (\ref{standardbound}) and $|u^{\prime}|^2 + |v^{\prime}|^2 = |u|^2 + |v|^2$.
\end{proof}
\\
\\
Next we will show the continuity part of Theorem \ref{continuity}. The goal of following three subsections is to show
\begin{eqnarray}
&&\text{For fixed  \ } \varepsilon>0  \ \text{and } (t,x,v)\in [0,T]\times \Omega\times\mathbb{R}^3, \ \text{there is \ } \delta>0 \text{ \ such that}\nonumber\\
&&|Q_{+}(\bar{w}h,\bar{w}h)(\bar{t},\bar{x},\bar{v})- Q_{+}(\bar{w}h,\bar{w}h)(t,x,v)| < \varepsilon \ \ \text{for \ } |(\bar{t},\bar{x},\bar{v})-(t,x,v)|< \delta. \label{continuitystatement}
\end{eqnarray}
\subsection{Decomposition and Change of Variables}
In this section, we use the Carleman's representation to split $Q_{+}(\bar{w}h,\bar{w}h)(\bar{t},\bar{x},\bar{v})- Q_{+}(\bar{w}h,\bar{w}h)(t,x,v)$ in a natural way (\ref{diff}), and introduce two change of variables (\ref{changeofvariables1}) and (\ref{changeofvariables2}).\\
\\
It is convenient to define
\begin{equation*}
h\equiv \bar{w}^{-1}F,
\end{equation*}
where $||h||_{\infty} \equiv ||h||_{L^{\infty}([0,T]\times (\bar{\Omega}\times\mathbb{R}^3))} = ||\bar{w}^{-1} F ||_{L^{\infty}([0,T]\times(\bar{\Omega}\times\mathbf{R}^3))}.$ Choose $(\bar{t},\bar{x},\bar{v})\sim(t,x,v)$. Using the Carleman's Representation (\ref{carleman}) we have
\begin{equation*}
Q_{+}(\bar{w}h,\bar{w}h)(\bar{t},\bar{x},\bar{v})-
Q_{+}(\bar{w}h,\bar{w}h)(t,x,v) \ \ \ \ \ \ \ \ \ \  \ \ \ \ \ \ \ \ \ \ \ \ \ \ \ \ \ \ \ \ \ \ \ \ \  \ \ \ \ \ \ \ \ \ \ \ \ \ \ \ \ \ \ \ \ \ \ \ \ \    \ \ \ \ \ \ \ \ \ \ \ \ \ \  \ \ \ \ \ \ \ \ \ \ \ \ \ \ \
\end{equation*}
\begin{eqnarray}
&=&2 \int_{\mathbb{R}^3}\underbrace{\bar{w}(v'')h(\bar{t},\bar{x},v'')
\frac{1}{|\bar{v}-v''|^2}}_{\mathcal{A}}
{\int_{E_{\bar{v} v''}}\underbrace{\bar{w}({v_1^{\prime\prime}})h(\bar{t},\bar{x},{v_1^{\prime\prime}})
B\big(2\bar{v}-{v''}-{v_1^{\prime\prime}},\frac{{v''}-{v_1^{\prime\prime}}}{|{v''}-{v_1^{\prime\prime}}|}\big)}_{\mathcal{B}}  d{v_1^{\prime\prime}}}   d{v''}\nonumber\\
&&-2\int_{\mathbb{R}^3}\underbrace{\bar{w}(v')h(t,x,v')
\frac{1}{|v-v'|^2}}_{\mathcal{A}^{\prime}}
{\int_{E_{v v'}}\underbrace{\bar{w}(v_1^{\prime})h(t,x,v_1^{\prime})
B\big(2v-v'-v_1^{\prime},\frac{v'-v_1^{\prime}}{|v'-v_1^{\prime}|}\big)}_{\mathcal{B}^{\prime}} dv_1^{\prime}}  dv' \nonumber\\
&=&2\int_{\mathbb{R}^3} \ \{\mathcal{A} - \mathcal{A}^{\prime}\} \ \int_{E_{\bar{v}v^{\prime\prime}}}  \mathcal{B} \ dv_1^{\prime\prime} dv^{\prime\prime} \
 + \ 2 \int_{\mathbb{R}^3} \ \mathcal{A}^{\prime} \ \int_{E_{\bar{v}v^{\prime\prime}}} \{\mathcal{B}-\mathcal{B}^{\prime} \}  \ dv_1^{\prime\prime} dv^{\prime}.
\label{diff}
\end{eqnarray}
In order to control the first term of (\ref{diff}), we need to compare arguments of $\mathcal{A}$ and $\mathcal{A}^{\prime}$. For that purpose, we introduce the following change of variables :
\begin{lemma}
For fixed $v$ and $\bar{v}$ in $\mathbb{R}^3$ we define
\begin{equation}
v^{\prime\prime} \equiv v^{\prime\prime}(v^{\prime};v,\bar{v}) = v^{\prime}-(v-\bar{v}).\label{changeofvariables1}
\end{equation}
Then two planes $E_{\bar{v} {v''}} \text{  and  }E_{vv'}$ have same normal direction. The distance between to planes is $|(\bar{v}-v)\cdot\frac{v'-v}{|v'-v|}|$.
\end{lemma}
\begin{proof}
Assume (\ref{changeofvariables1}). Clearly $\frac{\partial v^{\prime\prime}(v^{\prime})}{\partial v^{\prime}} = I $ where $I$ is $3\times3$ identity matrix. The normal direction of $E_{\bar{v} v^{\prime\prime}}$ is $\frac{v^{\prime\prime}-\bar{v}}{|v^{\prime\prime}-\bar{v}|}=\frac{v^{\prime}-v}{|v^{\prime}-v|}$ which is also the normal direction of $E_{vv'}$. To measure a distance between two planes $E_{vv'}$ and $E_{\bar{v} v^{\prime\prime}}$, we consider the line passing $v$ and directing $\frac{v^{\prime}-v}{|v^{\prime}-v|}$, which is $v(s)=\frac{v^{\prime}-v}{|v^{\prime}-v|} s + v$. The solution of $v(s_*)\in E_{\bar{v}v^{\prime\prime}}$ is a solution of $0=(v^{\prime\prime}-\bar{v})\cdot (v(s)-\bar{v}) = (v^{\prime}-v)\cdot (v(s)-\bar{v})= |v'-v|s+(v'-v)\cdot(v-\bar{v})$. Easily we have $s_*=\frac{(v'-v)\cdot (\bar{v}-v)}{|v' -v|}$. Since $v(s)$ is unit-speed line we know that $|v(s_*)-v(0)|$ is the distance between $E_{\bar{v} {v''}} \text{  and  }E_{vv'}$.
\end{proof}\\
\\
An important property of (\ref{changeofvariables1}) is that two planes $E_{\bar{v} {v''}} \text{  and  }E_{vv'}$ have the same normal direction. In order to control the second term of (\ref{diff}), we need to compare arguments of $\mathcal{B}$ and $\mathcal{B}^{\prime}$ especially $v_1^{\prime}\in E_{vv^{\prime}}$ and $v_1^{\prime\prime}\in E_{\bar{v}v^{\prime\prime}}$. For that purpose, we introduce the following change of variables :
\begin{lemma}
For fixed $v,v'$ and $\bar{v}$ in $\mathbb{R}^3$, we define a unit Jacobian change of variables
\begin{equation}
v_1^{\prime\prime} \equiv v_1^{\prime\prime}(v_1^{\prime};v,v^{\prime},\bar{v})
= v_1^{\prime}+\frac{v'-v}{|v'-v|}\{(\bar{v}-v)\cdot \frac{v'-v}{|v'-v|}\}.\label{changeofvariables2}
\end{equation}
In this change of variables $v_1^{\prime\prime} \in E_{\bar{v}v^{\prime\prime}}$ if and only if $ \ v_1^{\prime}\in E_{vv^{\prime}}$.
\end{lemma}
\begin{proof}
Assume (\ref{changeofvariables1}) and (\ref{changeofvariables2}). Clearly $\frac{\partial v_1^{\prime\prime}(v_1^{\prime})}{\partial v_1^{\prime}}= I$. We can check following equality :
\begin{eqnarray*}
&&({v_1^{\prime\prime}}-\bar{v})\cdot({v''}-\bar{v})=({v_1}'-\bar{v}+\frac{v'-v}{|v'-v|}\{(\bar{v}-v)\cdot\frac{v'-v}{|v'-v|}\})\cdot(v'-v)\\
&=& ({v_1}'-\bar{v})\cdot(v'-v)+|v'-v|\{(\bar{v}-v)\cdot\frac{v'-v}{|v'-v|}\}= (v_1^{\prime} -v)\cdot (v^{\prime}-v) +(v-\bar{v})\cdot (v^{\prime}-v) +
(\bar{v}-v)\cdot (v^{\prime}-v)\\
&=&(v_1^{\prime} -v)\cdot (v^{\prime}-v).
\end{eqnarray*}
By definition, $v_1^{\prime}\in E_{vv^{\prime}}$ is equivalent to $(v_1^{\prime}-v) \cdot \frac{v^{\prime}-v}{|v^{\prime}-v|}=0$. Then, from the above equality, we conclude $({v_1^{\prime\prime}}-\bar{v})\cdot\frac{{v''}-\bar{v}}{|{v''}-\bar{v}|}=0$ which is equivalent to $v_1^{\prime\prime} \in E_{\bar{v}v^{\prime\prime}}$.
\end{proof}
\\
Under the first change of variables (\ref{changeofvariables1}), we can rewrite the first term of $(\ref{diff})$ as
\begin{eqnarray}
2\int_{\mathbb{R}^3}
\underbrace{\frac{1}{|v-v'|^2}
\Big\{
\bar{w}(v'')h(\bar{t},\bar{x},v'')-\bar{w}(v')h(t,x,v')
\Big\}}_{\mathcal{(C)}} \int_{E_{\bar{v} v^{\prime\prime}}}\underbrace{\bar{w}({v_1^{\prime\prime}})h(\bar{t},\bar{x},{v_1^{\prime\prime}})
B\big(2\bar{v}-{v''}-{v_1^{\prime\prime}},\frac{{v''}-{v_1^{\prime\prime}}}{|{v''}-{v_1^{\prime\prime}}|}\big)}_{\mathcal{(D)}} d{v_1^{\prime\prime}}  dv'. \label{estimategainterm}
\end{eqnarray}
Under the second change of variables (\ref{changeofvariables2}), we can rewrite the second term of $(\ref{diff})$ as
\begin{eqnarray}
2\int_{\mathbb{R}^3}\underbrace{\bar{w}(v')h(t,x,v')
\frac{1}{|v-v'|^2}}_{\mathcal{(E)}} \ \ \ \ \ \ \ \ \ \ \ \ \ \ \ \ \ \ \ \ \ \ \ \ \ \ \ \ \ \ \ \ \ \ \ \ \ \ \  \ \ \ \ \ \ \ \ \ \ \ \ \ \ \ \ \ \ \ \ \ \ \ \ \ \ \ \ \ \ \ \ \ \ \ \ \ \ \ \ \ \ \ \ \ \ \ \ \ \ \ \ \ \ \ \ \ \ \ \ \ \nonumber\\
\times\int_{E_{v v'}}
\underbrace{\Big\{\bar{w}({v_1^{\prime\prime}})h(\bar{t},\bar{x},{v_1^{\prime\prime}})
B\big(2\bar{v}-{v''}-{v_1^{\prime\prime}},\frac{{v''}-{v_1^{\prime\prime}}}{|{v''}-{v_1^{\prime\prime}}|}\big) -
\bar{w}(v_1^{\prime})h(t,x,v_1^{\prime})
B\big(2v-v'-v_1^{\prime},\frac{v'-v_1^{\prime}}{|v'-v_1^{\prime}|}\big)
\Big\}}_{\mathcal{(F)}}d{v_1^{\prime}} dv'.\label{estimategainterm2}
\end{eqnarray}
We will estimate (\ref{estimategainterm}) and (\ref{estimategainterm2}) separately in following two sections.

\subsection{Estimate of (\ref{estimategainterm})}
We divide into several cases :
\newline $\mathbf{Case \ 1 : }$
$|v|\geq N$. From Lemma \ref{estimwhb}, for $N>0$ we can estimate
\begin{eqnarray*}
Q_+ (\bar{w}h,\bar{w}h)(t,x,v)\mathbf{1}_{|v|>N}&\leq& C ||h||_{\infty}^2 \mathbf{1}_{|v| > N}\int_{\mathbf{R}^3} \bar{w}(v^{\prime}) \left( \frac{1}{|v-v'|^2}+\frac{1}{|v-v'|^{2-\gamma}}
\right)dv^{\prime}\nonumber\\
&\leq& C||h||_{\infty}^2  \left( \frac{1}{(1+|v|)^2} +\frac{1}{(1+|v|)^{2-\gamma}}
\right) \mathbf{1}_{|v|> N}
\leq \frac{C}{N}||h||_{\infty}^2.
\end{eqnarray*}
Hence we have
\begin{equation}
(\ref{estimategainterm})\mathbf{1}_{|v|\geq N} \leq \frac{C}{N}||h||_{\infty}^2.\label{sum1}
\end{equation}
\newline $\mathbf{Case \ 2 : }$
$|v|\leq N$ and $|v^{\prime}|\geq 2N$, or $|v|\leq N$ and $ |v^{\prime}|\leq\frac{1}{M}$. Also assume $0<\delta<<1$.
\begin{eqnarray}
&&2\times\mathbf{1}_{|v|\leq N} \int_{\{|v^{\prime}|\geq 2N \ \text{or} \ |v^{\prime}|\leq\frac{1}{M}\} }\mathcal{(C)} \ \int_{E_{\bar{v}v^{\prime\prime}}}\mathcal{(D)} \ dv_1^{\prime\prime} dv^{\prime}\nonumber\\
&\leq& C \mathbf{1}_{|v|\leq N}\int_{|v^{\prime}|\geq 2N} \left\{ \frac{1}{|v-v^{\prime}|^2} + \frac{1}{|v-v^{\prime}|^{2-\gamma}}
\right\} e^{-\frac{|v^{\prime}|^2}{8}} dv^{\prime} e^{\frac{\delta^2}{4}} || h||_{\infty}^2
+\underbrace{C\int_{|v^{\prime}|\leq \frac{1}{M}}\left\{\frac{1}{|v^{\prime}|^2}+ \frac{1}{|v^{\prime}|^{2-\gamma}}\right\}e^{-\frac{|v^{\prime}|^2}{8}}dv^{\prime}}_{o(\frac{1}{M})} e^{\frac{\delta^2}{4}} || h||_{\infty}^2
\nonumber\\
&\leq& C \left(\frac{1}{N^2} + \frac{1}{N^{2-\gamma}}\right) || h||_{\infty}^2
+o(\frac{1}{M}) ||h||_{\infty}^2,
\label{sum2}
\end{eqnarray}
where we have used $\bar{w}(v^{\prime})\leq e^{-\frac{|v^{\prime}|^2}{4}}$ and $\bar{w}(v^{\prime\prime}) \leq e^{-\frac{|v^{\prime}|^2}{8}} e^{\frac{\delta^2}{4}}$ and Lemma \ref{estimwhb}.
\newline $\mathbf{Case \ 3 : }$ $|v|\leq N$ and $\frac{1}{M}\leq|v^{\prime}|\leq 2N$.
\begin{eqnarray}
&&2\times\mathbf{1}_{|v|\leq N} \int_{\frac{1}{M}\leq|v^{\prime}|\leq 2N} \ \mathcal{(C)} \
\int_{E_{\bar{v}v^{\prime\prime}}} \ \mathcal{(D)} \ dv_1^{\prime\prime} dv^{\prime}\nonumber\\
&\leq&
C||h||_{\infty} \int_{\frac{1}{M}\leq|v^{\prime}|\leq 2N} \mathbf{1}_{|v|\leq N} \left(\frac{1}{|v-v^{\prime}|^2} + \frac{1}{|v-v^{\prime}|^{2-\gamma}}\right) \ |\bar{w}(v^{\prime\prime}) h(\bar{t},\bar{x},v^{\prime\prime}) -\bar{w}(v^{\prime})h(t,x,v^{\prime})| dv^{\prime}.\label{a1}
\end{eqnarray}
Since $\left(\frac{1}{|v-v^{\prime}|^2} +\frac{1}{|v-v^{\prime}|^{2-\gamma}}\right)$ is integrable we can choose a smooth function $\mathbf{z}(v,v^{\prime})$ with compact support such that
\begin{equation}
\sup_{|v|\leq N} \int_{|v^{\prime}|\leq 2N} \left| \left(\frac{1}{|v-v^{\prime}|^2} +\frac{1}{|v-v^{\prime}|^{2-\gamma}}\right) - \mathbf{z}(v,v^{\prime})
\right| dv^{\prime} \leq \frac{1}{N}.\label{estimate2}
\end{equation}
Therefore we can bound (\ref{a1}) by two parts
\begin{eqnarray}
&&C ||h||_{L^{\infty}}^2  \int_{|v^{\prime}|\leq 2N} \mathbf{1}_{|v|\leq N} \left| \left(\frac{1}{|v-v^{\prime}|^2} +\frac{1}{|v-v^{\prime}|^{2-\gamma}}\right) - \mathbf{z}(v,v^{\prime})
\right| e^{-\frac{|v^{\prime}|^2}{8}} e^{\frac{\delta^2}{4}} dv^{\prime}
\label{e111}\\
&+& C \sup_{|v|\leq N , |v^{\prime}|\leq 2N} |\mathbf{z}(v,v^{\prime})|\times
||h||_{L^{\infty}} \int_{\frac{1}{M}\leq|v^{\prime}|\leq 2N} \mathbf{1}_{|v|\leq N} | \bar{w}(v^{\prime\prime}(v^{\prime}))h(\bar{t},\bar{x},v^{\prime\prime}(v^{\prime}))
-\bar{w}(v^{\prime})h(t,x,v^{\prime})| dv^{\prime}. \label{estimate112}
\end{eqnarray}
From (\ref{estimate2}), it is easy to control the first term
\begin{equation}
|(\ref{e111})| \leq \frac{C}{N}||h||_{\infty}^2.\label{sum3}
\end{equation}
Now we are going to estimate the second term (\ref{estimate112}). Applying Lemma \ref{uniformlycontinuous} to $\bar{w}(v^{\prime})h(t,x,v^{\prime})$, we can choose $\delta=\delta(\bar{w}h,\Omega,\varepsilon,N_* ,x,\frac{1}{M},2N)>0$ and an open set $U_{x}\subset \{\frac{1}{M}\leq |v|\leq 2N\}$ with $|U_{x}|<\frac{\varepsilon}{N_*} $ such that
\begin{equation*}
|\bar{w}(v^{\prime\prime}(v^{\prime}))h(\bar{t},\bar{x},v^{\prime\prime}(v^{\prime}))
-\bar{w}(v^{\prime})h(t,x,v^{\prime})| < \frac{\varepsilon}{N_*},
\end{equation*}
for $v^{\prime}\in \{v\in\mathbb{R}^3 : \frac{1}{M}\leq |v|\leq N \}\backslash U_{x}$ and $|(\bar{t},\bar{x},\bar{v})-(t,x,v)|<\delta$. Therefore we can split the second part (\ref{estimate112}) as integration over $U_x$ and $U_x^c$ and control as
\begin{eqnarray}
C\sup_{|v|\leq N, |v^{\prime}|\leq 2N}|\mathbf{z}(v,v^{\prime})|\times||h||_{\infty}^2 \times m_3(U_{x})\nonumber
+ C ||h||_{\infty} \int_{\{ \frac{1}{M}\leq|v^{\prime}|\leq 2N\} \cap U_{x}^c} |\bar{w}(v^{\prime\prime}(v^{\prime}))h(\bar{t},\bar{x},v^{\prime\prime}(v^{\prime}))
-\bar{w}(v^{\prime})h(t,x,v^{\prime})| dv^{\prime}\\
\leq  C\sup_{|v|\leq N, |v^{\prime}|\leq 2N}|\mathbf{z}(v,v^{\prime})|\times||h||_{\infty}^2 \frac{\varepsilon}{N_*} \
+ \ C||h||_{\infty} N^3 \frac{\varepsilon}{N_*}. \ \ \ \ \ \ \ \ \ \ \ \ \ \ \ \ \ \ \ \ \ \ \ \ \ \ \ \ \ \ \ \ \ \ \ \ \ \ \ \ \ \ \ \ \ \ \ \ \ \ \ \ \ \ \ \ \ \ \ \ \ \ \ \ \ \ \ \ \ \
\label{sum4}
\end{eqnarray}
\newline In summary, combinig $(\ref{sum1}), (\ref{sum2}), (\ref{sum3})$ and $(\ref{sum4})$, we have established
\begin{equation*}
(\ref{estimategainterm}) \leq C||h||_{\infty}^2 \left\{ \frac{1}{N} + o(\frac{1}{M}) +  \sup_{|v|\leq N, |v^{\prime}|\leq 2N}|\mathbf{z}(v,v^{\prime})| \frac{\varepsilon}{N_*}
\right\}
+C||h||_{\infty}N^3 \frac{\varepsilon}{N_*}.
\label{tosum1}
\end{equation*}
Choosing sufficiently large $N,M>0$ and $N_*>0$ then
\begin{equation}
(\ref{estimategainterm})\leq\frac{\varepsilon}{2}.\label{finalest1}
\end{equation}
\subsection{Estimate of (\ref{estimategainterm2})}
The estimate of (\ref{estimategainterm2}) is much more delicate. The reason is that we cannot expect $\int_{E_{vv^{\prime}}} \ (\mathcal{F}) \ dv_1^{\prime}$ in (\ref{estimategainterm2}) is small for all $v^{\prime}\in\mathbb{R}^3$. We know that $h(t,x,v_1^{\prime})$ may not be continuous on $v_1^{\prime}\in \mathfrak{G}_x$. Even $\mathfrak{G}_x$ is radial symmetric and has a small measure by Lemma \ref{hongjie}, a bad situation, the intersection of $\mathfrak{G}_x$ and $E_{vv^{\prime}}$ could have large (even infinite) 2-dimensional Lebesque measure, can happen. However we can show that such bad situations only happen for very rare $v^{\prime}$'s in $\mathbb{R}^3$. Using the integration over $v^{\prime}\in\mathbb{R}^3$, we are able to control (\ref{estimategainterm2}) small.\\
\\
Recall $(\mathcal{E})$ and $(\mathcal{F})$ in (\ref{estimategainterm2}). We divide into several cases :
\newline\textbf{Case 1 : }$|v|\geq N$. Follow exactly same proof of \textbf{Case 1} of the previous subsection, we conclude
\begin{equation}
(\ref{estimategainterm2})\mathbf{1}_{|v|\geq N} \leq \frac{C}{N}||h||_{\infty}^2.\label{sum21}
\end{equation}
\newline\textbf{Case 2 : }$|v|\leq N$ and $|v^{\prime}|\geq 2N$. We go back to original formula, the second term of (\ref{diff}), and use Lemma \ref{estimwhb} to estimate
\begin{eqnarray}
2\int_{|v^{\prime}|\geq 2N} (\mathcal{E}) \int_{E_{vv^{\prime}}} (\mathcal{F}) dv_1^{\prime} dv^{\prime}\mathbf{1}_{|v|\leq N}\leq 4 ||h||_{\infty}^2 \int_{|v^{\prime}|\geq 2N} \bar{w}(v^{\prime})\frac{1}{|v-v^{\prime}|^2} (1+|v-v^{\prime}|)^{\gamma}dv^{\prime} \mathbf{1}_{|v|\leq N}
\leq 4||h||_{\infty}^2  \left(\frac{1}{N^2}+ \frac{1}{N^{2-\gamma}}\right).\label{sum22}
\end{eqnarray}
\newline\textbf{Case 3 : }$|v|\leq N$, $|v^{\prime}|\leq 2N$, and $|v_1^{\prime}|\leq \frac{1}{N}$ or $|v_1^{\prime}|\geq N$. In the case of $|v_1^{\prime}|\leq \frac{1}{N}$, we have
\begin{eqnarray}
&&2\times\mathbf{1}_{|v|\leq N} \int_{|v^{\prime}|\geq 2N} \mathcal{(E)} \int_{\{|v_1^{\prime}|\leq \frac{1}{N}\}\cap E_{vv^{\prime}}} \mathcal{(F)} \ dv_1^{\prime} dv^{\prime}\nonumber\\
&&\leq2 ||h||_{\infty}^2 \int_{\mathbb{R}^3} \frac{\bar{w}(v^{\prime})}{|v-v^{\prime}|^2} dv^{\prime} \int_{\{|v_1^{\prime}|\leq \frac{1}{N}\}\cap E_{vv^{\prime}}} \left\{e^{-\frac{|v_1^{\prime}|^2}{8}}e^{\frac{\delta^2}{4}}(4N+\frac{1}{N}+\delta)^{\gamma}
+e^{-\frac{|v_1^{\prime}|^2}{4}} (4N+\frac{1}{N})^{\gamma}
\right\} dv_1^{\prime}
\leq C \frac{||h||_{\infty}^2 }{N^{2-\gamma}}.\label{sum23}
\end{eqnarray}
In the case of $|v_1^{\prime}|\geq N$ we have
\begin{eqnarray}
&&2\times\mathbf{1}_{|v|\leq N} \int_{|v^{\prime}|\geq 2N} (\mathcal{E}) \int_{\{|v_1^{\prime}|\leq \frac{1}{N}\}\cap E_{vv^{\prime}}} (\mathcal{F}) \ dv_1^{\prime} dv^{\prime}\nonumber\\
&&\leq2 ||h||_{\infty}^2 \int_{\mathbb{R}^3} \frac{\bar{w}(v^{\prime})}{|v-v^{\prime}|^2} dv^{\prime} \int_{\{|v_1^{\prime}|\geq N\}\cap E_{vv^{\prime}}} \left\{e^{-\frac{|v_1^{\prime}|^2}{8}}e^{\frac{\delta^2}{4}}(4N+\frac{1}{N}+\delta)^{\gamma}
+e^{-\frac{|v_1^{\prime}|^2}{4}} (4N+\frac{1}{N})^{\gamma}
\right\} dv_1^{\prime}\nonumber\\
&&\leq C ||h||_{\infty}^2 e^{-\frac{N^2}{16}} \int_{\mathbb{R}^3} e^{-\frac{|v_1^{\prime}|^2}{16}} dv^{\prime}\times N^{\gamma}e^{-\frac{N^2}{16}}\leq C||h||_{\infty}^2 e^{-\frac{N^2}{16}}.\label{sum24}
\end{eqnarray}
\newline\textbf{Case 4 : }$|v|\leq N, |v^{\prime}|\leq 2N,$ and $\frac{1}{N}\leq |v_1^{\prime}| \leq N$.
In order to remove the unboundedness of $\frac{1}{|v-v^{\prime}|^{2}}$ in (\ref{estimategainterm2}), we choose a positive smooth function $\mathbf{Z}(v,v^{\prime})$ with compact support such that
\begin{eqnarray}
\sup_{|v|\leq N} \int_{|v^{\prime}|\leq 2N} \left| \frac{1}{|v-v^{\prime}|^2} -\mathbf{Z}(v,v^{\prime})
\right|dv^{\prime} < \frac{1}{N^{10}}.\label{z1}
\end{eqnarray}
Splitting $2\times\mathbf{1}_{|v|\leq N} \int_{|v^{\prime}|\leq 2N} (\mathcal{E}) \int_{\frac{1}{N}\leq |v_1^{\prime}| \leq N} (\mathcal{F}) \ dv_1^{\prime} dv^{\prime}$ into two parts
\begin{eqnarray}
&&2\times\mathbf{1}_{|v|\leq N} \int_{|v^{\prime}|\leq 2N}\bar{w}(v^{\prime})|h(t,x,v^{\prime})|\left|\frac{1}{|v-v^{\prime}|^2}-\mathbf{Z}(v,v^{\prime})\right|\int_{E_{vv^{\prime}}\cap \{\frac{1}{N}\leq |v_1^{\prime}|\leq N\}} (\mathcal{F}) \ dv_1 dv^{\prime}\leq C||h||_{\infty}^2\frac{1}{N^{10}}N^{\gamma+2},\label{sum25}\\
&&C \int_{|v^{\prime}|\leq 2N} ||h||_{\infty} \sup_{|v|\leq N, |v^{\prime}|\leq 2N}|\mathbf{Z}(v,v^{\prime})|   \int_{E_{vv^{\prime}}\cap\{\frac{1}{N}\leq |v_1^{\prime}|\leq N\}} (\mathcal{F}) \ dv_1^{\prime} dv^{\prime},\label{Z1}
\end{eqnarray}
where we used (\ref{z1}) for the first line. From now we will focus to estimate (\ref{Z1}).
\newline\textbf{Case 5 : } $|v|\leq N, |v^{\prime}|\leq 2N, \frac{1}{N}\leq |v_1^{\prime}|\leq N$ and $|2v-v^{\prime}-v_1^{\prime}|<\frac{1}{M}$ or $|v^{\prime}-v_1^{\prime}|<\frac{1}{M}$. This region included the part where the collision kernel $B(\cdot,\cdot)$ has a singular behavior.
\begin{eqnarray}
&&C \int_{|v^{\prime}|\leq 2N} ||h||_{\infty} \sup_{|v|\leq N, |v^{\prime}|\leq 2N}|\mathbf{Z}(v,v^{\prime})|   \int_{E_{vv^{\prime}}\cap\{\frac{1}{N}\leq |v_1^{\prime}|\leq N\}} (\mathcal{F}) \
\mathbf{1}_{\{|(2v-v^{\prime})-v_1^{\prime}|<\frac{1}{M} \ \text{or} \ |v^{\prime} -v^{\prime}_1| < \frac{1}{M} \}}(v^{\prime},v_1^{\prime}) \
dv_1^{\prime} dv^{\prime}\nonumber\\
&\leq&
C\sup_{|v|\leq N , |v^{\prime}|\leq 2N} |\mathbf{Z}(v,v^{\prime})| \times ||h||_{\infty}^2 \times \int_{|v^{\prime}|\leq 2N} dv^{\prime} e^{-\frac{|v^{\prime}|^2}{4}}\nonumber
 \int_{E_{vv^{\prime}}} dv_1^{\prime} \Big\{\mathbf{1}_{\{|(2v-v^{\prime})-v_1^{\prime}|<\frac{1}{M}\}}(v_1^{\prime})+
\mathbf{1}_{\{|v^{\prime}-v_1^{\prime}|< \frac{1}{M}\}}(v_1^{\prime})
\Big\}\times N^{\gamma}\nonumber\\
&\leq& C\sup_{|v|\leq N , |v^{\prime}|\leq 2N} |\mathbf{Z}(v,v^{\prime})| \times ||h||_{\infty}^2 \frac{N^{\gamma}}{M^2}.\label{sum26}
\end{eqnarray}
\newline\textbf{Case 6 : } $|v|\leq N , |v^{\prime}|\leq 2N , \frac{1}{N}\leq|v_1^{\prime}|\leq N$ and $|2v-v^{\prime}-v_1^{\prime}|>\frac{1}{M}$ and $|v^{\prime}-v_1^{\prime}|>\frac{1}{M}$ and $0<\delta < \frac{1}{10M}$. We estimate
\begin{eqnarray}
2\times\mathbf{1}_{|v|\leq N} \int_{|v^{\prime}|\leq 2N} dv^{\prime} \bar{w}(v^{\prime}) h(t,x,v^{\prime}) \mathbf{Z}(v,v^{\prime}) \int_{E_{vv^{\prime}} \cap \{\frac{1}{N}\leq |v_1^{\prime}|\leq N\}}
\{ \bar{w}(v_1^{\prime\prime})h(\bar{t},\bar{x},v_1^{\prime\prime})B(2\bar{v}-v^{\prime\prime}-v_1^{\prime\prime}, \frac{v^{\prime\prime}-v_1^{\prime\prime}}{|v^{\prime\prime}-v_1^{\prime\prime}|})\nonumber\\
-\bar{w}(v_1^{\prime})h(t,x,v_1^{\prime})B(2v-v^{\prime}-v_1^{\prime}, \frac{v^{\prime}-v_1^{\prime}}{|v^{\prime}-v_1^{\prime}|})
\}
\mathbf{1}_{\{ |2v-v^{\prime}-v_1^{\prime}|> \frac{1}{M}
\}} \mathbf{1}_{\{ |v^{\prime}-v_1^{\prime}|> \frac{1}{M}
\}}dv_1^{\prime}.\label{case5}
\end{eqnarray}
We need this step because of the singular behavior of
\begin{eqnarray*}
B(u_1 ,u_2) =|u_1|^{\gamma} q_0 (\frac{u_1}{|u_1|}\cdot \frac{u_2}{|u_2|}) = |u_1|^{\gamma} (q_0 \circ \mathfrak{F}) (u_1 ,u_2),
\end{eqnarray*}
where $\mathfrak{F} : \mathbb{R}^3 \times \mathbb{R}^3 \rightarrow \mathbb{R}$ with $\mathfrak{F}(u_1 ,u_2)= \frac{u_1}{|u_1|} \cdot \frac{u_2}{|u_2|}$. The function $\mathfrak{F}(u_1 ,u_2)$ is not continuous at $(u_1 ,u_2)=(0,0)$ and continuous away from $(0,0)$, i.e. the restriction of $\mathfrak{F}$ on a compact set,
\begin{equation*}
\mathfrak{F}_{M,N} : \{\frac{1}{2M} \leq |u_1| \leq 6N \} \times \{\frac{1}{2M} \leq |u_2| \leq 4N \} \rightarrow \mathbb{R}
\end{equation*}
is uniformly continuous.
From $|2v-v^{\prime}-v_1^{\prime}|> \frac{1}{M}$ and $|v-\bar{v}|< \delta < \frac{1}{10M}$ we have lower bound of
\begin{equation*}
|2\bar{v}-v^{\prime\prime}-v_1^{\prime\prime}|\geq \left| |2v-v^{\prime}-v_1^{\prime}|-|\bar{v}-v -\frac{v^{\prime}-v}{|v^{\prime}-v|}\{(\bar{v}-v)\cdot \frac{v^{\prime}-v}{|v^{\prime}-v|}\}| \right| \geq \frac{1}{2M}.
\end{equation*}
Similarly from $|v^{\prime}-v_1^{\prime}|>\frac{1}{M}$ and $|v-\bar{v}|<\delta <\frac{1}{10 M}$ we have a lower bound of
\begin{equation*}
|v^{\prime\prime}-v_1^{\prime\prime}|\geq \left| |v^{\prime}-v_1^{\prime}| - |\bar{v}-v -\frac{v^{\prime}-v}{|v^{\prime}-v|}\{(\bar{v}-v)\cdot \frac{v^{\prime}-v}{|v^{\prime}-v|}\}|
\right| \geq \frac{1}{2M}.
\end{equation*}
Therefore for any $\varepsilon >0$, we can choose $\delta>0$ so that
\begin{eqnarray}
&&\left| B(2\bar{v}-v^{\prime\prime}-v_1^{\prime\prime}, \frac{v^{\prime\prime}-v_1^{\prime\prime}}{|v^{\prime\prime}-v_1^{\prime\prime}|})
- B(2v-v^{\prime}-v_1^{\prime}, \frac{v^{\prime}-v_1^{\prime}}{|v^{\prime}-v_1^{\prime}|})
\right|\nonumber\\
&=& \Big| |2\bar{v}-v^{\prime\prime}-v_1^{\prime\prime}|^{\gamma} (q_0 \circ \mathfrak{F})(2\bar{v}-v^{\prime\prime}-v_1^{\prime\prime}, v^{\prime\prime}-v_1^{\prime\prime})
-|2v-v^{\prime}-v_1^{\prime}|^{\gamma} (q_0 \circ \mathfrak{F})(2v-v^{\prime}-v_1^{\prime}, v^{\prime}-v_1^{\prime})
\Big|
< \frac{\varepsilon}{N_*},\label{e1case5}
\end{eqnarray}
for $|2v-v^{\prime}-v_1^{\prime}|>\frac{1}{M}$ and $|v^{\prime}-v_1^{\prime}|>\frac{1}{M}$ and $0<\delta < \frac{1}{10M}$.
\\
\newline Now we split (\ref{case5}) by two parts
\begin{eqnarray}
&&2\times\mathbf{1}_{|v|\leq N} \int_{|v^{\prime}|\leq 2N} dv^{\prime} .. \int_{E_{vv^{\prime}}\cap \{\frac{1}{N}\leq |v_1^{\prime}|\leq N\}} \bar{w}(v_1^{\prime\prime})h(\bar{t},\bar{x},v_1^{\prime\prime})\left\{ B(2\bar{v}-v^{\prime\prime}-v_1^{\prime\prime}, \frac{v^{\prime\prime}-v_1^{\prime\prime}}{|v^{\prime\prime}-v_1^{\prime\prime}|})-
B(2v-v^{\prime}-v_1^{\prime}, \frac{v^{\prime}-v_1^{\prime}}{|v^{\prime}-v_1^{\prime}|})
\right\}\nonumber\\
&& \ \ \ \ \ \ \ \ \ \ \ \ \ \ \ \ \ \ \ \ \ \ \ \ \ \ \ \ \ \ \ \ \ \ \ \ \ \ \ \ \ \ \ \ \ \ \ \ \ \ \ \ \ \ \ \ \ \ \ \ \ \ \ \ \ \ \ \ \ \ \ \ \ \ \ \ \ \ \ \ \ \ \ \ \ \ \ \  \ \ \ \ \ \ \  \ \ \ \ \ \ \ \ \ \ \times\mathbf{1}_{\{ |2v-v^{\prime}-v_1^{\prime}|> \frac{1}{M}
\}} \mathbf{1}_{\{ |v^{\prime}-v_1^{\prime}|> \frac{1}{M}
\}} dv_1^{\prime}
\nonumber\\
&+& 2\times\mathbf{1}_{|v|\leq N} \int_{|v^{\prime}|\leq 2N} dv^{\prime} .. \int_{E_{vv^{\prime}}\cap \{\frac{1}{N}\leq |v_1^{\prime}|\leq N\}} \Big\{\bar{w}(v_1^{\prime\prime})h(\bar{t},\bar{x},v_1^{\prime\prime})-\bar{w}(v_1^{\prime})h(t,x,v_1^{\prime})\Big\}
B(2v-v^{\prime}-v_1^{\prime}, \frac{v^{\prime}-v_1^{\prime}}{|v^{\prime}-v_1^{\prime}|}).\label{ab2}
\end{eqnarray}
Using (\ref{e1case5}), the continuity of $B(\cdot,\cdot)$ away from $(0,0)$, the first line above is bounded by
\begin{equation}
C\sup_{v,v^{\prime}}|\mathbf{Z}(v,v^{\prime})| \times || h||_{\infty}^2 \frac{\varepsilon}{N_*}.\label{sum27}
\end{equation}
\\
In the remainder of this section we will focus on (\ref{ab2}) :\\
\newline\textbf{Estimate of (\ref{ab2})}
\\
\begin{equation}
(\ref{ab2})\leq CN^2 ||h||_{\infty} \sup_{v,v^{\prime}}|\mathbf{Z}(v,v^{\prime})|
\int_{|v^{\prime}|\leq 2N} \bar{w}(v^{\prime}) \underbrace{\int_{E_{vv^{\prime}}\cap \{\frac{1}{N}\leq |v_1^{\prime}|\leq N\}} {|\bar{w}(v_1^{\prime\prime})h(\bar{t},\bar{x},v_1^{\prime\prime})-\bar{w}(v_1^{\prime})h(t,x,v_1^{\prime})|} dv_1^{\prime}}_{\blacklozenge}dv^{\prime},\label{abc}
\end{equation}
where we used $\sup_{|v|\leq N , |v^{\prime}|\leq 2N , |v_1^{\prime}|\leq N} B(2v-v^{\prime}-v_1^{\prime},\frac{v^{\prime}-v_1^{\prime}}{|v^{\prime}-v_1^{\prime}|}) < \infty$.
Recall our choice of $v''$ and $v_1^{\prime\prime}$ from (\ref{changeofvariables1}) and (\ref{changeofvariables2}) to have
\begin{equation*}
|{v_1^{\prime\prime}}-v_1^{\prime}|\leq\Big{|}\frac{v'-v}{|v'-v|}\{(\bar{v}-v)\cdot\frac{v'-v}{|v'-v|}\}\Big{|}\leq |\bar{v}-v|<\delta. \label{smallness1}
\end{equation*}
We will use the followin strategy : separate $\int_{E_{vv^{\prime}}\cap\{\frac{1}{N}\leq |v_1^{\prime}|\leq N\}}... dv_1^{\prime}$ into two parts
\begin{equation}
\int_{U_x \cap E_{vv^{\prime}}\cap\{\frac{1}{N}\leq |v_1^{\prime}|\leq N\}}... dv_1^{\prime} + \int_{U_x^c \cap E_{vv^{\prime}}\cap\{\frac{1}{N}\leq |v_1^{\prime}|\leq N\}}... dv_1^{\prime}.\nonumber
\end{equation}
The first part is the integration over $U_x$, a neighborhood of $\mathfrak{G}_x$ that contains possible discontinuity of $h$. Moreover we expect the measure of the neighborhood $U_x$ is small so we can control the first term. For the second term, we will use the continuity of the integrand $\bar{w}h$. However if $v=0$ then $\mathfrak{G}_x$ could be a large measure set in $E_{vv^{\prime}}\cap \{\frac{1}{N}\leq |v_1^{\prime}|\leq N\}$. For example if $\mathfrak{G}_x \cap \mathbb{S}^2 = \{u\in\mathbb{S}^2 : u_3 =0 \}$ then $\mathfrak{G}_x$ is $xy-$plane and $E_{0\mathbf{e}_3}$ is also $xy-$plane. Therefore we have to divide two cases $v\neq 0$ and $v=0$ and study separately.
\\
\newline\textbf{Case of }$\mathbf{v\neq 0}$\\
In the case of $v\neq 0$, assume $\varrho < |v|^2 /2$ for sufficiently small $\varrho>0$. We will divide the velocity space $\mathbb{R}
^3$ into
\begin{equation*}
\mathfrak{B}=\left\{ v^{\prime}\in\mathbb{R}^3 : |v|-\frac{\varrho}{|v|} \leq v^{\prime}\cdot \frac{v}{|v|} \leq |v|+\frac{\varrho}{|v|}\right\} \ \ \text{and} \ \ \mathfrak{B}^c=\left\{ v^{\prime}\in \mathbb{R}^3 : \left|v^{\prime}\cdot\frac{v}{|v|}-|v|\right| > \frac{\varrho}{|v|}\right\}.\label{mathcalB}
\end{equation*}
The important property of $\mathfrak{B}$ is that if $v\in\mathfrak{B}^c$ then $E_{vv^{\prime}}$ does not contain zero.
We can split the integration part of (\ref{abc}), $\blacklozenge$ into
\begin{eqnarray}
\int_{v^{\prime} \in B_{2N} \cap \mathfrak{B}} \bar{w}(v^{\prime}) \int_{E_{vv^{\prime}}\cap\{\frac{1}{N}\leq |v_1^{\prime}| \leq N\}} |\bar{w}(v_1^{\prime\prime})h(\bar{t},\bar{x},v_1^{\prime\prime})-\bar{w}(v_1^{\prime})h(t,x,v_1^{\prime})| dv_1^{\prime} dv^{\prime} \label{B1}\\
+\int_{v^{\prime} \in B_{2N} \backslash\mathfrak{B}} \bar{w}(v^{\prime}) {\int_{E_{vv^{\prime}}\cap\{\frac{1}{N}\leq |v_1^{\prime}| \leq N\}} |\bar{w}(v_1^{\prime\prime})h(\bar{t},\bar{x},v_1^{\prime\prime})-\bar{w}(v_1^{\prime})h(t,x,v_1^{\prime})| dv_1^{\prime}} dv^{\prime}. \label{B2}
\end{eqnarray}
Notice that $\mathfrak{B}\cap B_{2N}$ has a small measure :
\begin{equation*}
m_3 (\mathfrak{B}\cap B_{2N}) \leq 2\pi (2N)^2 \times 2\frac{\varrho}{|v|} \leq 2\pi (2N)^2 \times 2\frac{\varrho}{\sqrt{2\varrho}} \leq 2\sqrt{2} \pi (2N)^2 \sqrt{\varrho}.
\end{equation*}
Therefore we have
\begin{equation}
|(\ref{B1})|\leq C N^4 ||h||_{L^{\infty}}\sqrt{\varrho}.\label{sum28}
\end{equation}
Now we are going to estimate (\ref{B2}). Here we use a property of $\mathfrak{B}^c$ : for $v^{\prime}\in \mathfrak{B}^c$ we have
\begin{equation*}
\text{dist}(0,E_{vv^{\prime}}) = \left| v\cdot \frac{v^{\prime}-v}{|v^{\prime}-v|} \right| = \frac{|v^{\prime}\cdot v -|v|^2|}{|v^{\prime}-v|} > \frac{\varrho}{|v^{\prime}-v|} > \frac{\varrho}{2N+|v|} \geq \frac{\varrho}{3N},
\end{equation*}
where we also have used $|v^{\prime}|\leq 2N$ and $|v|\leq N$.
From Lemma \ref{uniformlycontinuous} we use $U_x$, an open radial symmetric subset of $\{\frac{1}{N}\leq |v_1^{\prime}|\leq N\}$ with a small measure and $\bar{w}h$ is uniformly continuous on $U_x^c$, to split (\ref{B2}) into
\begin{eqnarray}
\int_{v^{\prime} \in B_{2N} \backslash\mathfrak{B}} \bar{w}(v^{\prime}) \int_{E_{vv^{\prime}}\cap\{\frac{1}{N}\leq |v_1^{\prime}| \leq N\} \cap U_x} |\bar{w}(v_1^{\prime\prime})h(\bar{t},\bar{x},v_1^{\prime\prime})-\bar{w}(v_1^{\prime})h(t,x,v_1^{\prime})| dv_1^{\prime} dv^{\prime}\label{U1}\\
+\int_{v^{\prime} \in B_{2N} \backslash\mathfrak{B}} \bar{w}(v^{\prime}) \int_{E_{vv^{\prime}}\cap\{\frac{1}{N}\leq |v_1^{\prime}| \leq N\} \cap U_x^c} |\bar{w}(v_1^{\prime\prime})h(\bar{t},\bar{x},v_1^{\prime\prime})-\bar{w}(v_1^{\prime})h(t,x,v_1^{\prime})| dv_1^{\prime} dv^{\prime}.\label{U2}
\end{eqnarray}
For the last line, we use Lemma \ref{uniformlycontinuous} to know estimate $|\bar{w}(v_1^{\prime\prime})h(\bar{t},\bar{x},v_1^{\prime\prime})-\bar{w}(v_1^{\prime})h(t,x,v_1^{\prime})|<\frac{\varepsilon}{N_*}$,
 for $v_1^{\prime}\in E_{vv^{\prime}}\cap \{\frac{1}{N}\leq |v_1^{\prime}|\leq N\}\backslash U_x$ and $|v_1^{\prime\prime}-v_1^{\prime}|\leq |v-\bar{v}|< \delta$. Therefore
\begin{equation}
|(\ref{U2})| \leq CN^2 \frac{\varepsilon}{N_*} ||h||_{\infty}.\label{sum30}
\end{equation}
In order to show that $(\ref{U1})$ is small, we introduce following projection :
\begin{lemma}\label{Projection}
Assume $0<\varrho < \frac{|v|^2}{2}$. Let $E_{vv^{\prime}}= \{ v_1^{\prime} \in \mathbb{R}^3 : (v_1-v)\cdot (v^{\prime}-v) =0 \}$ . We define a projection
\begin{eqnarray*}
\mathbb{P} \ : \ \mathbb{S}^2 &\rightarrow& E_{vv^{\prime}} \\
u \in \mathbb{S}^2 &\mapsto& \left\{\frac{v\cdot(v^{\prime}-v)}{u\cdot (v^{\prime}-v)}\right\} \ u \ \in \ E_{vv^{\prime}}.
\end{eqnarray*}
For $v^{\prime}\in \{v^{\prime}\in\mathbb{R}^3 : |v^{\prime}|\leq 2N\} \backslash \mathfrak{B}$, define the restricted projection
\begin{eqnarray*}
\mathbb{P}^{\prime} \equiv \mathbb{P}|_{\mathbb{P}^{-1}( E_{vv^{\prime}} \cap \{1/N \leq |v_1^{\prime}| \leq N\}
)} \ : \  \mathbb{P}^{-1}( E_{vv^{\prime}} \cap \{1/N \leq |v_1^{\prime}| \leq N\}
) \rightarrow E_{vv^{\prime}}\cap \{1/N \leq |v_1^{\prime}| \leq N\}.
\end{eqnarray*}
Then for $v^{\prime}\in B_{2N} \backslash \mathfrak{B}$ the Jacobian of $\mathbb{P}^{\prime}$ is bounded :
\begin{equation*}
Jac(\mathbb{P}^{\prime}) = \left|\frac{\partial\mathbb{P}^{\prime}}{\partial u}\right| = \left(v\cdot \frac{v^{\prime}-v}{|v^{\prime}-v|}\right)^2 | \sec^2 \theta \tan \theta | \leq  \frac{3N^4}{\varrho},
\end{equation*}
where $\theta$ is defined by $\cos\theta =u\cdot \frac{v^{\prime}-v}{|v^{\prime}-v|}$.
\end{lemma}
\begin{figure}[h]
\begin{center}
\epsfig{file=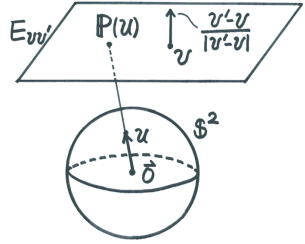,angle=90, height=3.5cm}
\end{center}
\caption{Projection Map}
\label{fig5}
\end{figure}
\begin{proof}
Without loss of generality, we may assume $\frac{v^{\prime}-v}{|v^{\prime}-v|}=(0,0,1)^T$. Using spherical coordinate,
\begin{equation*}
\mathbb{P}^{\prime}(u)=\frac{v\cdot (v^{\prime}-v)}{u\cdot (v^{\prime}-v)} u = \frac{v\cdot\frac{v^{\prime}-v}{|v^{\prime}-v|}}{u\cdot\frac{v^{\prime}-v}{|v^{\prime}-v|}} u
= \frac{v\cdot\frac{v^{\prime}-v}{|v^{\prime}-v|}}{\cos\theta}
\left(\begin{array}{ccc}
\sin\theta \cos\phi \\
\sin\theta\sin\phi \\
\cos\theta
\end{array}\right)
= v\cdot\frac{v^{\prime}-v}{|v^{\prime}-v|}
\left(\begin{array}{ccc}
\tan\theta \cos\phi \\
\tan\theta\sin\phi \\
1
\end{array}\right),
\end{equation*}
and a Jacobian matrix of $\mathbb{P}^{\prime}$,
\begin{equation*}
\frac{\partial\mathbb{P}^{\prime}}{\partial(\theta,\phi)}= v\cdot \frac{v^{\prime}-v}{|v^{\prime}-v|}
\left( \begin{array}{ccc}
\sec^2 \theta \cos\phi & -\tan\theta\sin\phi\\
\sec^2 \theta \sin \phi & \tan\theta \cos\phi\\
\end{array}\right).
\end{equation*}
Therefore a Jacobian of $\mathbb{P}^{\prime}$ is
\begin{equation*}
Jac (\mathbb{P}^{\prime}
) = \left| \frac{\partial\mathbb{P}^{\prime}}{\partial(\theta,\phi)} \right|
= \left(v\cdot\frac{v^{\prime}-v}{|v^{\prime}-v|}\right)^2 \sec^2 \theta |\tan\theta| \leq
\text{dist}(0,E_{vv^{\prime}})^2 |\sec\theta|^3
.
\end{equation*}
Notice that
\begin{eqnarray*}
|\sec\theta| =\frac{1}{|\cos\theta|} =
\frac{1}{\left|u\cdot\frac{v^{\prime}-v}{|v^{\prime}-v|}\right|}
=\left|\left\{\frac{v\cdot (v^{\prime}-v)}{u\cdot(v^{\prime}-v)}\right\}u\right|\frac{1}{\left|v\cdot\frac{v^{\prime}-v}{|v^{\prime}-v|}\right|} =\frac{|\mathbb{P}^{\prime}(u)|}{\text{dist}(0,E_{vv^{\prime}})}.
\end{eqnarray*}
Because $\mathbb{P}^{\prime}(u) \in \{\frac{1}{N} \leq |v_1^{\prime}| \leq N \}$ and $\text{dist}(0,E_{vv^{\prime}})\geq \frac{\varrho}{3N}$ we have
\begin{equation*}
Jac(\mathbb{P}^{\prime}) \leq \frac{|\mathbb{P}^{\prime}(u)|^3}{|\text{dist}(0,E_{vv^{\prime}})|} \leq \frac{3N^4}{\varrho}.
\end{equation*}
\end{proof}

Assume we choose $m_2(U_x \cap \mathbb{S}^2)\leq \frac{\varrho\varepsilon}{N_* N^2}$. By definition we know that $\mathbb{P}^{\prime}(U_x \cap \mathbb{S}^2) = E_{vv^{\prime}} \cap \{\frac{1}{N}\leq |v_1^{\prime}|\leq N\} \cap U_x$ and the 2-dimension Lebesque measure of $E_{vv^{\prime}} \cap \{\frac{1}{N}\leq |v_1^{\prime}|\leq N\} \cap U_x$ is bounded by
\begin{equation*}
m_2 (E_{vv^{\prime}} \cap \{\frac{1}{N}\leq |v_1^{\prime}|\leq N\} \cap U_x )
= m_2 (\mathbb{P}^{\prime}(U_x \cap \mathbb{S}^2) ) \leq Jac(\mathbb{P}^{\prime}) \times |U_x \cap \mathbb{S}^2 | \leq \frac{3N^4}{\varrho} \times \frac{\varepsilon}{N_* N^2}=\frac{3N^2}{\varrho N_* } \varepsilon.
\end{equation*}
Therefore we have an upper bound of (\ref{U1}) :
\begin{equation}
|(\ref{U1})| \leq C N^2 \varepsilon ||h||_{\infty}, \label{sum29}
\end{equation}
where $C=\int_{\mathbb{R}^3}\bar{w}(v^{\prime})dv^{\prime}$. In case of $v\neq 0$, from (\ref{sum28}), (\ref{sum30}) and (\ref{sum29}), we have
\begin{equation}
(\ref{ab2})\leq CN^2 ||h||_{\infty} \sup_{v,v^{\prime}}|\mathbf{Z}(v,v^{\prime})| \times \blacklozenge \leq
CN^4 ||h||_{\infty}^2 \sup_{|v|\leq N, |v^{\prime}|\leq 2N}|\mathbf{Z}(v,v^{\prime})|\{
N^2 \sqrt{\varrho} + (1+\frac{3N^2}{\varrho})\frac{\varepsilon}{N_*}\}
.\label{sum31}
\end{equation}
\\
\newline\textbf{Case of }$\mathbf{v= 0}$\\
In this case, we do not have a upper bound of the Jacobian of $\mathbb{P}^{\prime}$. Instead we will use the structure of $\mathfrak{G}_x$ of Lemma \ref{bouncebacklower} crucially.
In case of $v=0$, we split $(\ref{abc})$
\begin{eqnarray}
&&\int_{|v^{\prime}|\leq 2N} \bar{w}(v^{\prime}) \int_{E_{0v^{\prime}}\cap\{\frac{1}{N}\leq |v_1^{\prime}|\leq N\}} |\bar{w}(v_1^{\prime\prime})h(\bar{t},\bar{x},v_1^{\prime})-\bar{w}(v_1^{\prime})h(t,x,v_1^{\prime})| dv_1^{\prime} dv^{\prime}\nonumber\\
&&= \int_{|v^{\prime}|\leq 2N} \bar{w}(v^{\prime}) {\int_{E_{0v^{\prime}}\cap\{\frac{1}{N}\leq |v_1^{\prime}|\leq N\}} |\bar{w}(v_1^{\prime\prime})h(\bar{t},\bar{x},v_1^{\prime})-\bar{w}(v_1^{\prime})h(t,x,v_1^{\prime})|\times \mathbf{1}_{U_x}(v_1^{\prime})dv_1^{\prime}} dv^{\prime}\label{c1}\\
&&+\int_{|v^{\prime}|\leq 2N} \bar{w}(v^{\prime}) \int_{E_{0v^{\prime}}\cap\{\frac{1}{N}\leq |v_1^{\prime}|\leq N\}}
|\bar{w}(v_1^{\prime\prime})h(\bar{t},\bar{x},v_1^{\prime})-\bar{w}(v_1^{\prime})h(t,x,v_1^{\prime})|\times \mathbf{1}_{E_{0v^{\prime}}\cap \{\frac{1}{N}\leq |v_1^{\prime}|\leq N\} \backslash U_x}(v_1^{\prime})
dv_1^{\prime} dv^{\prime}\label{c2}
\end{eqnarray}
For $v^{\prime}$, we use a spherical polar coordinates $(r^{\prime},\theta^{\prime},\phi^{\prime})$ so that
\begin{equation}
v^{\prime} = (r^{\prime}\sin\theta^{\prime}\cos\phi^{\prime},r^{\prime}\sin\theta^{\prime}\sin\phi^{\prime} ,r^{\prime}\cos\theta^{\prime}).\label{polar1}
\end{equation}
By definition, $E_{0v^{\prime}}$ is a plane containing the origin and normal to $v^{\prime}$. We know that $E_{0v^{\prime}}$ is generated by two unit vectors
\begin{eqnarray*}
E_{0v^{\prime}} = \Bigg{\langle} \left(\begin{array}{ccc}\cos\theta^{\prime} \cos\phi^{\prime}\\ \cos\theta^{\prime}\sin\phi^{\prime} \\ -\sin\theta^{\prime} \end{array}\right),
\left( \begin{array}{ccc} -\sin\phi^{\prime} \\ \cos\phi^{\prime} \\ 0
\end{array}
\right)
\Bigg{\rangle}.
\end{eqnarray*}
We will use a polar coordinate $(r_1^{\prime},\theta_1^{\prime})$ for $v_1^{\prime}\in E_{0v^{\prime}}$, i.e.
\begin{eqnarray}
v_1^{\prime} =
\left(\begin{array}{ccc}
(v_1^{\prime})_1\\
(v_1^{\prime})_2\\
(v_1^{\prime})_3
\end{array}\right)(r_1^{\prime},\theta_1^{\prime};\theta^{\prime},\phi^{\prime})
\equiv r_1^{\prime}\left(\begin{array}{ccc}
\cos\theta^{\prime}\cos\phi^{\prime} & -\sin\phi^{\prime} & \sin\theta^{\prime}\cos\phi^{\prime}\\
\cos\theta^{\prime}\sin\phi^{\prime} & \cos\phi^{\prime} & \sin\theta^{\prime}\sin\phi^{\prime}\\
-\sin\theta^{\prime} & 0 & \cos\theta^{\prime}
\end{array}
\right)\left(\begin{array}{ccc}\cos\theta_1^{\prime}\\ \sin\theta_1^{\prime} \\ 0\end{array}\right).\label{polar2}
\end{eqnarray}
Direct computation gives $\det\left(\frac{\partial(v_1^{\prime})}{\partial(r_1^{\prime}, \theta_1^{\prime},\theta^{\prime})}\right)=$
\begin{eqnarray*}
&&(r_1^{\prime})^2 \cos\theta_1^{\prime} \det\left(\begin{array}{ccc}
\cos\theta^{\prime}\cos\phi^{\prime}\cos\theta_1^{\prime}-\sin\phi^{\prime}\sin\theta\theta_1^{\prime} & -\cos\theta^{\prime} \cos\phi^{\prime} \sin\theta_1^{\prime} -\sin\phi^{\prime}\cos\theta_1^{\prime} & \sin\theta^{\prime} \cos\phi^{\prime}\\
\cos\theta^{\prime} \sin\phi^{\prime}\cos\theta_1^{\prime}+ \cos\phi^{\prime} \sin\theta_1^{\prime} & -\cos\theta^{\prime} \sin\phi^{\prime}\sin\theta_1^{\prime} + \cos\phi^{\prime} \cos\theta_1^{\prime} & \sin\theta^{\prime}\sin\phi^{\prime}\\
\sin\theta^{\prime}\cos\theta_1^{\prime} & \sin\theta^{\prime}\sin\theta_1^{\prime} & \cos\theta^{\prime}
\end{array}\right)
=(r_1^{\prime})^2 \cos\theta_1^{\prime}.
\end{eqnarray*}
Therefore we have following identity
\begin{equation}
\int_{\mathbb{R}^3}\cdot\cdot\cdot \ dv_1^{\prime} = \int_0^{\infty} \int_0^{2\pi} \int_{0}^{\pi}\cdot\cdot\cdot \ (r_1^{\prime})^2 \cos\theta_1^{\prime} d\theta^{\prime} d\theta_1^{\prime} dr_1^{\prime}.\label{identity}
\end{equation}
Recall the standard 3-dimensional polar coordinates and 2-dimensional polar coordinates :
\begin{eqnarray*}
&&\int_{|v^{\prime}|\leq 2N}\cdot\cdot\cdot \ dv^{\prime} = \int_0^{2N} \int_0^{2\pi} \int_0^{\pi}\cdot\cdot\cdot \ (r^{\prime})^2 \sin\theta^{\prime} d\theta^{\prime}d\phi^{\prime}dr^{\prime},\\
&&\int_{E_{0v^{\prime}}\cap\{\frac{1}{N}\leq |v_1^{\prime}|\leq N\}}\cdot\cdot\cdot \ dv_1^{\prime} = \int_{\frac{1}{N}}^N \int_0^{2\pi}\cdot\cdot\cdot \ r_1^{\prime}d\theta_1^{\prime}dr_1^{\prime},
\end{eqnarray*}
and use above identities to control $(\ref{c1})$ by
\begin{eqnarray}
\int_0^{2N}dr^{\prime} (r^{\prime})^2  \bar{w}(r^{\prime}) \int_0^{2\pi} d\phi^{\prime}
\underbrace{\int_0^{\pi} d\theta^{\prime} \sin\theta^{\prime} \int_{\frac{1}{N}}^N dr_1^{\prime} r_1^{\prime} e^{-\frac{(r_1^{\prime})^2}{8}}
\int_0^{2\pi} d\theta_1^{\prime} \mathbf{1}_{U_x}(v_1^{\prime}(r_1^{\prime},\theta_1^{\prime};\theta^{\prime},\phi^{\prime}))}_{\bigotimes} ||h||_{\infty}.\label{C1}
\end{eqnarray}
We focus on the inner integration $\bigotimes$ and divide it into
\begin{eqnarray}
&&\int_0^{\pi}d\theta^{\prime}\sin\theta^{\prime} \int_{\frac{1}{N}}^N dr_1^{\prime} r_1^{\prime} e^{-\frac{(r_1^{\prime})^2}{8}}\int_0^{2\pi} d\theta_1^{\prime}\mathbf{1}_{U_x}(v_1^{\prime})
\mathbf{1}_{\theta_1^{\prime}\in (\frac{\pi}{2}-\varrho, \frac{\pi}{2}+\rho) \cup (\frac{3\pi}{2}-\varrho , \frac{3\pi}{2}+\varrho)}\label{X2}\\
&&+\int_0^{\pi}d\theta^{\prime}\sin\theta^{\prime} \int_{\frac{1}{N}}^N dr_1^{\prime} r_1^{\prime} e^{-\frac{(r_1^{\prime})^2}{8}}\int_0^{2\pi} d\theta_1^{\prime}\mathbf{1}_{U_x}(v_1^{\prime})\mathbf{1}_{\theta_1^{\prime}\in[0,\frac{\pi}{2}-\varrho]\cup[\frac{\pi}{2}+\varrho,\frac{3\pi}{2}-\varrho]\cup[\frac{3\pi}{2}+\varrho,2\pi]}\label{X1}.
\end{eqnarray}
Easily $(\ref{X2})\leq 2\varrho (e^{-\frac{1}{8N^2}}-e^{-\frac{N^2}{8}})\leq 4\varrho$. For (\ref{X1}), we use $1\leq\frac{\cos\theta_1^{\prime}}{\varrho}$ and $1\leq Nr_1^{\prime}$ on $\theta_1^{\prime}\in[0,\frac{\pi}{2}-\varrho]\cup[\frac{\pi}{2}+\varrho,\frac{3\pi}{2}-\varrho]\cup[\frac{3\pi}{2}+\varrho,2\pi]$ and $r_1^{\prime}\in [\frac{1}{N},N]$ to have
\begin{eqnarray}
(\ref{X1})\leq \varrho^{-1}N\int_0^{\pi} d\theta^{\prime}\int_{\frac{1}{N}}^N dr_1^{\prime} (r_1^{\prime})^2 \int_0^{2\pi} d\theta_1^{\prime} \cos\theta_1^{\prime}\mathbf{1}_{U_x}(v_1^{\prime}(r_1^{\prime},\theta_1^{\prime};\theta^{\prime},\phi^{\prime}))
= \varrho^{-1} N \times m_3(U_x\cap \{\frac{1}{N}\leq |v_1^{\prime}|\leq N\}),
\end{eqnarray}
where we used (\ref{identity}).
To sum we have
\begin{eqnarray}
(\ref{c1})\leq (\ref{C1})\leq C ||h||_{\infty} \left\{4\varrho + \varrho^{-1}N \times \frac{\varepsilon}{N_*}\right\}.\label{ssum1}
\end{eqnarray}
\\
On the other hand for (\ref{c2}) we can use Lemma \ref{uniformlycontinuous} to have
\begin{equation}
(\ref{c2}) \leq C \frac{\varepsilon}{N_*}. \label{ssum2}
\end{equation}
From (\ref{ssum1}) and (\ref{ssum2}) we have
\begin{eqnarray}
(\ref{ab2})\leq CN^2 ||h||_{\infty} \sup_{v,v^{\prime}}|\mathbf{Z}(v,v^{\prime})| \times \blacklozenge
&=&CN^2 ||h||_{\infty} \sup_{v,v^{\prime}}|\mathbf{Z}(v,v^{\prime})| \times \{(\ref{c1})+(\ref{c2})\}\nonumber\\
&\leq& CN^2 ||h||_{\infty} \sup_{v,v^{\prime}}|\mathbf{Z}(v,v^{\prime})|\left\{\frac{\varepsilon}{N_*}+||h||_{\infty}\Big\{
4\varrho + \varrho^{-1}N \times \frac{\varepsilon}{N_*}
\Big\}\right\}.\label{sum32}
\end{eqnarray}
To summarize, from (\ref{sum21}), (\ref{sum22}), (\ref{sum23}), (\ref{sum24}), (\ref{sum25}), (\ref{sum26}), (\ref{sum27}), (\ref{sum31}) and (\ref{sum32}), we have established
\begin{eqnarray}
&&(\ref{estimategainterm2})\leq C||h||_{\infty}^2 \{\frac{1}{N}+e^{-\frac{N^2}{16}}\} + C||h||_{\infty}^2 \sup_{|v|\leq N, |v^{\prime}|\leq 2N} |\mathbf{Z}(v,v^{\prime})| \frac{N^{\gamma}}{M^2}
+C||h||_{\infty}^2 \sup_{|v|\leq N, |v^{\prime}|\leq 2N} |\mathbf{Z}(v,v^{\prime})|(N^6 \sqrt{\varrho} + 4N^2 \varrho)\nonumber\\
&& \ \ \ \ \ \ \ \ \ \ +\frac{\varepsilon}{N_*}C ||h||_{\infty} \sup_{|v|\leq N, |v^{\prime}|\leq 2N} |\mathbf{Z}(v,v^{\prime})| \left\{
N^2 + ||h||_{\infty}\Big( 1+ \frac{3N^6}{\varrho}+ N^4 + N^3 \varrho
\Big)
\right\}.
\label{last}
\end{eqnarray}
We choose $N,M,N_*>0$ sufficiently large and $\varrho>0$ sufficiently small so that we can control $(\ref{estimategainterm2}) < \frac{\varepsilon}{2}$. Combining with the result of previous subsection (\ref{finalest1}), we conclude (\ref{continuitystatement}) and and prove Theorem \ref{continuity}.

\subsection{Continuity of Collision Operators $Kf$ and $\Gamma(f,f)$}

The following is a consequence of Theorem \ref{continuity}.
\begin{corollary}\label{continuity1}
Assume $f(t,x,v)$ is continuous on $[0,T]\times (\bar{\Omega}\times\mathbb{R}^3)\backslash\mathfrak{G}$ and
\begin{equation*}
w(v)f(t,x,v) = (1+\rho^2 |v|^2)^{\beta} f(t,x,v) \in L^{\infty}([0,T]\times (\bar{\Omega}\times\mathbb{R}^3)).
\end{equation*}
Then $Kf(t,x,v)$ and $\Gamma_+(f,f)(t,x,v)$ are continuous in $[0,T]\times\Omega\times\mathbb{R}^3$ and
\begin{equation*}
\sup_{[0,T]\times\bar{\Omega}\times\mathbb{R}^3}|\nu^{-1}(v)w(v)K(f)|<\infty, \ \
\sup_{[0,T]\times\bar{\Omega}\times\mathbb{R}^3}|\nu^{-1}(v)w(v)\Gamma_{+}(f,f)| < \infty.\label{boundedKGamma}
\end{equation*}
\end{corollary}
\begin{proof}
The above boundedness is a direct consequence of (\ref{wk}) and (\ref{Gamma1}). Thanks to Theorem \ref{continuity}, we already established the continuity of $\Gamma_+$. Therefore we only need to show the continuity of
\begin{eqnarray*}
\frac{1}{\sqrt{\mu}}Q_-(\sqrt{\mu}f,\mu)= e^{-\frac{|v|^2}{4}}\int_{\mathbb{R}^3}\int_{\mathbb{S}^2} B(v-u,\omega) f(t,x,u) e^{-\frac{|u|^2}{4}} d\omega du.
\end{eqnarray*}
Choose $(\bar{t},\bar{x},\bar{v})\sim (t,x,v)$ so that $|(\bar{t},\bar{x},\bar{v})-(t,x,v)|<\delta$. We will estimate
\begin{eqnarray}
&&\frac{1}{\sqrt{\mu}}Q_- (\sqrt{\mu}f,\mu)(\bar{t},\bar{x},\bar{v})-\frac{1}{\sqrt{\mu}}Q_- (\sqrt{\mu}f,\mu)(t,x,v)\nonumber\\
&=&\frac{1}{\sqrt{\mu}}\int_{\mathbb{R}^3} \int_{\mathbb{S}^2} e^{-\frac{|u|^2}{4}}\{B(v-u,\omega)f(t,x,v)-\underline{B(\bar{v}-u,\omega)f(\bar{t},\bar{x},u)}\} d\omega du\nonumber\\
&=&\frac{1}{\sqrt{\mu}}\int_{\mathbb{R}^3}\int_{\mathbb{S}^2} B(v-u,\omega)e^{-\frac{|u|^2}{4}} f(t,x,v) d\omega du
- \frac{1}{\sqrt{\mu}}\int_{\mathbb{R}^3} \int_{\mathbb{S}^2} B(v-u^{\prime},\omega) e^{-\frac{|u^{\prime}-(v-\bar{v})|^2}{4}} f(\bar{t},\bar{x},u^{\prime}-(v-\bar{v}))
d\omega du^{\prime}\nonumber\\
&\leq& \frac{1}{\sqrt{\mu}}\int_{u\in\mathbb{R}^3} \int_{\mathbb{S}^2} |B(v-u,\omega)| |e^{-\frac{|u|^2}{4}}-e^{-\frac{|u-(v-\bar{v})|^2}{4}}| w^{-1}(u-(v-\bar{v}))||wf||_{\infty} d\omega du\label{difference11}\\
&&+ \frac{1}{\sqrt{\mu}}\int_{\mathbb{R}^3} \int_{\mathbb{S}^2} \underbrace{|B(v-u,\omega)| e^{-\frac{|u|^2}{4}} |f(t,x,u)-f(\bar{t},\bar{x},u-(v-\bar{v}))|}_{\clubsuit}
d\omega du,\label{difference12}
\end{eqnarray}
where we used a change of variables $u^{\prime}= u + (v-\bar{v})$ for the underlined term. Using the Taylor's expansion we control
\begin{equation*}
e^{-\frac{|u-(v-\bar{v})|^2}{4}} = e^{-\frac{|u|^2}{4}} + \frac{1}{2}|u_*| e^{-\frac{|u_*|^2}{4}} |v-\bar{v}| \leq
\frac{1}{2}(|u|+\delta)e^{\frac{\delta^2}{4}} e^{-\frac{|u|^2}{4}} \times |v-\bar{v}|,
\end{equation*}
where $u_* = s_* \{u-(v-\bar{v})\} + (1-s_*) u$ for some $s_*\in (0,1)$ and $|v-\bar{v}|< \delta$. Therefore we control
\begin{eqnarray}
|(\ref{difference11})| &\leq&
e^{\frac{|v|^2}{4}} \int_{\mathbb{R}^3}
|v-u|^{\gamma} \frac{1}{2}(|u|+\delta) e^{\frac{\delta^2}{4}} e^{-\frac{|u|^2}{4}} du \times \sup_{v,u}\int_{\mathbb{S}^2} q_0 \big(\frac{v-u}{|v-u|}\cdot \omega\big) d\omega |v-\bar{v}|\times ||wf||_{\infty}\nonumber\\
&\leq& C (1+|v|)^{\gamma} e^{\frac{|v|^2}{4}} ||wf||_{\infty},\label{1-1}
\end{eqnarray}
where we have used the the angular cutoff assumption (\ref{Gradcutoff}).
Now we estimate (\ref{difference12}) with following steps :
\newline\textbf{Case 1 : } $|u|\geq N$. Since $e^{-\frac{|u|^2}{4}}\leq e^{-\frac{N^2}{8}} e^{-\frac{|u|^2}{8}}$, we estimate
\begin{equation}
\int_{|u|\geq N} \int_{\mathbb{S}^2} \ \clubsuit \ d\omega du \ \leq C e^{-\frac{N^2}{8}} \int_{\mathbb{R}^3}  e^{-\frac{|u|^2}{8}}|u-v|^{\gamma} du \times ||wf||_{\infty}\leq C e^{-\frac{N^2}{8}}\nu(v)||wf||_{\infty}.\label{1-2}
\end{equation}
\newline\textbf{Case 2 : } $|u|\leq N$. A function $f$ is continuous on $[0,T]\times(\bar{\Omega}\times B(0;N))\backslash\mathfrak{G}$. By the Lemma \ref{uniformlycontinuous}, we can choose $U_x \subset B(0;N)$ with $|U_x|<\frac{\varepsilon}{N}$ with $|f(t,x,u)-f(\bar{t},\bar{x},u-(v-\bar{v}))|< \frac{\varepsilon}{N}$ for $|(t,x,u)-(\bar{t},\bar{x},u-(v-\bar{v}))|\leq \delta$ with $u\in B(0;N)\backslash U_x$. Therefore $\int_{|u|\leq N} \int_{\mathbb{S}^2} \ \clubsuit \ d\omega du$ is bounded by
\begin{eqnarray}
\int_{u\in B(0;N)\cap U_x} \int_{\mathbb{S}^2} \ \clubsuit \ d\omega du + \int_{u\in B(0;N)\backslash U_x} \int_{\mathbb{S}^2} \ \clubsuit \ d\omega du
\leq C\frac{\varepsilon}{N} \nu(v) ||wf||_{\infty}.\label{1-3}
\end{eqnarray}
From (\ref{1-1}), (\ref{1-2}) and (\ref{1-3}), we summarize
\begin{equation*}
\frac{1}{\sqrt{\mu}}|Q_-(\sqrt{\mu}f,\mu)(\bar{t},\bar{x},\bar{v})-Q_-(\sqrt{\mu}f,\mu)(t,x,v)|\leq (o(\delta) + e^{-\frac{N^2}{8}}+\frac{\varepsilon}{N}) \frac{\nu(v)}{\sqrt{\mu}}||wf||_{\infty},
\end{equation*}
which is less than $\varepsilon$ for sufficiently large $N$ and sufficiently small $\delta$.
\end{proof}
\\
\\
\\
\newline In following sections, we will prove Theorem 1, Theorem 2 and Theorem 3, for each boundary conditions. In order to write theorems in the unified way \cite{Guo08} for all boundary condition cases, we use the weight function $w(v)=\{1+\rho^2 |v|^2\}^{\beta}$ in (\ref{weight}) and define
\begin{equation*}
h\equiv w(v)\times \frac{F-\mu}{\sqrt{\mu}}.\label{h}
\end{equation*}
In terms of $h$, the Boltzmann equation (\ref{boltzmann}) is equivalent to
\begin{eqnarray}
\{\partial_t + v\cdot\nabla_x + \nu -K_w\} h = w\Gamma(\frac{h}{w},\frac{h}{w}),\label{hBol}
\end{eqnarray}
where $K_w h \equiv w K(\frac{h}{w})$ with boundary conditions :
\begin{eqnarray}
&1.& \text{In-flow boundary condition : }  \ \ \ h(t,x,v)=w(v)g(t,x,v) \ \ \ \ \text{ for } \ (x,v)\in\gamma_-. \ \ \ \ \ \ \ \ \ \ \ \ \  \label{inflow}\\
&2.& \text{Diffuse boundary condition : } \nonumber\\
&& \ \ \ \ h(t,x,v)=w(v)\sqrt{\mu (v)}\int_{v^{\prime }\cdot
n(x)>0}h(t,x,v^{\prime })\frac{1}{w(v^{\prime})\sqrt{\mu(v^{\prime})}}
c_{\mu }\mu(v^{\prime})\{n_{x}\cdot v^{\prime
}\}dv^{\prime } \ \ \ \ \text{ for } \ (x,v)\in\gamma_-,\ \ \ \ \ \ \ \ \ \ \ \ \ \ \ \ \ \ \ \ \ \label{diffuse}\\
&& \ \ \ \ \ \ \ \ \ \ \ \ \ \ \ \ \ \ \ \ \ \ \ \ \ \ \ \ \ \ \ \ \  \ \ \ \ \ \ \ \ \ \ \  \ \ \ \ \ \ \ \ \ \ \ \ \ \ \ \ \ \ \ \ \ \ \ \ \ \ \ \ \ \ \ \ \ \ \text{             with a normalized constant } c_{\mu} \text{ in } (\ref{normalizedconstant}).\nonumber\\
&3.& \text{Bounce-back boundary condition : } \ \ \ h(t,x,v)=h(t,x,-v) \ \ \ \ \text{ for } \ (x,v)\in\gamma_-. \label{bounceback}
\end{eqnarray}

\section{In-Flow Boundary Condition}
In this section, we consider the linearized Boltzmann equation (\ref{hBol}) with the in-flow boundary condition (\ref{inflow}). First we will show the formation of discontinuity using a pointwise estimate of the Boltzmann solution \cite{Guo08}. Then we use the continuity of collision operators, Theorem \ref{continuity}, to show a continuity of solution on the continuity set $\mathfrak{C}$ and the propagation of discontinuity on the discontinuity set $\mathfrak{D}$.
\subsection{Formation of Discontinuity}
We prove Part 1 of Theorem \ref{formationofsingularity}. Without loss of generality we may assume $x_0=(0,0,0)$ and $v_0=(1,0,0)$ and $(x_0,v_0)\in\gamma_0^{\mathbf{S}}$. Locally the boundary is a graph, i.e. $\Omega\cap B(\mathbf{0};\delta) = \{(x_1,x_2,x_3)\in B(\mathbf{0};\delta) : x_3 > \Phi(x_1,x_2)\}$. The condition $(x_0,v_0)\in\gamma_0^{\mathbf{S}}$ implies $t_{\mathbf{b}}(x_0,v_0)\neq 0$ and $t_{\mathbf{b}}(x_0,-v_0)\neq 0$ which means $\Phi(\xi,0)<0$ for $\xi\in(-\delta,\delta)\backslash \{0\}$. (See Figure 3)
\\
For simplicity we assume a zero boundary datum, i.e. $g\equiv 0$. From Theorem 1 of \cite{Guo08}, we have a global solution of the linearized Boltzmann equation (\ref{hBol}) with zero in-flow boundary condition, satisfying
\begin{equation*}
\sup_{t\in [0,\infty)}||h(t)||_{\infty} \leq C^{\prime} e^{-\lambda t}||h_0||_{\infty},
\end{equation*}
for some $\lambda>0$. In the proof we do not use the decay estimate but just boundedness
\begin{equation}
\sup_{t\in [0,\infty)}||h(t)||_{\infty} \leq C^{\prime}||h_0||_{\infty}.\label{boundedness}
\end{equation}
Recall the constants $C_{\mathbf{k}}$ and $C_{\Gamma}$ from (\ref{wk}) and (\ref{Gamma1}).
Choose $t_0 \in (0,\min\{\frac{\delta}{2},\frac{t_{\mathbf{b}}(x_0,-v_0)}{2}\})$ sufficiently small so that
\begin{equation}
\frac{1}{2}\leq\left( e^{-\nu(1)t_0} -t_0 C_{\mathbf{k}}C^{\prime} -(1-e^{-\nu(1)t_0})C_{\Gamma}(C^{\prime})^2
\right),\label{1/2}
\end{equation}
where $\nu(1)\equiv\nu(v_0)$ for any $v_0\in\mathbb{R}^3$ with $|v_0|=1$.
This choice is possible because the right hand side of (\ref{1/2}) is a continuous function of $t_0\in\mathbb{R}$ and it has a value $1$ when $t_0 =0$. Furthermore assume a condition for our initial datum $h_0$ : there is sufficiently small $\delta^{\prime}=\delta^{\prime}(\Omega,t_0)>0$ such that $B((-t_0 ,0,0);\delta^{\prime}) \subset \Omega$ and
\begin{eqnarray}
h_0(x_0,v_0) \equiv ||h_0||_{\infty}>0 \ \ \ \text{for} \ \ (x,v)\in B((-t_0,0,0);\delta^{\prime}) \times B((1,0,0);\delta^{\prime}).\label{supp}
\end{eqnarray}
We claim the Boltzmann solution $h$ with such an initial datum $h_0$ and zero in-flow boundary condition is not continuous at $(t_0,x_0,v_0)=(t_0,(0,0,0),(1,0,0))$.
We will use a contradiction argument : Suppose
\begin{equation}
[h(t_0)]_{x_0,v_0} = 0.\label{contradiction}
\end{equation}
Choose sequences of points $(x_n^{\prime},v_n^{\prime})=((0,0,\frac{1}{n}), (1,0,0))$ and $(x_n ,v_n)=((\frac{1}{n},0,\Phi(\frac{1}{n},0)),(1,0,\frac{1}{n}))$. Because of our choice, for sufficiently large $n\in\mathbb{N}$, the characteristics $[X(0;t_0,x_0,v_0),V(0;t_0,x_0,v_0)]$ is near to $((-t_0,0,0),(1,0,0))$, i.e.
\begin{equation*}
(x_n^{\prime}-t_0 v_n^{\prime},v_n^{\prime} )=((-t_0 ,0, \frac{1}{n}),(1,0,0)) \in B((-t_0,0,0);\delta^{\prime}) \times B((1,0,0);\delta^{\prime}).
\end{equation*}
Hence the Boltzmann solution at $(t_0,x_n^{\prime},v_n^{\prime})$ is
\begin{eqnarray*}
h(t_0,x_n^{\prime},v_n^{\prime}) &=& h_0(x_n^{\prime}-t_0 v_n^{\prime},v_n^{\prime}) e^{-\nu(v_n^{\prime})t_0} + \int_0^{t_0} e^{-\nu(v_n^{\prime})(t_0 -\tau)}\left\{K_w h+w\Gamma(\frac{h}{w},\frac{h}{w})\right\}(\tau, x_n^{\prime}-v(t_0 -\tau),v_n^{\prime})d\tau\\
&=&||h_0||_{\infty}e^{-\nu(v_n^{\prime})t_0} + \int_0^{t_0} e^{-\nu(v_n^{\prime})(t_0 -\tau)}\left\{K_w h+w\Gamma(\frac{h}{w},\frac{h}{w})\right\}(\tau, x-v_n^{\prime}(t_0 -\tau),v_n^{\prime})d\tau.
\end{eqnarray*}
Combining $h(t_0,x_n,v_n) = w(v_n)g(t_0,x_n ,v_n) =0$ with (\ref{contradiction}), we conclude
\begin{equation}
h(t_0^{\prime},x_n^{\prime},v_n^{\prime}) \rightarrow 0 \ \text{as} \ n\rightarrow 0.\label{hhh}
\end{equation}
On the other hand, using (\ref{boundedness}) we can estimate
\begin{eqnarray*}
\liminf_{n\rightarrow \infty}|h(t_0 ,x_n^{\prime},v_n^{\prime})| &=& \liminf_{n\rightarrow \infty}|h(t_0 ,x_n^{\prime},v_n^{\prime})-h(t_0 ,x_n ,v_n)|\\
&\geq&\liminf_{n\rightarrow \infty} \big{|} ||h_0||_{\infty}e^{-\nu(v_n^{\prime})t_0} -
\int_0^{t_0} C_{\mathbf{k}}C^{\prime} ||h_0||_{\infty} d\tau
+\int_0^{t_0} \nu(v_n^{\prime}) e^{-\nu(v_n^{\prime})(t-\tau)} C_{\Gamma}(C^{\prime})^2 ||h_0||_{\infty}^2 d\tau
\big{|}\\
&\geq&  ||h_0||_{\infty} e^{-\nu(1)t_0} - t_0 C_{\mathbf{k}}C^{\prime}||h_0||_{\infty} - (1-e^{-\nu(1)t_0}) C_{\Gamma} (C^{\prime})^2 ||h_0||_{\infty}^2\\
&=& ||h_0||_{\infty} \left( e^{-\nu(1)t_0} -t_0 C_{\mathbf{k}}C^{\prime} -(1-e^{-\nu(1)t_0})C_{\Gamma}(C^{\prime})^2
\right) \geq \frac{||h_0||_{\infty}}{2}\neq 0,
\end{eqnarray*}
which is contradiction to (\ref{hhh}).
\subsection{Continuity away from $\mathfrak{D}$}
We aim to prove Part 1 of Theorem \ref{continuityawayfromD} in this section. First we recall Lemma 12 of \cite{Guo08}, the representation for solution operator $G(t,0)$ for the homogeneous transport equation with in-flow boundary condition :
\begin{lemma}
\label{ginflowdecay} \cite{Guo08} Let $h_{0}\in L^{\infty }$ and $wg\in L^{\infty }.$
Let $\{G(t,0)h_{0}\}$ be the solution to the transport equation
\begin{equation*}
\{\partial _{t}+v\cdot \nabla _{x}\}G(t,0)h_{0}=0,\text{ \ \ \ \ }%
G(0,0)h_{0}=h_{0},\text{ \ \ }\{G(t,0)h_{0}\}_{\gamma _{-}}=wg.
\end{equation*}%
For $(x,v)\notin \gamma _{0}\cap \gamma _{-},$
\begin{eqnarray}
\{G(t,0)h_{0}\}(t,x,v) =\mathbf{1}_{t-t_{\mathbf{b}}\leq 0}h_{0}(x-tv,v)  \notag +\mathbf{1}_{t-t_{\mathbf{b}}>0}\{wg\}(t-t_{%
\mathbf{b}},x-t_{\mathbf{b}}v,v).  \label{inflowformula}
\end{eqnarray}
\end{lemma}
Next we prove a generalized version of Lemma 13 in \cite{Guo08}.
\begin{lemma}[{Continuity away from $\mathfrak{D}$ : Transport Equation}]\label{transportequation}
Let $\Omega$ be an open subset of  $\mathbb{R}^3$ with a smooth boundary $\partial\Omega$ and an initial datum $h_0 (x,v)$ be continuous in $\Omega\times\mathbb{R}^3 \cup\{\gamma_- \cup \gamma_+ \cup \gamma_0^{I-}\}$ and a boundary datum $g$ be continuous in $[0,T]\times \{\gamma_-  \cup \gamma_0^{I-}\}$. Also assume $q(t,x,v)$ and $\phi(t,x,v)$ are continuous in the interior of $[0,T]\times\Omega\times\mathbb{R}^3$ and satisfy $\sup_{[0,T]\times\Omega\times\mathbb{R}^3}\big|q(t,x,v)\big| < \infty$ and $\sup_{[0,T]\times\Omega}\big|\phi(\cdot,\cdot,v)\big| < \infty$ for all $v\in\mathbb{R}^3$. Let $h(t,x,v)$ be the solution of
\begin{equation*}
\{\partial_t + v\cdot\nabla_x + \phi\}h = q \ , \ \ \ h(0,x,v) = h_0 \ , \ \ \ h|_{\gamma_-} = wg.
\end{equation*}
Assume the compatibility condition on $\gamma_- \cup \gamma_0^{I-}$,
\begin{equation}
h_0(x,v)=w(v)g(0,x,v).
\end{equation}
Then the Boltzmann solution $h(t,x,v)$ is continuous on the continuity set $\mathfrak{C}$. Furthermore, if the boundary $\partial\Omega$ does not include a line segment (Definition \ref{linesegment}) then $h(t,x,v)$ is continuous on a complementary set of the discontinuity set, i.e. $\{[0,T]\times\bar{\Omega}\times\mathbb{R}^3\}\backslash\mathfrak{D}$.
\end{lemma}
\begin{proof}
Continuity on $\{\{0\}\times\bar{\Omega}\times\mathbb{R}^3\}\cup \{(0,\infty)\times[\gamma_- \cup \gamma_0^{I-}]\}$ is obvious from the assumption. Fix $(t,x,v)\in \mathfrak{C}$.
Notice that
\begin{equation}
\left\{\frac{d}{ds} \{h(s,X(s),V(s))e^{-\int_s^t \phi(\tau,X(\tau),V(\tau))d\tau}\}
-q(s,X(s),V(s))e^{-\int_s^t \phi(\tau,X(\tau),V(\tau))d\tau}\right\}\mathbf{1}_{[\max\{0,t-t_{\mathbf{b}}(x,v)\},t]}(s)=0,\label{eqntalongXV}
\end{equation}
along the characteristics $X(s;t,x,v)=x-v(t-s), V(s;t,x,v)=v$ until the characteristics hits on the boundary. Choose $(\bar{t},\bar{x},\bar{v})\sim(t,x,v)$ and use a change of variables $\bar{s}= s-(\bar{t}-t)$ with $\bar{s}\in[t-\bar{t},t]$ to have
\begin{eqnarray}
&&\Big\{\frac{d}{d\bar{s}} \{h(\bar{s}+(\bar{t}-t),\bar{X}(\bar{s}),\bar{V}(\bar{s}))e^{-\int_{\bar{s}}^t \phi(\tau+(\bar{t}-t),\bar{X}(\tau),\bar{V}(\tau))d\tau}\}\nonumber
\\&& \ \ \ \ \ \ \ \ \ \ \ \ \ \ \ \ \ \ -q(\bar{s}+(\bar{t}-t),\bar{X}(\bar{s}),\bar{V}(\bar{s}))e^{-\int_{\bar{s}}^t \phi(\tau+(\bar{t}-t),\bar{X}(\tau),\bar{V}(\tau))d\tau}\Big\}\mathbf{1}_{[-(\bar{t}-t)+\max\{0,\bar{t}-t_{\mathbf{b}}(\bar{x},\bar{v})\},t]}(s)=0,\label{eqntalongXV1}
\end{eqnarray}
where $\bar{X}(\bar{s})=X(\bar{s}+(\bar{t}-t);\bar{t},\bar{x},\bar{v})$ and $\bar{V}(\bar{s})=V(\bar{s}+(\bar{t}-t);\bar{t},\bar{x},\bar{v})$.

By the definition $\mathfrak{C}$, we can separate two cases : $t<t_{\mathbf{b}}(x,v)$ , $(x_{\mathbf{b}}(x,v),v)\in \gamma_- \cup \gamma_0^{I-}$.\\
\newline \underline{Case of $t-t_{\mathbf{b}}(x,v)< 0$} \ \
From the assumption $t-t_{\mathbf{b}}(x,v)< 0$, we know that (\ref{eqntalongXV}) holds for $0\leq s\leq t$. Now we choose $(\bar{t},\bar{x},\bar{v})$ near $(t,x,v)$ so that $\bar{t}-t_{\mathbf{b}}(\bar{x},\bar{v})<0$, and $\bar{X}(\bar{s})=X(\bar{s}+(\bar{t}-t);\bar{t},\bar{x},\bar{v})$ is in the interior of $\Omega$ for all $\bar{s}\in[t-\bar{t},t]$. Taking the integration over $[\min\{0,t-\bar{t}\},t]$ of $(\ref{eqntalongXV})-(\ref{eqntalongXV1})$ to have
\begin{eqnarray*}
&&h(t,x,v)-h(\bar{t},\bar{x},\bar{v})=h_0(X(0),V(0))e^{-\int_0^t \phi(\tau,X(\tau),V(\tau))d\tau} -h_0(\bar{X}(t-\bar{t}),\bar{V}(t-\bar{t}))e^{-\int_{t-\bar{t}}^t \phi(\tau+(\bar{t}-t),\bar{X}(\tau),\bar{V}(\tau))d\tau}\\
&& \ \ \ \ \ \ \ \ \ \ \ \ \ \ \ \ \ \ \ \ \ \ \ \ \ \ \ \ \ \ +\int_{\min\{0,t-\bar{t}\}}^t
\Big\{
\mathbf{1}_{[\max\{0,t-t_{\mathbf{b}}(x,v)\},t]}(s) q(s,X(s),V(s))e^{-\int_s^t \phi(\tau,X(\tau),V(\tau))d\tau}\\
&& \ \ \ \ \ \ \ \ \ \ \ \ \ \ \ \ \ \ \ \ \ \ \ \ \ \ \ \ \ \ -\mathbf{1}_{[t-\bar{t}+\max\{0,\bar{t}-t_{\mathbf{b}}(\bar{x},\bar{v})\},t]}(s)
q(s+(\bar{t}-t),\bar{X}(s),\bar{V}(s))
e^{-\int_s^t \phi(\tau+(\bar{t}-t),\bar{X}(\tau),\bar{V}(\tau))d\tau}
\Big\}ds.
\end{eqnarray*}
Since $h_0$ and $\phi$ is continuous, it is easy to see that the first line above goes to zero when $(\bar{t},\bar{x},\bar{v})\rightarrow (t,x,v)$. For the remainder we separate cases : $t-\bar{t}>0$ and $t-\bar{t}\leq 0$. If $t-\bar{t}>0$ the remainder is bounded by
\begin{eqnarray}
\int_{t-\bar{t}}^t | q(s)e^{\int^t_s \phi(\tau)\tau} - q(s+(\tau{t}-t)) e^{-\int_s^t \phi(\tau+(\bar{t}-t)}
| + |t-\bar{t}| \sup_{0\leq s\leq t} ||q(s)||_{\infty} e^{t\sup_{0\leq s\leq t}||\phi(s)||_{\infty}},\nonumber
\end{eqnarray}
where the first term is small using continuity of $q$ and $\phi$, and the second term is small as $(\bar{t},\bar{x},\bar{v})\rightarrow (t,x,v)$. The case $t-\bar{t}\leq 0$ is similar.
\\
\newline \underline{Case of $(x_{\mathbf{b}}(x,v),v)\in \gamma_- \cup \gamma_0^{I-}$} \ \
We only have to consider cases of $t>t_{\mathbf{b}}(x,v)$ and $t=t_{\mathbf{b}}(x,v)$. By definition $(x_{\mathbf{b}}(x,v),v)\in\gamma_- \cup \gamma_0^{I-}$. From Lemma \ref{tbcon}, we know that $t_{\mathbf{b}}(x,v)$ is a continuous function when $(x_{\mathbf{b}}(x,v),v)\notin \gamma_- \cup \gamma_0^{I-}$. In the case of $t>t_{\mathbf{b}}(x,v)$, for $(\bar{t},\bar{x},\bar{v})\sim (t,x,v)$, we have $\bar{t}>t_{\mathbf{b}}(\bar{x},\bar{v})$. Taking the integration over $[\min\{0,t-\bar{t}\},t]$ of $(\ref{eqntalongXV})-(\ref{eqntalongXV1})$ to have
\begin{eqnarray*}
h(t,x,v)-h(\bar{t},\bar{x},\bar{v})&=&wg(t-t_{\mathbf{b}}(x,v),X(t_{\mathbf{b}}(x,v)),V(t_{\mathbf{b}}(x,v)))e^{-\int_{t-t_{\mathbf{b}}(x,v)}^t \phi(\tau,X(\tau),V(\tau))d\tau}\\
&&-wg(\bar{t}-t_{\mathbf{b}}(\bar{x},\bar{v}),X(t_{\mathbf{b}}(\bar{x},\bar{v})),V(t_{\mathbf{b}}(\bar{x},\bar{v})))e^{-\int^t_{\bar{t}-t_{\mathbf{b}}(\bar{x},\bar{v})} \phi(\tau+(\bar{t}-t),\bar{X}(\tau),\bar{V}(\tau))d\tau}\\
&&+\int_{t-t_{\mathbf{b}}(x,v)}^t q(s,X(s),V(s))e^{-\int_s^t \phi(\tau,X(\tau),V(\tau))d\tau}ds\\
&&-\int_{t-t_{\mathbf{b}}(\bar{x},\bar{v})}^t q(s+(\bar{t}-t),\bar{X}(s),\bar{V}(s))
e^{-\int_s^t \phi(\tau+(\bar{t}-t),\bar{X}(\tau),\bar{V}(\tau))d\tau}
ds.
\end{eqnarray*}
Using the continuity of $t_{\mathbf{b}}$ and $q$ and $\phi$, it is easy to show that $|h(t,x,v)-h(\bar{t},\bar{x},\bar{v})|\rightarrow 0$ as $(\bar{t},\bar{x},\bar{v})\rightarrow (t,x,v)$. In the case of $t=t_{\mathbf{b}}(x,v)$ we can choose $(\bar{t},\bar{x},\bar{v})\sim (t,x,v)$ so that $t_{\mathbf{b}}(\bar{x},\bar{v})\in (t-\epsilon ,t+\epsilon)$. Taking the integration over $[\min\{0,t-\bar{t}\},t]$ of $(\ref{eqntalongXV})-(\ref{eqntalongXV1})$ to have
\begin{eqnarray*}
&&|h(t,x,v)-h(\bar{t},\bar{x},\bar{v})|\leq wg(t-t_{\mathbf{b}}(x,v),X(t_{\mathbf{b}}(x,v)),V(t_{\mathbf{b}}(x,v)))e^{-\int_{t-t_{\mathbf{b}}(x,v)}^t \phi(\tau,X(\tau),V(\tau))d\tau} \\
&&\ \ \ \ \ \ \ \ \ \ \ \ \ \ \ \ \ \ \ \ \ \ \ \ \ \ \ \ \ \ -\mathbf{1}_{\bar{t}>t_{\mathbf{b}}(\bar{x},\bar{v})}wg(\bar{t}-t_{\mathbf{b}}(\bar{x},\bar{v}),X(t_{\mathbf{b}}(\bar{x},\bar{v}),V(t_{\mathbf{b}}(\bar{x},\bar{v}))))e^{-\int^t_{\bar{t}-t_{\mathbf{b}}(\bar{x},\bar{v})} \phi(\tau+(\bar{t}-t),\bar{X}(\tau),\bar{V}(\tau))d\tau}\\
&&\ \ \ \ \ \ \ \ \ \ \ \ \ \ \ \ \ \ \ \ \ \ \ \ \  \ \ \ \ -\mathbf{1}_{\bar{t}\leq t_{\mathbf{b}}(\bar{x},\bar{v})}h_0(\bar{X}(t-\bar{t}),\bar{V}(t-\bar{t}))e^{-\int_{t-\bar{t}}^t \phi(\tau+(\bar{t}-t),\bar{X}(\tau),\bar{V}(\tau))d\tau}\\
&&\ \ \ \ \ \ \ \ \ \ \ \ \ \ \ \ \ +\int_{t-t_{\mathbf{b}}(x,v)+\varepsilon}^t
\Big|q(s,X(s),V(s))e^{-\int_s^t \phi(\tau,X(\tau),V(\tau))d\tau}
- q(s+(\bar{t}-t),\bar{X}(s),\bar{V}(s))
e^{-\int_s^t \phi(\tau+(\bar{t}-t),\bar{X}(\tau),\bar{V}(\tau))d\tau}
\Big|
ds\\
&&\ \ \ \ \ \ \ \ \ \ \ \ \ \ \ \ \ \ \ \ \ \ \ \ \  \ \ \ \ + 2\varepsilon \sup_{0\leq s\leq t}||q(s)||_{\infty} e^{t\sup_{0\leq s\leq t}||\phi(s)||_{\infty}},
\end{eqnarray*}
where the first three lines can be small using the compatibility condition and continuity of $h_0$ in $\Omega\times\mathbb{R}^3 \cup\{\gamma_- \cup \gamma_+ \cup \gamma_0^{I-}\}$ and a continuity of $g$ on $[0,T]\times\{\gamma_- \cup \gamma_0^{I-}\}$ and continuity of $\phi$. For the fourth line above, we use the continuity of $q$ and $\phi$.
\\
If the boundary $\partial\Omega$ does not include a line segment (Definition \ref{linesegment}) we have $\mathfrak{C} = \{[0,T]\times\bar{\Omega}\times\mathbb{R}^3\}\backslash\mathfrak{D}$.
\end{proof}
\\
\newline\textbf{Proof of Part 1 of Theorem \ref{continuityawayfromD}}
\newline We will use the following iteration scheme
\begin{eqnarray}
\{\partial_t + v\cdot\nabla_x + \nu \}h^{m+1} = K_w h^m + w \Gamma_+\left(\frac{h^m}{w},\frac{h^m}{w}\right) -w \Gamma_-\left(\frac{h^m}{w},\frac{h^{m+1}}{w}\right),\label{hm}
\end{eqnarray}
with $h^{m+1}|_{t=0}=h_0$ and $h^{m+1}(t,x,v) = wg(t,x,v)$ with $(t,x,v)\in\gamma_- \cup \gamma_0^{I-}$. For simplicity we define
\begin{equation}
q^m = K_w h^m + w\Gamma_+\left(\frac{h^m}{w},\frac{h^m}{w}\right)
-w\Gamma_-\left(\frac{h^m}{w},\frac{h^{m+1}}{w}\right).\label{qmm}
\end{equation}
\newline \textbf{Step 1} : We claim
\begin{equation}
h^i \text{ \ is a continuous function in \ } \mathfrak{C}_T\label{mcontii}
\end{equation}
for all $i\in\mathbb{N}$ and for any $T>0$ where
\begin{equation}
\mathfrak{C}_T \equiv \mathfrak{C} \cap \{[0,T]\times\bar{\Omega}\times\mathbb{R}^3\}, \label{CT}
\end{equation}
where the continuity set $\mathfrak{C}$ is defined in (\ref{C}). We will use mathematical induction to show (\ref{mcontii}). We choose $h^0 =0$ then (\ref{mcontii}) is satisfied for $i=0$. Assume (\ref{mcontii}) for all $i=0,1,2,...,m$.
Rewrite $w\Gamma_-\left(\frac{h^m}{w},\frac{h^{m+1}}{w}\right) = \nu\left(\sqrt{\mu}\frac{h^m}{w}\right)h^{m+1}$ then the equation of $h^{m+1}$ is
\begin{eqnarray}
\{\partial_t + v\cdot\nabla_x + \nu(v) +\nu\left(\sqrt{\mu}\frac{h^m}{w}\right) \}h^{m+1} = K_w h^m + w \Gamma_+\left(\frac{h^m}{w},\frac{h^m}{w}\right).\label{hmequation}
\end{eqnarray}
From Theorem \ref{continuity} and Corollary \ref{continuity1} we know that $\nu\left(\sqrt{\mu}\frac{h^m}{w}\right)$ and $w\Gamma_+\left(\frac{h^m}{w},\frac{h^m}{w}\right)$ is continuous in $[0,T]\times\Omega\times\mathbb{R}^3$. Apply Lemma \ref{transportequation} where $\phi(t,x,v)$ corresponds to $\nu(v) +\nu(\sqrt{\nu}\frac{h^m}{w})$ and $q(t,x,v)$ corresponds to the right hand side of (\ref{hmequation}). Then we check (\ref{mcontii}) for $i=m+1$.
\newline \textbf{Step 2} : We claim that there exist $C>0$ and $\delta>0$ such that if $C\{||h_0||_{\infty} + \sup_{0\leq s < \infty}||wg(s)||_{\infty}\} < \delta$ and $C||h_0||_{\infty}<\delta$ then there exists $T= T(C,\delta)>0$ so that
\begin{equation}
\sup_{0\leq s\leq T}||h^m(s)||_{\infty} \leq C ||h_0||_{\infty},\label{boundedh-m}
\end{equation}
for all $m\in\mathbb{N}$. Moreover $\{h^m\}_{m=0}^{\infty}$ is Cauchy in $L^{\infty}([0,T]\times\bar{\Omega}\times\mathbb{R}^3)$.\\
First we will show a boundedness (\ref{boundedh-m}) for all $m\in\mathbb{N}$. We use mathematical induction on $m$. Assume $\sup_{0\leq s\leq T}||h^m(s)||_{\infty} \leq C ||h_0||_{\infty}$ where $T>0$ will be determined later. Integrating (\ref{hm}) along the trajectory, we have
\begin{eqnarray*}
h^{m+1}(t,x,v)&=&\mathbf{1}_{t<t_{\mathbf{b}}(x,v)}e^{-\nu(v)t} h_0 (x-tv,v) + \mathbf{1}_{t \geq t_{\mathbf{b}}(x,v)} e^{-\nu(v)t_{\mathbf{b}}(x,v)} w(v)g(t-t_{\mathbf{b}}(x,v),x_{\mathbf{b}}(x,v),v)\\
&&+\int_{\max\{t-t_{\mathbf{b}}(x,v),0\}}^t e^{-\nu(v)(t-s)}\big\{ {K_w h^m + w\Gamma_+\left(\frac{h^m}{w},\frac{h^m}{w}\right)
-w\Gamma_-\left(\frac{h^m}{w},\frac{h^{m+1}}{w}\right)}\big\}(s,x-(t-s)v,v) ds\\
&\leq& ||h_0||_{\infty} + \sup_{0\leq s\leq t}||wg(s)||_{\infty}
+ t C_{\mathbf{k}} \sup_{0\leq s\leq t}||h^m (s)||_{\infty}
+C_{\Gamma} \sup_{0\leq s\leq t} ||h^m(s)||_{\infty} \sup_{0\leq s\leq t}\left(
||h^m(s)||_{\infty}+||h^{m+1}(s)||_{\infty}
\right),
\end{eqnarray*}
and
\begin{eqnarray*}
\sup_{0\leq s \leq t} ||h^{m+1}(s)||_{\infty} &\leq& \frac{
1+tC_{\mathbf{k}}C + C_{\Gamma}C \{||h_0||_{\infty}+\sup_s ||wg(s)||_{\infty}\}
}{1-C_{\Gamma}C\{||h_0||_{\infty}+\sup_{s}||wg(s)||_{\infty}\}} \left\{ ||h_0||_{\infty} +\sup_{0\leq s\leq t}||wg(s)||_{\infty}
\right\}\\
&\leq& C \left\{ ||h_0||_{\infty} +\sup_{0\leq s\leq t}||wg(s)||_{\infty}
\right\},
\end{eqnarray*}
where we choose $C>4$ and then $ \{||h_0||_{\infty}+\sup_{0\leq s\leq t}||wg(s)||_{\infty}\}\leq\frac{1}{2C_{\Gamma}C}$ and then $T = \frac{C-3}{2C_{\mathbf{k}}C}$.\\
Newt we will show the sequence $\{h^m\}$ is Cauchy in $L^{\infty}([0,T]\times\bar{\Omega}\times\mathbb{R}^3)$. The equation of $h^{m+1}-h^m$ is
\begin{eqnarray}
&&\{\partial_t + v\cdot\nabla_x + \nu\}(h^{m+1}-h^m) = \tilde{q}^m,\label{hm1}\\
&& \ \ (h^{m+1}-h^m)|_{t=0} = 0 , \ \ (h^{m+1}-h^m)|_{\gamma_-} =0\nonumber
\end{eqnarray}
where
\begin{scriptsize}\begin{equation}\tilde{q}^m = K_w(h^m-h^{m-1}) + w\Gamma_+\left(\frac{h^m}{w},\frac{h^m-h^{m-1}}{w}\right)
-w\Gamma_+ \left(\frac{h^{m-1}-h^m}{w},\frac{h^{m-1}}{w}\right)-w\Gamma_-\left(\frac{h^m}{w},\frac{h^{m+1}-h^m}{w}\right)
+w\Gamma_-\left(\frac{h^{m-1}-h^m}{w},\frac{h^m}{w}\right).\label{tildeqm}
\end{equation}\end{scriptsize}
From (\ref{wk}) and (\ref{Gamma1}), we have a bound of $\tilde{q}_m$,
\begin{eqnarray}
&&\sup_{0\leq s\leq t}||\tilde{q}^m(s)||_{\infty} \leq C_{\mathbf{k}}\sup_{0\leq s\leq t}||\{h^m-h^{m-1}\}(s)||_{\infty}\label{qm}\\
&&+C_{\Gamma}\nu(v) \{\sup_{0\leq s\leq t} ||\{h^{m}-h^{m-1}\}(s)||_{\infty} + \sup_{0\leq s\leq t} ||\{h^{m+1}-h^{m}\}(s)||_{\infty} \} \times (\sup_{0\leq s\leq t}||h^m(s)||_{\infty} + \sup_{0\leq s\leq t}||h^{m+1}(s)||_{\infty}). \nonumber
\end{eqnarray}
Integrating (\ref{hm1}) along the trajectory, we have
\begin{eqnarray*}
||\{h^{m+1}-h^m\}(t)||_{\infty} &\leq& \int_0^t e^{-\nu(v)(t-s)} ||\tilde{q}^m(s,x-(t-s)v,v)||_{\infty} ds\\
&\leq&
C_{\mathbf{k}}t \sup_{0\leq s\leq t}||\{h^m -h^{m-1}\}(s)||_{\infty}\\
&&+ CC_{\Gamma}\left(||h_0||_{\infty}+ \sup_{0\leq s\leq t}||wg(s)||_{\infty}\right) \left\{ \sup_{0\leq s\leq t}||\{h^m-h^{m-1}\}(s)||_{\infty} +\sup_{0\leq s\leq t}||\{h^{m+1}-h^m\}(s)||_{\infty}
\right\}.
\end{eqnarray*}
If we choose $CC_{\Gamma}||h_0||_{\infty} \leq \frac{1}{4}$ and $C_{\mathbf{k}}T \leq \frac{1}{8}$ then
\begin{equation*}
\sup_{0\leq s\leq T} ||\{h^{m+1}-h^m \}(s)||_{\infty} \leq \frac{1}{2} \sup_{0\leq s\leq T} ||\{h^{m}-h^{m-1} \}(s)||_{\infty}.
\end{equation*}
Then we have
\begin{eqnarray*}
\sup_{0\leq s\leq T} ||\{h^{m} -h^{m-1}\}(s)||_{\infty} &\leq& \sup_{0\leq s\leq T}||\{h^m -h^{m-1}\}(s)||_{\infty} +\cdot\cdot\cdot + \sup_{0\leq s\leq T} ||\{h^{n+1}-h^n\}(s)||_{\infty}\\
&\leq& \{ \frac{1}{2^{m-n-1}} + \cdot\cdot\cdot + \frac{1}{2^0}\} \sup_{0\leq s\leq T} ||\{h^{n+1}-h^n\}(s)||_{\infty}\leq
 \frac{2}{2^{n}} \sup_{0\leq s\leq T} ||\{h^1-h^0\}(s)||_{\infty}\\
&\leq& \frac{4}{2^{n}} C\{||h_0||_{\infty} + \sup_{0\leq s\leq T}||wg(s)||_{\infty}\},
\end{eqnarray*}
which means that the sequence $\{h^m\}$ is Cauchy in $L^{\infty}([0,T]\times\bar{\Omega}\times\mathbb{R}^3)$.
\newline \textbf{Step 3} : From previous steps we obtain that $h$ with $\lim_{n\rightarrow\infty}h^n$ is continuous function on $\mathfrak{C}_T$. Now we claim that $h$ is continuous in $\mathfrak{C}$. Notice that $T$ only depends on $||h_0||_{\infty}$ and $\sup_{0\leq s\leq T}||wg(s)||_{\infty}$. Using unform bound of $\sup_{0\leq s < \infty}||h(s)||_{\infty}$ (Theorem 1 of \cite{Guo08}) we can obtain the continuity for $h$ for all time by repeating $[0,T],[T,2T],...$. If the boundary $\partial\Omega$ does not include a line segment (Definition \ref{linesegment}) then every step is valid with $[0,\infty)\times\{\bar{\Omega}\times\mathbb{R}^3\}\backslash\mathfrak{D}$ instead of $\mathfrak{C}$ and $[0,T]\times\{\bar{\Omega}\times\mathbb{R}^3\}\backslash\mathfrak{D}$ instead of $\mathfrak{C}_T$.
\subsection{Propagation of Discontinuity}
\textbf{Proof of 1 of Theorem \ref{propagation}}
\newline \textbf{Proof of (\ref{1sidedinequality})} \\
In order to show the upper bound of discontinuity jump (\ref{1sidedinequality}), we will show
\begin{equation}
[h(t)]_{x_0 +(t-t_0)v_0,v_0} \leq [h]_{t_0,x_0,v_0} e^{-(\frac{1}{C_{\nu}}+ \frac{C^{\prime}}{C_w}||h_0||_{\infty})(1+|v_0|)^{\gamma}(t-t_0)},
\end{equation}
when $(x_0,v_0)\in\gamma_0^{\mathbf{S}}$ and $t\in (t_0,t_0 +t_{\mathbf{b}}(x_0,-v_0))$. Choose two points $(x^{\prime},v^{\prime}),(x^{\prime\prime},v^{\prime\prime})\in \{\bar{\Omega}\times\mathbb{R}^3 \backslash \mathfrak{G}\}\cap B((x,v);\delta)\backslash (x,v)$ and compare the representation
\begin{scriptsize}
\begin{eqnarray*}
|h(t,x^{\prime},v^{\prime})-h(t,x^{\prime\prime},v^{\prime\prime})|
&\leq&
\Big{|}\mathbf{1}_{t-t_0\geq t_{\mathbf{b}}(x^{\prime},v^{\prime})}h(t-t_{\mathbf{b}}(x^{\prime},v^{\prime}), x_{\mathbf{b}}(x^{\prime},v^{\prime}), v^{\prime})e^{-\nu(v^{\prime})t_{\mathbf{b}}(x^{\prime},v^{\prime})
-\int_{t-t_{\mathbf{b}}(x^{\prime},v^{\prime})}^t
\nu(\sqrt{\mu}\frac{h}{w})(\tau,x^{\prime}-(t-\tau)v^{\prime},v^{\prime})d\tau}\label{firstterm}\\
&&+
\mathbf{1}_{t-t_0< t_{\mathbf{b}}(x^{\prime},v^{\prime})}h(t_0, x^{\prime}-(t-t_0)v^{\prime}, v^{\prime})e^{-\nu(v^{\prime})(t-t_0) -\int_{t_0}^t \nu(\sqrt{\mu}\frac{h}{w})(\tau,x^{\prime}-(t-\tau)v^{\prime},v^{\prime})d\tau}
\nonumber\\
&&-\mathbf{1}_{t-t_0\geq t_{\mathbf{b}}(x^{\prime\prime},v^{\prime\prime})}h(t-t_{\mathbf{b}}(x^{\prime\prime},v^{\prime\prime}), x_{\mathbf{b}}(x^{\prime\prime},v^{\prime\prime}), v^{\prime\prime})e^{-\nu(v^{\prime\prime})t_{\mathbf{b}}(x^{\prime\prime},v^{\prime\prime})-\int_{t-t_{\mathbf{b}}(x^{\prime\prime},v^{\prime\prime})}^t \nu(\sqrt{\mu}\frac{h}{w})(\tau,x^{\prime}-(t-\tau)v^{\prime},v^{\prime})d\tau}\label{thirdterm}\\
&&-\mathbf{1}_{t-t_0< t_{\mathbf{b}}(x^{\prime\prime},v^{\prime\prime})}h(t_0, x^{\prime\prime}-(t-t_0)v^{\prime\prime}, v^{\prime\prime})e^{-\nu(v^{\prime\prime})(t-t_0) -\int_{t_0}^t \nu(\sqrt{\mu}\frac{h}{w})(\tau,x^{\prime\prime}-(t-\tau)v^{\prime\prime},v^{\prime\prime})d\tau}
\Big{|}\nonumber\\
&+&\Big{|}\int_{\max\{0,t-t_0-t_{\mathbf{b}}(x^{\prime},v^{\prime})\}}^t \{K_w h + w\Gamma_+ (\frac{h}{w},\frac{h}{w})\}
(s,x^{\prime}-(t-s)v^{\prime},v^{\prime})
e^{-\nu(v^{\prime})(t-s)-\int_s^t \nu(\sqrt{\mu}\frac{h}{w})(\tau,x^{\prime}-(t-\tau)v^{\prime},v^{\prime})d\tau}
ds\nonumber\\
&&-\int_{\max\{0,t-t_0-t_{\mathbf{b}}(x^{\prime\prime},v^{\prime\prime})\}}^t \{K_w h + w\Gamma_+ (\frac{h}{w},\frac{h}{w})\}
(s,x^{\prime\prime}-(t-s)v^{\prime\prime},v^{\prime\prime})
e^{-\nu(v^{\prime\prime})(t-s)-\int_s^t \nu(\sqrt{\mu}\frac{h}{w})(\tau,x^{\prime\prime}-(t-\tau)v^{\prime\prime},v^{\prime\prime})d\tau}
ds\Big{|}\nonumber.
\end{eqnarray*}
\end{scriptsize}
It is easy to see that if $t-t_0 \geq t_{\mathbf{b}}(x^{\prime},v^{\prime})$ then as $\delta\rightarrow 0$ we have
\begin{eqnarray*}
t-t_{\mathbf{b}}(x^{\prime},v^{\prime}) \rightarrow t_0 \ , \ x_{\mathbf{b}}(x^{\prime},v^{\prime}) \rightarrow x_0 \ , \
\end{eqnarray*}
and if $t-t_0 < t_{\mathbf{b}}(x^{\prime},v^{\prime})$ then as $\delta\rightarrow 0$ we have
\begin{eqnarray*}
x^{\prime}-(t-t_0)v^{\prime} \rightarrow x_0.
\end{eqnarray*}
Therefore the first four lines converge to $[h]_{t_0,x_0,v_0} \times e^{-\nu(v_0)(t-t_0)-\int_{t_0}^t \nu(\sqrt{\mu}\frac{h}{w})(\tau,x_0 -(t_0-\tau)v_0,v_0)d\tau}$. For the last two lines, using the continuity of $K_w h, \Gamma(\frac{h}{w},\frac{h}{w}), \nu(\sqrt{\mu}\frac{h}{w})$ we conclude that it converges to zero. Therefore we have
\begin{eqnarray*}
[h(t)]_{x_0 + (t-t_0)v_0,v_0} &\leq& [h]_{t_0,x_0,v_0} e^{-\nu(v_0)(t-t_0)-\int_{t_0}^t \nu(\sqrt{\mu}\frac{h}{w})(\tau, x_0 -(t_0 -\tau)v_0, v_0) d\tau}\\
&\leq& [h]_{t_0,x_0,v_0} \times e^{-(\frac{1}{C_{\nu}}-C_w C^{\prime}||h_0||_{\infty})(1+|v_0|)^{\gamma} (t-t_0)},
\end{eqnarray*}
where we used
\begin{eqnarray}
\nu_w(v)\equiv \int_{\mathbb{R}^3}\int_{\mathbb{S}^2}B(v-u,\omega)e^{-\frac{|u|^2}{4}}w^{-1}(u) d\omega du
\label{nuw}\\
\text{with} \ \ \ \frac{1}{C_w}(1+|v|)^{\gamma} \leq \nu_w (v) \leq C_w (1+|v|)^{\gamma}.\label{lowerupperboundsnuw}
\end{eqnarray}
Remark that \textbf{Proof of (\ref{1sidedinequality})} is valid for in-flow, diffuse and bounce-back cases.
\newline \textbf{\newline \textbf{Proof of (\ref{decaydiscontinuity})}}\\
Assume $[h(t_0)]_{x_0,v_0}\neq 0$ and $t_0\in (0,t_{\mathbf{b}}(x_0,-v_0))$ with $(x_0,v_0)\in\gamma_0^{\mathbf{S}}$. Further assume that the boundary $\partial\Omega$ is strictly concave at $x_0$ along $v_0$ direction (\ref{convexity}).
\\
\textbf{Step 1 } Claim : We can choose sequences $(t_n^{\prime},x_n^{\prime},v_n^{\prime}), (t_n^{\prime\prime},x_n^{\prime\prime},v_n^{\prime\prime}) \in [0,\infty)\times\bar{\Omega}\times\mathbb{R}^3 \cap B((t_0,x_0,v_0);\frac{1}{n})\backslash (t_0,x_0,v_0)$ such that $\lim_{n\rightarrow \infty} |h(t_n^{\prime},x_n^{\prime},v_n^{\prime})-h(t_n^{\prime\prime},x_n^{\prime\prime},v_n^{\prime\prime})|\geq \frac{1}{2}[h(t_0)]_{x_0,v_0}\neq 0$. \\
From $[h(t_0)]_{x_0,v_0}\neq 0$ we may assume
\begin{equation}
\sup_{(x_0^{\prime},v_0^{\prime}),(x_0^{\prime\prime},v_0^{\prime\prime})\in B((x_0,v_0);\frac{1}{n})\backslash (x_0,v_0)} |h(t_0,x_0^{\prime},v_0^{\prime})- h(t_0,x_0^{\prime\prime},v_0^{\prime\prime})| \geq \frac{3}{4} [h(t_0)]_{x_0,v_0}\neq 0,
\end{equation}
for all $n\in \mathbb{N}$. And for each $n\in\mathbb{N}$ we can choose desired sequences.
\\
\\
\textbf{Step 2 } Claim : For given $\varepsilon>0$, up to subsequence we may assume that
\begin{eqnarray}
(x_{\mathbf{b}}(x_n^{\prime},v_n^{\prime}),v_n^{\prime})\in B((x_0,v_0);\varepsilon)\backslash \mathfrak{G} \ , \ (x_{\mathbf{b}}(x_n^{\prime\prime},v_n^{\prime\prime}),v_n^{\prime\prime})\notin B((x_0,v_0);\varepsilon)\cup \mathfrak{G}
\ \ \ \text{for all } n\in\mathbb{N}. \ \label{twosequences}
\end{eqnarray}
We remark that a continuity $G(t,x,v)=w(v)g(t,x,v)$ on $[0,\infty)\times\{\gamma_- \cup \gamma_0^{\mathbf{S}}\}$, i.e.
\begin{eqnarray}
[ \ G |_{[0,\infty)\times\gamma_-}]_{t_0,x_0,v_0} = w(v_0)[ \ g |_{[0,\infty)\times\gamma_-}]_{t_0,x_0,v_0}=0
 \ \ \ \text{for all } \ (t_0,x_0,v_0)\in[0,\infty)\times\{\gamma_- \cup \gamma_0^{\mathbf{S}}\} \label{contG}
\end{eqnarray}
is crucially used in this step. In order to show the final goal (\ref{twosequences}) of this step, we need to prove following statement.
\begin{eqnarray}
&&\text{Assume } (x_0,v_0)\in\gamma_0^{\mathbf{S}} \text{ and } t_{\mathbf{b}}(x_0,v_0)>t_0. \text{ Then for sufficiently small } \varepsilon>0 \text{ there exists } N>0 \text{ such that }\nonumber\\
&&\text{if }(x,v)\in B((x_0,v_0);\frac{1}{n}) \text{ for } n>\mathbf{N}
\text{ and } x_{\mathbf{b}}(x,v)\notin B((x_0,v_0);\varepsilon)
\text{ then we have } t_{\mathbf{b}}(x,v)>t_0.\label{canreachthere}
\end{eqnarray}
We will prove (\ref{canreachthere}) later and show (\ref{twosequences}) using (\ref{canreachthere}). It suffices to show that there are only finite $n\in\mathbb{N}$ such that
\begin{eqnarray}
&&(x_{\mathbf{b}}(x_n^{\prime},v_n^{\prime}),v_n^{\prime})\in B((x_0,v_0);\frac{1}{n})\backslash\mathfrak{G} \ , \ (x_{\mathbf{b}}(x_n^{\prime\prime},v_n^{\prime\prime}),v_n^{\prime\prime})\in B((x_0,v_0);\frac{1}{n})\backslash\mathfrak{G}
\label{sequence1}
\\
&&or \ \ (x_{\mathbf{b}}(x_n^{\prime},v_n^{\prime}),v_n^{\prime})\notin B((x_0,v_0);\frac{1}{n})\cup\mathfrak{G} \ , \ (x_{\mathbf{b}}(x_n^{\prime\prime},v_n^{\prime\prime}),v_n^{\prime\prime})\notin B((x_0,v_0);\frac{1}{n})\cup\mathfrak{G}.\label{sequence2}
\end{eqnarray}
Suppose there are infinitely many $n^{\prime}\in\mathbb{N}$ satisfying (\ref{sequence1}). If $\varepsilon>0$ is sufficiently small then (\ref{sequence1}) implies that $t_0 >t_{\mathbf{b}}(x_{n^{\prime}}^{\prime},v_{n^{\prime}}^{\prime})$ and $t_0 >t_{\mathbf{b}}(x_{n^{\prime}}^{\prime\prime},v_{n^{\prime}}^{\prime\prime})$. The Boltzmann solution $h$ at $(t_0,x_{n^{\prime}}^{\prime},v_{n^{\prime}}^{\prime})$ is
\begin{eqnarray*}
h(t_0,x_{n^{\prime}}^{\prime},v_{n^{\prime}}^{\prime})=
{h(t_0 -t_{\mathbf{b}}(x_{n^{\prime}}^{\prime},v_{n^{\prime}}^{\prime}),x_{\mathbf{b}}(x_{n^{\prime}}^{\prime},v_{n^{\prime}}^{\prime}),v_{n^{\prime}}^{\prime})}
e^{-\nu(v_{n^{\prime}}^{\prime})(t_0-t_{\mathbf{b}}(x_{n^{\prime}}^{\prime},v_{n^{\prime}}^{\prime}))-\int_{t_0-t_{\mathbf{b}}(x_{n^{\prime}}^{\prime},v_{n^{\prime}}^{\prime})}^{t_0} \nu(\sqrt{\mu}\frac{h}{w})(\tau,x_{n^{\prime}}^{\prime}-(t_0-\tau)v_{n^{\prime}}^{\prime},v_{n^{\prime}}^{\prime})d\tau}\label{wg}\\
+\int_{t_0-t_{\mathbf{b}}(x_{n^{\prime}}^{\prime},v_{n^{\prime}}^{\prime})}^{t_0} \{K_w h + \Gamma_+(\frac{h}{w},\frac{h}{w})\}(s,x_{n^{\prime}}^{\prime}-(t_0-s)v_{n^{\prime}}^{\prime},v_{n^{\prime}}^{\prime}) e^{-\nu(v_{n^{\prime}}^{\prime})(t_0 -s)-\int_{s}^{t_0} \nu(F)(\tau,x_{n^{\prime}}^{\prime}-(t_0-\tau)v_{n^{\prime}}^{\prime},v_{n^{\prime}}^{\prime})d\tau} ds,\nonumber
\end{eqnarray*}
and a similar representation for $h(t_0,x_{n^{\prime}}^{\prime\prime},v_{n^{\prime}}^{\prime\prime})$. Compare representations of $h(t_0,x_{n^{\prime}}^{\prime},v_{n^{\prime}}^{\prime})$ and $h(t_0,x_{n^{\prime}}^{\prime\prime},v_{n^{\prime}}^{\prime\prime})$ to conclude
\begin{eqnarray*}
&&\lim_{n^{\prime}\rightarrow\infty}|h(t_0,x_{n^{\prime}}^{\prime},v_{n^{\prime}}^{\prime})
-h(t_0,x_{n^{\prime}}^{\prime\prime},v_{n^{\prime}}^{\prime\prime})|
=\lim_{n^{\prime}\rightarrow\infty}|h(t_0-t_{\mathbf{b}}(x_{n^{\prime}}^{\prime}, v_{n^{\prime}}^{\prime}),x_{\mathbf{b}}(x_{n^{\prime}}^{\prime}, v_{n^{\prime}}^{\prime}) , v_{n^{\prime}}^{\prime})
-h(t_0-t_{\mathbf{b}}(x_{n^{\prime}}^{\prime\prime},v_{n^{\prime}}^{\prime\prime}),x_{\mathbf{b}}(x_{n^{\prime}}^{\prime\prime}, v_{n^{\prime}}^{\prime\prime}) ,v_{n^{\prime}}^{\prime\prime})|\nonumber\\
&&\ \ \ \ \ \ \ \ \ \ \ \ \ \ \ \ \ \ \ \ \ \ \ \ \ \ \ \ \ \ \ \ \ \ \ \ \ \ \ \ \ \ \ \ \ \ \ \ \ \ \ \ \times
e^{-\nu(v_0)(t_0 -t_{\mathbf{b}}(x_0,v_0))-\int_{t_0 -t_{\mathbf{b}}(x_0,v_0)}^{t_0} \nu(\sqrt{\mu}\frac{h}{w})(\tau,x_0 -(t_0-\tau)v_0,v_0)d\tau}\\
&&\leq
[h|_{[0,\infty)\times\gamma_-}]_{t_0 -t_{\mathbf{b}}(x_0,v_0),x_{\mathbf{b}}(x_0,v_0),v_0}
\times
e^{-\nu(v_0)(t_0 -t_{\mathbf{b}}(x_0,v_0))-\int_{t_0 -t_{\mathbf{b}}(x_0,v_0)}^{t_0} \nu(\sqrt{\mu}\frac{h}{w})(\tau,x_0 -(t_0-\tau)v_0,v_0)d\tau},
\label{compareh}
\end{eqnarray*}
where we used the continuity of $\nu(\sqrt{\mu}\frac{h}{w})$ and $\Gamma_+ (\frac{h}{w},\frac{h}{w})$. Further using the in-flow boundary condition $h|_{\gamma_-}=wg$, we have
\begin{eqnarray*}
\lim_{n^{\prime}\rightarrow\infty}|h(t_0,x_{n^{\prime}}^{\prime},v_{n^{\prime}}^{\prime})
-h(t_0,x_{n^{\prime}}^{\prime\prime},v_{n^{\prime}}^{\prime\prime})|
\leq [ \ g|_{[0,\infty)\times \gamma_-}]_{t_0, x_0,v_0}
 w(v_0)
e^{-\nu(v_0)(t_0 -t_{\mathbf{b}}(x_0,v_0))-\int_{t_0 -t_{\mathbf{b}}(x_0,v_0)}^{t_0} \nu(\sqrt{\mu}\frac{h}{w})(\tau,x_0 -(t_0-\tau)v_0,v_0)d\tau} =0,\label{continuityuptogamma0}
\end{eqnarray*}
where we used the continuity of $g$ on $[0,\infty)\times\{\gamma_- \cup \gamma_0\}$, (\ref{contG}) at the last line. This is contradict because we choose the sequences $(x_{n^{\prime}}^{\prime},v_{n^{\prime}}^{\prime}) , (x_{n^{\prime}}^{\prime\prime},v_{n^{\prime}}^{\prime\prime})$ satisfying $\lim_{n\rightarrow \infty}|h(t_0,x_{n^{\prime}}^{\prime},v_{n^{\prime}}^{\prime})-h(t_0,x_{n^{\prime}}^{\prime\prime},v_{n^{\prime}}^{\prime\prime})|\geq \frac{1}{2}[h(t_0)]_{x_0,v_0}\neq 0$ in \textbf{Step 1}.\\
Now suppose there are infinitely many $n^{\prime\prime}\in\mathbb{N}$ satisfying (\ref{sequence2}). Because of (\ref{canreachthere}) we have $t_0 < t_{\mathbf{b}}(x^{\prime}_{n^{\prime\prime}},v^{\prime}_{n^{\prime\prime}})$ and $t_0 < t_{\mathbf{b}}(x^{\prime\prime}_{n^{\prime\prime}},v^{\prime\prime}_{n^{\prime\prime}})$.
The Boltzmann solution $h$ at $(t_0,x_{n^{\prime\prime}}^{\prime},v_{n^{\prime\prime}}^{\prime})$ is
\begin{eqnarray*}
&&h(t_0,x_{n^{\prime\prime}}^{\prime},v_{n^{\prime\prime}}^{\prime})=
h_0(x_{n^{\prime\prime}}^{\prime}-t_0 v_{n^{\prime\prime}}^{\prime},v_{n^{\prime\prime}}^{\prime}),v_{n^{\prime\prime}}^{\prime})
e^{-\nu(v_{n^{\prime\prime}}^{\prime})t_0-\int_{0}^{t_0} \nu(\sqrt{\mu}\frac{h}{w})(\tau,x_{n^{\prime\prime}}^{\prime}-(t_0-\tau)v_{n^{\prime\prime}}^{\prime},v_{n^{\prime\prime}}^{\prime})d\tau}\\
&&+\int_{0}^{t_0} \{K_w h + \Gamma_+(\frac{h}{w},\frac{h}{w})\}(s,x_{n^{\prime\prime}}^{\prime}-(t_0-s)v_{n^{\prime\prime}}^{\prime},v_{n^{\prime\prime}}^{\prime}) e^{-\nu(v_{n^{\prime\prime}}^{\prime})(t_0 -s)-\int_{s}^{t_0} \nu(\sqrt{\mu}\frac{h}{w})(\tau,x_{n^{\prime\prime}}^{\prime}-(t_0-\tau)v_{n^{\prime\prime}}^{\prime},v_{n^{\prime\prime}}^{\prime})d\tau} ds,
\end{eqnarray*}
and same representation for $h(t_0,x_{n^{\prime\prime}}^{\prime\prime},v_{n^{\prime\prime}}^{\prime\prime})$. Using the continuity of $h_0$ we see that
\begin{eqnarray*}
&&\lim_{n\rightarrow\infty}|h(t_0,x_{n^{\prime\prime}}^{\prime},v_{n^{\prime\prime}}^{\prime})-h(t_0,x_{n^{\prime\prime}}^{\prime\prime},v_{n^{\prime\prime}}^{\prime\prime})|\nonumber\\
&&= \lim_{n\rightarrow\infty}|h_0(x_{n^{\prime\prime}}^{\prime}-t_0 v_{n^{\prime\prime}}^{\prime} ,v_{n^{\prime\prime}}^{\prime})-h_0(x_{n^{\prime\prime}}^{\prime\prime}-t_0 v_{n^{\prime\prime}}^{\prime\prime},v_{n^{\prime\prime}}^{\prime\prime})|
 w(v_0)
e^{-\nu(v_0)(t_0 -t_{\mathbf{b}}(x_0,v_0))-\int_{t_0 -t_{\mathbf{b}}(x_0,v_0)}^{t_0} \nu(\sqrt{\mu}\frac{h}{w})(\tau,x_0 -(t_0-\tau)v_0,v_0)d\tau}\nonumber\\
&&=0,
\end{eqnarray*}
which is also contradiction.\\
Now we prove (\ref{canreachthere}). We can choose $\varepsilon>0$ sufficiently small so that $\partial\Omega \cap B(x_0;\varepsilon) = \{
(x_1,x_2,\Phi(x_1,x_2))\in B(x_0 ;\varepsilon)\}$. From $t_{\mathbf{b}}(x_0,v_0)> t_0$ we know that a line segment between $x_0$ and $x_0-t_0 x_0$ has only one intersection point $x_0$ with $\partial\Omega$, i.e.$\overline{x_0,x_0-t_0 v_0} \cap \partial\Omega = \{x_0\}$. Furthermore we can choose $\varrho>0$ so large that $\bigcup_{s\in [0,t_0]}B(x_0-sv_0;\varrho) \cap \partial\Omega \subset B(x_0;\varepsilon)$. Choose $N\in\mathbb{N}$ sufficiently large so that $\overline{x,x-t_0 v} \subset \bigcup_{s\in [0,t_0]}B(x_0-sv_0;\varrho) $ for all $(x,v)\in B((x_0,v_0);\frac{1}{n})$. If $x_{\mathbf{b}}(x,v)\notin B((x_0,v_0);\varepsilon)$ then $\overline{x,x-t_0 v} \cap \partial\Omega = \emptyset$ and this implies $t_{\mathbf{b}}(x,v)>t_0$. \\
\\
\textbf{Step 3 } Claim : Choose $t>0$ so that $t-t_0 \in [0,t_{\mathbf{b}}(x_0,-v_0))$ and denote $x=x_0 + (t-t_0) v_0 , \ v=v_0$. Then there exists $N\in\mathbb{N}$ so that $t-t_0 < t_{\mathbf{b}}(x_{n}^{\prime},-v_{n}^{\prime})$ for all $n>N$. \\
Using (\ref{canreachthere}) we only have to prove $x_{\mathbf{b}}(x_{n}^{\prime},-v_{n}^{\prime}) \notin B((x_0,-v_0);\varepsilon)$. From (\ref{twosequences}) we know that $x_{\mathbf{b}}(x_n^{\prime},v_n^{\prime})\in B(x_0 ; \varepsilon)$. We assume that $\Omega \cap B(x_0 ; \varepsilon) = \{x\in B(x_0 ; \varepsilon) : x_3 > \Phi(x_1,x_2)\}$ and $n(x_0)= (0,0,-1)$ and $v_0 = |v_0|(1,0,0)$. Let's define
\begin{equation}
\Psi(s) = \Phi((x_n^{\prime})_1+ s(v_n^{\prime})_1, (x_n^{\prime})_2 + s(v_n^{\prime})_2) - ((x_n^{\prime})_3 + s(v_n^{\prime})_2).\nonumber
\end{equation}
Since $x_n^{\prime}\in \Omega$ we have $\Psi(0)<0$ and $\Psi(t_{\mathbf{b}}(x_n^{\prime},-v_n^{\prime}))=0=\Psi(-t_{\mathbf{b}}(x_n^{\prime},v_n^{\prime}))$.
Because of the strict concavity along $v_0$ direction at $x_0$ (\ref{convexity}), for sufficiently large $n$ so that $(x_n^{\prime},v_n^{\prime})\sim (x_0,v_0)$ we have
\begin{equation}
\Psi^{\prime\prime}(s) = ((v_n^{\prime})_1, (v_n^{\prime})_2) \left(\begin{array}{ccc} \partial_{x_1}^2 \Phi & \partial_{x_1}\partial_{x_2} \Phi \\
\partial_{x_2}\partial_{x_1} \Phi & \partial_{x_2}^2 \Phi
\end{array}\right)
\left(\begin{array}{ccc} (v_n^{\prime})_1 \\ (v_n^{\prime})_2
\end{array}\right) < -\frac{1}{2} C_{x_0,v_0},\nonumber
\end{equation}
where the Hessian of $\Phi$ is evaluated at $((x_n^{\prime})_1+ s(v_n^{\prime})_1 ,(x_n^{\prime})_2+ s(v_n^{\prime})_2 )$. Since $\{ x_n^{\prime} + s v_n^{\prime} : s \in (-t_{\mathbf{b}}(x_n^{\prime},v_n^{\prime}),t_{\mathbf{b}}(x_n^{\prime},-v_n^{\prime}) )
\}\subset \Omega$ we have $\Psi(s)<0$ for $s\in (-t_{\mathbf{b}}(x_n^{\prime},v_n^{\prime}),t_{\mathbf{b}}(x_n^{\prime},-v_n^{\prime}) ).$ Therefore $\Phi^{\prime}(-t_{\mathbf{b}}(x_n^{\prime},v_n^{\prime}))\leq 0$ and $\Phi^{\prime}(t_{\mathbf{b}}(x_n^{\prime},-v_n^{\prime}))\geq 0$. This is contradiction because
\begin{equation*}
0 \leq \Phi^{\prime}(t_{\mathbf{b}}(x_n^{\prime},-v_n^{\prime})) = \Phi^{\prime}(-t_{\mathbf{b}}(x_n^{\prime},v_n^{\prime})) + \int_{-t_{\mathbf{b}}(x_n^{\prime},v_n^{\prime})}^{t_{\mathbf{b}}(x_n^{\prime},-v_n^{\prime})} \Phi^{\prime\prime}(s) ds \leq 0 -\frac{1}{2} C_{x_0,v_0} \{t_{\mathbf{b}}(x_n^{\prime},-v_n^{\prime})+ t_{\mathbf{b}}(x_n^{\prime},v_n^{\prime})\}<0.
\end{equation*}
The consequence of this step is that for $n> N$ we have a representation of $h$ at $(t,x,v)$
\begin{eqnarray}
h(t,x_n^{\prime}+ (t-t_0)v_n^{\prime} ,v_n^{\prime}) &=& h(t_0 ,x_n^{\prime},v_n^{\prime}) e^{-\nu(v_n^{\prime})(t-t_0)-\int_{t_0}^t \nu(\sqrt{\mu}\frac{h}{w})(\tau,x_n^{\prime}+(\tau-t_0)v_n^{\prime},v_n^{\prime})d\tau}\label{htxv}\\
&+& \int_{t_0}^t \{K_w + w\Gamma_+ (\frac{h}{w},\frac{h}{w})\}(s, x_n + (s-t_0)v_n^{\prime},v_n^{\prime})
e^{-\nu(v_n^{\prime})(t-s)-\int_{s}^t \nu(\sqrt{\mu}\frac{h}{w})(\tau, x_n^{\prime}+(\tau -t_0)v_n^{\prime},v_n^{\prime})d\tau} ds\nonumber
\end{eqnarray}
\\
\textbf{Step 4 } Claim : For given $\varepsilon>0$ there exists $\delta>0$ so that if $|(y,u)-(x_0,v_0)|<\delta$ and $|(x,v)-(x_0,v_0)|<\delta$ and $t_0 < t_{\mathbf{b}}(y,u)$ and $t_0 < t_{\mathbf{b}}(x,v)$ then
\begin{equation}
|h(t_0 , y , u ) - h(t_0, x , v )| < \varepsilon.  \ \label{unifconti}
\end{equation}
We have $h(t_0, y , u )=h_0(y-t_0 u, u) e^{-\nu(u)t_0 -\int_{0}^{t_0}\nu(\sqrt{\mu}\frac{h}{w})(\tau,y-(t_0 -\tau)u, u)d\tau}$
\begin{eqnarray*}
+\int_{0}^{t_0}\{ K_w h + \Gamma_+ (\frac{h}{w},\frac{h}{w})\}(s,y-(t_0-s)u, u) e^{-\nu(u)(t_0 -s) -\int_s^{t_0}\nu(\sqrt{\mu}\frac{h}{w})(\tau,y-(t_0-\tau)u,u)d\tau}ds,
\end{eqnarray*}
and similarly $h(t_0, x , v )=h_0(x-t_0 v, v ) e^{-\nu(v)t_0 -\int_{0}^{t_0}\nu(\sqrt{\mu}\frac{h}{w})(\tau,x-(t_0 -\tau)v, v)d\tau}$
\begin{eqnarray*}
+\int_{0}^{t_0}\{ K_w h + \Gamma_+ (\frac{h}{w},\frac{h}{w})\}(s,x-(t_0-s)v, v) e^{-\nu(v)(t_0 -s) -\int_s^{t_0}\nu(\sqrt{\mu}\frac{h}{w})(\tau,x-(t_0-\tau)v,v)d\tau}ds.
\end{eqnarray*}
Let's compare the arguments of two representations :
\begin{eqnarray*}
|(y-t_0 u,u)-(x-t_0 v,v)|< 2(1+t_0)\delta &\text{ for }& h_0,\\
|(\tau,y-(t_0 -\tau)u,u)-(\tau,x-(t_0 -\tau)v, v)| < 2(1+t_0)\delta &\text{ for }& \nu(\sqrt{\mu}\frac{h}{w}),\\
|(s,y-(t_0 -s)u,u)-(s,x-(t_0 -s)v, v)| < 2(1+t_0)\delta &\text{ for }& K_w h + \Gamma_+ (\frac{h}{w},\frac{h}{w}).
\end{eqnarray*}
Using the continuity of $h_0, \ \nu(\sqrt{\mu}\frac{h}{w}), \ K_w h$ and $\Gamma_+(\frac{h}{w},\frac{h}{w})$ we can choose desired $\varepsilon>0$ to conclude (\ref{unifconti}).\\
\\
\textbf{Step 5 } Claim : Choose $t>0$ so that $t\in [t_0,t_0+t_{\mathbf{b}}(x_0,-v_0))$ and denote $x=x_0 + (t-t_0) v_0 , \ v=v_0$. Let $\varepsilon \leq \frac{1}{10}[h(t_0)]_{x_0,v_0} $ and $\delta > 0$ be chosen in \textbf{Step 4}. Then we can choose $u_{n}^{\prime\prime}\in\Omega$ so that $|u_{n}^{\prime\prime}-v_{n}^{\prime\prime}|< \delta$ and $t_0 < t_{\mathbf{b}}(x_{n}^{\prime\prime},u_{n}^{\prime\prime})$ and $t-t_0 < t_{\mathbf{b}}(x_{n}^{\prime\prime},-u_{n}^{\prime\prime})$. \\
If there are infinitely many $u_n^{\prime\prime}$ so that $t_0 < t_{\mathbf{b}}(x_n^{\prime\prime},u_n^{\prime\prime})$ and $t-t_0 < t_{\mathbf{b}}(x_n^{\prime\prime},-u_n^{\prime\prime})$ then up to subsequence we can define $u_n^{\prime\prime} = v_n^{\prime\prime}$. Therefore we may assume $t-t_0 \geq t_{\mathbf{b}}(x_n^{\prime\prime},-v_n^{\prime\prime})$ for all $n\in\mathbb{N}$. We assume that $\Omega \cap B(x_0 ; \varepsilon) = \{x\in B(x_0 ;\varepsilon) : x_3 > \Phi(x_1 ,x_2)\}$ and $n(x_0)= (0,0,-1)$ and $v_0 = |v_0|(1,0,0)$. Now we illustrate how to choose such a $u_n^{\prime\prime}$. Denote $x_n^{\prime\prime}=x=(x_1,x_2, x_3 )$ and $v_n^{\prime\prime}=(v_1,v_2,v_3)$. First we will choose $(u_1,u_2,u_3)$ and $s>0$ so that
\begin{equation}
n(x_{\mathbf{b}}(x,-u))\cdot u=0,\label{grazing}
\end{equation}
and $x_{\mathbf{b}}(x,-u)=(x_1 + s\frac{u_1}{\sqrt{u_1^2 + u_2^2}},x_2 + s\frac{u_2}{\sqrt{u_1^2 + u_2^2}}, \Phi(x_1 + s\frac{u_1}{\sqrt{u_1^2 + u_2^2}}, x_2 + s\frac{u_2}{\sqrt{u_1^2 + u_2^2}}) )$. The condition (\ref{grazing}) implies that
\begin{eqnarray}
\frac{u_3}{\sqrt{u_1^2 + u_2^2}}=\frac{d}{ds}\Phi(x_1 + s\frac{u_1}{\sqrt{u_1^{2}+u_2^2}}, x_2 + s\frac{u_2}{\sqrt{u_1^2 + u_2^2}}) = \frac{\Phi(x_1+ s\frac{u_1}{\sqrt{u_1^2 +u_2^2}})-x_3}{s}.\label{tangent}
\end{eqnarray}
In order to use the implicit function theorem we define
\begin{eqnarray*}
&&\Psi(x_1,x_2,x_3;u_1,u_2;s) = \Phi(x_1 + s\frac{u_1}{\sqrt{u_1^{2}+u_2^2}}, x_2 + s\frac{u_2}{\sqrt{u_1^2 + u_2^2}})-x_3\\
&&-s \left\{\frac{u_1}{\sqrt{u_1^2 + u_2^2}} \partial_{x_1}\Phi(x_1 + s\frac{u_1}{\sqrt{u_1^2 + u_2^2}},x_2+s\frac{u_2}{\sqrt{u_1^2 +u_2^2}})+ \frac{u_2}{\sqrt{u_1^2 + u_2^2}} \partial_{x_2}\Phi(x_1 + s\frac{u_1}{\sqrt{u_1^2 + u_2^2}},x_2+s\frac{u_2}{\sqrt{u_1^2 +u_2^2}})\right\},
\end{eqnarray*}
and compute ,using (\ref{convexity})
\begin{equation}
\partial_{s}\Psi = -s (\frac{u_1}{\sqrt{u_1^2+u_2^2}},\frac{u_2}{\sqrt{u_1^2 + u_2^2}}) \left(\begin{array}{ccc}\partial_{x_1}^2 \Phi & \partial_{x_1}\partial_{x_2}\Phi \\ \partial_{x_1}\partial_{x_2}\Phi & \partial_{x_2}^2 \Phi\end{array}\right)\left(\begin{array}{ccc} \frac{u_1}{\sqrt{u_1^2+u_2^2}}\\ \frac{u_2}{\sqrt{u_1^2+u_2^2}} \end{array}\right)< -\frac{1}{2}C_{x_0,v_0},
\end{equation}
for $x\sim x_0$, $v\sim v_0$ and the Hessian is evaluated at $(x_1 +s\frac{u_1}{\sqrt{u_1^2+u_2^2}} ,x_2 +s\frac{u_2}{\sqrt{u_1^2+u_2^2}})$. Hence $s=s(x_1,x_2,x_3;w_1,w_2)$ is a smooth function near $x\sim x_0$ and $(u_1,u_2)\sim(v_1,v_2)$. In order to study the behavior of $s$ we use the Taylor's expansion : from $\Psi(x_1,x_2,x_3;u_1,u_2;s)=0$ we have
\begin{eqnarray*}
\Phi(x_1,x_2)-x_3 = \frac{1}{u_1^2 +u_2^2}\bigg\{ (u_1,u_2)\underbrace{\left(\begin{array}{ccc} \partial_{x_1}^2 \Phi & \partial_{x_1}\partial_{x_2}\Phi\\ \partial_{x_1}\partial_{x_2}\Phi & \partial_{x_2}^2 \Phi
\end{array}\right)}_{(*)}\left(\begin{array}{ccc}u_1 \\u_2\end{array}\right) - \frac{1}{2}(u_1,u_2)\underbrace{\left(\begin{array}{ccc} \partial_{x_1}^2 \Phi & \partial_{x_1}\partial_{x_2}\Phi\\ \partial_{x_1}\partial_{x_2}\Phi & \partial_{x_2}^2 \Phi
\end{array}\right)}_{(**)}\left(\begin{array}{ccc}u_1 \\u_2\end{array}\right)
\bigg\} s^2,
\end{eqnarray*}
where the Hessian $(*)$ is evaluated at $(x_1 +s_* \frac{u_1}{\sqrt{u_1^2 +u_2^2}}, x_2 +s_* \frac{u_2}{\sqrt{u_1^2 +u_2^2}})$ and the Hessian $(**)$ is evaluated at $(x_1 +s_{**} \frac{u_1}{\sqrt{u_1^2 +u_2^2}}, x_2 +s_{**} \frac{u_2}{\sqrt{u_1^2 +u_2^2}})$ with $s_* , s_{**}\in (0,s)$. For $x\sim x_0$ and $(u_1,u_2)\sim (v_1,v_2)$ we know that the right hand side of the above equation converges to
\begin{equation*}
-\frac{1}{2(v_1^2 + v_2^2)} (v_1,v_2)\left(\begin{array}{ccc} \partial_{x_1}^2 \Phi ((x_0)_1,(x_0)_2) & \partial_{x_1}\partial_{x_2} \Phi ((x_0)_1,(x_0)_2) \\ \partial_{x_1}\partial_{x_2}\Phi((x_0)_1,(x_0)_2)  & \partial_{x_2}^2 \Phi((x_0)_1,(x_0)_2)
\end{array}\right)\left(\begin{array}{ccc}v_1 \\ v_2\end{array}\right) \neq 0.
\end{equation*}
Hence we have a control of $s$, i.e
\begin{equation}
\frac{1}{C} |\Phi(x_1,x_2)-x_3|^{\frac{1}{2}} \leq s\leq C |\Phi(x_1,x_2)-x_3|^{\frac{1}{2}}. \label{controlofs}
\end{equation}
From (\ref{tangent}), $u_3 = \sqrt{u_1^2 + u_2^2} \frac{d}{ds}\Phi(x_1 + s\frac{u_1}{\sqrt{u_1^2 + u_2^2}},x_2 + s\frac{u_2}{\sqrt{u_1^2 + u_2^2}} )$ equals
\begin{eqnarray}
\sqrt{u_1^2 + u_2^2} \left(\begin{array}{ccc}\frac{u_1}{\sqrt{u_1^2 + u_2^2}}\\ \frac{u_2}{\sqrt{u_1^2 + u_2^2}}\end{array}\right)\cdot \left(\begin{array}{ccc}
\partial_{x_1}\Phi(x_1,x_2) + \frac{u_1}{\sqrt{u_1^2+u_2^2}} \partial_{x_1}^2 \Phi(x^{\prime}_1 ,x^{\prime}_2) s + \frac{u_2}{\sqrt{u_1^2 + u_2^2}}\partial_{x_1}\partial_{x_2}\Phi(x^{\prime}_1,x^{\prime}_2)s\\
\partial_{x_2}\Phi(x_1,x_2) + \frac{u_1}{\sqrt{u_1^2+u_2^2}} \partial_{x_1}\partial_{x_2} \Phi(x^{\prime}_1 ,x^{\prime}_2) s + \frac{u_2}{\sqrt{u_1^2 + u_2^2}}\partial_{x_2}^2\Phi(x^{\prime}_1,x^{\prime}_2)s
\end{array}\right),\label{dd1}
\end{eqnarray}
where $x_1^{\prime} = x_1 + s^{\prime}\frac{u_1}{\sqrt{u_1^2 + u_2^2}}, x_2^{\prime} = x_2 + s^{\prime}\frac{u_2}{\sqrt{u_1^2 + u_2^2}}$ for some $0<s^{\prime}<s\leq C |\Phi(x_1,x_2)-x_3|^{\frac{1}{2}}$. Using the smoothness of $\Phi$ we can bound (\ref{dd1}) as
\begin{eqnarray}
\frac{1}{C}|(u_1,u_2)|\left( |(x_1,x_2)| + |\Phi(x_1,x_2)-x_3|^{\frac{1}{2}}
\right) \leq (\ref{dd1})\leq
C |(u_1,u_2)|\left( |(x_1,x_2)| + |\Phi(x_1,x_2)-x_3|^{\frac{1}{2}}
\right).\label{ubounded}
\end{eqnarray}
To sum for fixed $x$ and direction $\frac{1}{\sqrt{u_1^2 + u_2^2}}(u_1,u_2)$ we can choose $u_3$ such that $n(x_{\mathbf{b}}(x,-(u_1,u_2,u_3)))\cdot (u_1,u_2,u_3)=0$ and $u_3$ is controlled by (\ref{ubounded}).
Finally we choose $(u_1,u_2)= \sqrt{\frac{u_1^2 + u_2^2}{v_1^2 + v_2^2}}(v_1,v_2)$ and find the corresponding $u_3$ so that $|v|=|(u_1,u_2,u_3)|$. Define $u_n^{\prime\prime}= -v + 2(v\cdot\frac{(u_1,u_2,u_3)}{|(u_1,u_2,u_3)|})(u_1,u_2,u_3)$. Then we have desired $u_n^{\prime\prime}$ for sufficiently large $n\in\mathbb{N}$. \\
\\
\textbf{Step 6 } To sum for $(t,x_n^{\prime\prime}+(t-t_0)u_n^{\prime\prime}, u_n^{\prime\prime})$ we have $t-t_0 < t_{\mathbf{b}}(x_n^{\prime\prime},-u_n^{\prime\prime})$ and $t_0 < t_{\mathbf{b}}(x_n^{\prime\prime},u_n^{\prime\prime})$ and $|h(t_0 ,x_n^{\prime\prime},u_n^{\prime\prime})-h(t_0,x_n^{\prime\prime},v_n^{\prime\prime})|< \frac{1}{10}[h(t_0)]_{x_0,v_0}$. Hence the representation of the Boltzmann solution $h$ at $(t,x_n^{\prime\prime}+(t-t_0)u_n^{\prime\prime}, u_n^{\prime\prime})$ is given by
\begin{eqnarray*}
h(t,x_n^{\prime\prime}+(t-t_0)v_n^{\prime\prime},u_n^{\prime\prime}) &=& h(t_0,x_n^{\prime\prime},u_n^{\prime\prime})e^{-\nu(u_n^{\prime\prime})(t-t_0)-\int_{t_0}^t \nu(\sqrt{\mu}\frac{h}{w})(\tau, x_n^{\prime\prime}+(\tau-t_0)u_n^{\prime\prime},u_n^{\prime\prime})d\tau}\\
&+& \int_{t_0}^t \{K_w h + w\Gamma(\frac{h}{w},\frac{h}{w})\}(s,x_n^{\prime\prime}+(s-t_0)u_n^{\prime\prime},u_n^{\prime\prime}) e^{-\nu(u_n^{\prime\prime})(t-s)-\int_{s}^t \nu(\sqrt{\mu}\frac{h}{w})(\tau, x_n^{\prime\prime}+(\tau-t_0)u_n^{\prime\prime},u_n^{\prime\prime})d\tau}ds.
\end{eqnarray*}
Using (\ref{htxv}) we have
\begin{eqnarray*}
&&\lim_{n\rightarrow \infty} |h(t,x_n^{\prime}+(t-t_0)v_n^{\prime},v_n^{\prime})-h(t,x_n^{\prime\prime}+(t-t_0)u_n^{\prime\prime},u_n^{\prime\prime})|\\
&&= \lim_{n\rightarrow \infty} |h(t_0,x_n^{\prime},v_n^{\prime})-h(t_0 ,x_n^{\prime\prime},u_n^{\prime\prime})| e^{-\nu(v_0)(t-t_0)-\int_{t_0}^t \nu(\sqrt{\mu}\frac{h}{w})(\tau, x_0 +(\tau - t_0)v_0,v_0)d\tau}\\
&&\geq \left\{
 \lim_{n\rightarrow \infty} |h(t_0,x_n^{\prime},v_n^{\prime})-h(t_0 ,x_n^{\prime\prime},v_n^{\prime\prime})|
 - \lim_{n\rightarrow \infty} |h(t_0,x_n^{\prime\prime},v_n^{\prime\prime})-h(t_0 ,x_n^{\prime\prime},u_n^{\prime\prime})|
\right\}e^{-\nu(v_0)(t-t_0)-\int_{t_0}^t \nu(\sqrt{\mu}\frac{h}{w})(\tau, x_0 +(\tau - t_0)v_0,v_0)d\tau}\\
&&\geq \frac{1}{4} [h(t_0)]_{x_0,v_0}e^{-\nu(v_0)(t-t_0)-\int_{t_0}^t \nu(\sqrt{\mu}\frac{h}{w})(\tau, x_0 +(\tau - t_0)v_0,v_0)d\tau},
\end{eqnarray*}
which implies that
\begin{equation*}
[h(t)]_{x,v} \geq \frac{1}{4}[h(t_0)]_{x_0,v_0} \times e^{-(C_{\mu} + C^{\prime}C_w ||h_0||_{\infty})(1+|v|)^{\gamma}(t-t_0)} \neq 0.
\end{equation*}
\newline \textbf{Remark}  \textit{Through Step 1 to Step 6, we only used the in-flow boundary datum $g$ explicitely in Step 2. All the other steps are valid for diffuse and bounce-back boundary condition cases. In Step 2, we only used (\ref{contG}) the continuity of $G=wg$ on $[0,\infty)\times\{\gamma_- \cup \gamma_0^{\mathbf{S}}\}$. Therefore, if we can show the continuity of $F$ on $[0,\infty)\times\{\gamma_- \cup \gamma_0^{\mathbf{S}}\}$ then we can prove (\ref{decaydiscontinuity}). For diffuse and bounce-back boundary we will prove such a continuity to conclude (\ref{decaydiscontinuity}).}

\section{Diffuse Reflection Boundary Condition}
In this section, we consider the linearized Boltzmann equation (\ref{hBol}) with the diffuse boundary condition (\ref{diffuse}). In spite of the averaging effect of diffuse reflection operator, we can observe the formation and propagation of discontinuity. Continuity away from $\mathfrak{D}$ is also established.
\subsection{Formation of Discontinuity}
We prove Part 2 of Theorem \ref{formationofsingularity}. The idea of proof is similar to in-flow case but we also use $|v_0|$ not only $t_0$ as a parameter. Without loss of generality we may assume $x_0=(0,0,0)$ and $v_0=(|v_0|,0,0)$ and $(x_0,v_0)\in\gamma_0^{\mathbf{S}}$. Locally the boundary is a graph, i.e. $\Omega\cap B(\mathbf{0};\delta) = \{(x_1,x_2,x_3)\in B(\mathbf{0};\delta) : x_3 > \Phi(x_1,x_2)\}$ and $\Phi(\xi,0)<0$ for $\xi\in(-\delta,\delta)\backslash\{0\}$. (See Figure 3)

Assume that $||h_0||_{\infty}<\delta$ is sufficiently small so that the global solution $h$ of (\ref{hBol}) with diffuse boundary (\ref{diffuse}) has a uniform bound (\ref{boundedness}), from Theorem 4 of \cite{Guo08}. Choose $t_0 \in (0,\min\{\delta, t_{\mathbf{b}}(x_0,-v_0)\})$ sufficiently small and $|v_0|>0$ sufficiently large so that
\begin{equation}
\frac{1}{2}\leq \left( e^{-\nu(|v_0|)t_0} -t_0 C_{\mathbf{k}}C^{\prime} -(1-e^{-\nu(|v_0|)t_0})C_{\Gamma}(C^{\prime})^2 -C^{\prime}\frac{1}{\tilde{w}(v_0)}\int_{\{v_1^{\prime}>0\}}\tilde{w}(v^{\prime})d\sigma(v^{\prime})
\right),\label{D1/2}
\end{equation}
where $\nu(|v|)=\nu(v)$ and $C_{\mathbf{k}}$ and $C_{\Gamma}$ from (\ref{wk}) and (\ref{Gamma1}). More precisely, first choose $|v_0|>0$ large enough to have
\begin{equation*}
\frac{1}{\tilde{w}(v_0)} = \frac{(1+\rho^2 |v_0|^2)^{\beta}}{e^{\frac{|v_0|^2}{4}}} \leq \frac{1}{10C^{\prime}},
\end{equation*}
then choose $t_0>0$ as
\begin{equation*}
0<t_0 = \min\left\{\frac{\delta}{2}, \frac{t_{\mathbf{b}}(x_0,-v_0)}{2}, \frac{\delta}{|v_0|},\frac{1}{10\nu(|v_0|)}, \frac{1}{10C_{\mathbf{k}}C^{\prime}} , \frac{1}{\nu(|v_0|)}\log\left(\frac{10C_{\Gamma}(C^{\prime})^2}{10C_{\Gamma}(C^{\prime})^2 -1}\right)
\right\}.
\end{equation*}
Assume the condition for initial datum $h_0$ : there is sufficiently small $\delta^{\prime}=\delta^{\prime}(\Omega,t_0|v_0|)>0$ such that $B((-t_0|v_0| ,0,0),\delta^{\prime}) \subset \Omega$ and
\begin{equation}
h_0(x_0,v_0) \equiv ||h_0||_{\infty}>0 \ \ \ \text{for} \ \ (x,v)\in B((-t_0|v_0|,0,0);\delta^{\prime}) \times B((|v_0|,0,0);\delta^{\prime}).\label{supp1}
\end{equation}
We claim that the Boltzmann solution $h$ with such initial datum $h_0$ is not continuous at $(t_0,x_0,v_0)=(t_0,(0,0,0),(|v_0|,0,0))$. We will use a contradiction argument : Assume the Boltzmann solution $h$ is continuous at $(t_0,x_0,v_0)$, i.e (\ref{contradiction}) is valid. Choose sequences of points $(x_n^{\prime},v_n^{\prime})=((0,0,\frac{1}{n}), (|v_0|,0,0))$ and $(x_n ,v_n)=((\frac{1}{n},0,\Phi(\frac{1}{n},0)),\frac{1}{\sqrt{1+\frac{1}{n^2}}}(|v_0|,0,\frac{|v_0|}{n}))$. Because of our choice, for sufficiently large $n\in\mathbb{N}$, we have
\begin{equation*}
(x_n^{\prime}-t_0 v_n^{\prime},v_n^{\prime} )=((-t_0 |v_0| ,0, \frac{1}{n}),(|v_0|,0,0)) \in B((-t_0 |v_0|,0,0);\delta^{\prime}) \times B((|v_0|,0,0);\delta^{\prime}).
\end{equation*}
Hence the Boltzmann solution at $(t_0,x_n^{\prime},v_n^{\prime})$ is
\begin{eqnarray*}
h(t_0,x_n^{\prime},v_n^{\prime}) &=& h_0(x_n^{\prime}-t_0 v_n^{\prime},v_n^{\prime}) e^{-\nu(v_n^{\prime})t_0} + \int_0^{t_0} e^{-\nu(v_n^{\prime})(t_0 -\tau)}\left\{K_w h+w\Gamma(\frac{h}{w},\frac{h}{w})\right\}(\tau, x_n^{\prime}-v(t_0 -\tau),v_n^{\prime})d\tau\\
&=&||h_0||_{\infty}e^{-\nu(|v_n^{\prime}|)t_0} + \int_0^{t_0} e^{-\nu(|v_n^{\prime}|)(t_0 -\tau)}\left\{K_w h+w\Gamma(\frac{h}{w},\frac{h}{w})\right\}(\tau, x-v_n^{\prime}(t_0 -\tau),v_n^{\prime})d\tau. \\
\end{eqnarray*}
Using the diffuse boundary condition (\ref{diffuse}), the Boltzmann solution at $(t_0,x_n,v_n)\in[0,\infty)\times\gamma_-$ is
\begin{eqnarray*}
h(t_0,x_n,v_n) &=& \frac{1}{\tilde{w}(|v_0|)}\int_{\mathcal{V}(x_n)} h(t_0 ,x_n ,v^{\prime})\tilde{w}(v^{\prime})d\sigma(v^{\prime}).
\end{eqnarray*}
Using a pointwise boundedness (\ref{boundedness}) of $h$,  and $||h_0||_{\infty}\leq 1$, we can estimate
\begin{eqnarray*}
&&|h(t_0 ,x_n^{\prime},v_n^{\prime})-h(t_0 ,x_n ,v_n)|\\
&& \ \ \ \ \ \ \ \geq\big{|} \ ||h_0||_{\infty}e^{-\nu(|v_0|)t_0} -
\int_0^{t_0} \{C_{\mathbf{k}}C^{\prime} ||h_0||_{\infty}
+ \nu(v_n^{\prime}) e^{-\nu(v_n^{\prime})(t_0-\tau)} C_{\Gamma}(C^{\prime})^2 ||h_0||_{\infty}^2\} d\tau
-C^{\prime}||h_0||_{\infty} \frac{1}{\tilde{w}(|v_0|)}\int_{\mathcal{V}}\tilde{w}(v^{\prime})d\sigma{(v^{\prime})}
\big{|}\\
&& \ \ \ \ \ \ \ \geq  ||h_0||_{\infty} e^{-\nu(|v_0|)t_0} - t_0 C_{\mathbf{k}}C^{\prime}||h_0||_{\infty} - (1-e^{-\nu(|v_0|)t_0}) C_{\Gamma} (C^{\prime})^2 ||h_0||_{\infty}^2
-C^{\prime}||h_0||_{\infty} \frac{1}{\tilde{w}(|v_0|)}\int_{\mathcal{V}}\tilde{w}(v^{\prime})d\sigma{(v^{\prime})}
\\
&& \ \ \ \ \ \ \ = ||h_0||_{\infty} \left( e^{-\nu(|v_0|)t_0} -t_0 C_{\mathbf{k}}C^{\prime} -(1-e^{-\nu(|v_0|)t_0})C_{\Gamma}(C^{\prime})^2 -C^{\prime} \frac{1}{\tilde{w}(|v_0|)}\int_{\mathcal{V}}\tilde{w}(v^{\prime})d\sigma{(v^{\prime})}
\right) \geq \frac{||h_0||_{\infty}}{2}\neq 0,
\end{eqnarray*}
which is contradiction to (\ref{contradiction}).

\subsection{Continuity away from $\mathfrak{D}$}
Instead of using the argument of \cite{Guo08} to show continuity in the case of diffuse reflection boundary condition we will use the sequence (\ref{hm}) with the boundary condition (\ref{iterationdiffuse}) and Lemma \ref{transportequation}. This argument also gives a new proof of the continuity of Boltzmann solution in a strictly convex domain with simpler way than \cite{Guo08}.
\newline \textbf{Proof of 2 of Theorem \ref{continuityawayfromD}}
\newline We will use the sequence (\ref{hm}) with $h^{m+1}|_{t=0}=h_0$ with following boundary condition
\begin{equation}
h^{m+1}(t,x,v)=\frac{1}{\tilde{w}(v)}\int_{\mathcal{V}(x)}h^m (t,x,v^{\prime})\tilde{w}(v^{\prime})d\sigma, \label{iterationdiffuse}
\end{equation}
with $(t,x,v)\in\gamma_-$.
\newline \textbf{Step 1} : We claim that
\begin{equation}
\frac{1}{\tilde{w}(v)}\int_{\mathcal{V}(x)} h^m(t,x,v^{\prime})\tilde{w}(v^{\prime})d\sigma(v^{\prime}),
\end{equation}
is continuous function on $[0,T]\times\gamma$ even if $h^m\in L^{\infty}([0,T]\times\bar{\Omega}\times\mathbb{R}^3)$ is only continuous on $[0,T]\times\bar{\Omega}\times\mathbb{R}^3 \backslash \mathfrak{G}$. We will show as $(\bar{t},\bar{x},\bar{v})\rightarrow (t,x,v)$,
\begin{equation}
\frac{1}{\tilde{w}(v)}\int_{\mathcal{V}(x)}h^m (t,x,v^{\prime})\tilde{w}(v^{\prime})d\sigma(v^{\prime}) \rightarrow \frac{1}{\tilde{w}(\bar{v})}\int_{\mathcal{V}(\bar{x})}h^m(\bar{t},\bar{x},v^{\prime})\tilde{w}(v^{\prime})d\sigma(v^{\prime}).\label{136}
\end{equation}
Using the fact $|\mathcal{V}(x)\backslash \mathcal{V}(\bar{x})|, |\mathcal{V}(x)\backslash \mathcal{V}(\bar{x})| \rightarrow 0$ as $\bar{x}\rightarrow x$ and the exponentially decay weight function of $\tilde{w}d\sigma$ it suffices to show that
\begin{equation}
\int_{\mathcal{V}(x)\cap \mathcal{V}(\bar{x})\cap \{|v^{\prime}|\leq M\}}\{
\tilde{w}(v)^{-1} h^m (t,x,v^{\prime})\tilde{w}(v^{\prime}) - \tilde{w}(\bar{v})^{-1}h^m(\bar{t},\bar{x},v^{\prime})\tilde{w}(v^{\prime})
\} d\sigma(v^{\prime}),
\end{equation}
for sufficiently large $M>0$. Using Lemma \ref{uniformlycontinuous} we can choose open set $U_x\subset \{v^{\prime}\in \mathbb{R}^3 : |v^{\prime}|\leq M\}$ so that $|U_x|$ is small and $h^m$ is uniformly continuous on $\{|v^{\prime}|\leq M\}\backslash U_x$. Therefore we can make $\int_{\mathcal{V}(x)\cap \mathcal{V}(\bar{x})\cap\{|v^{\prime}|\leq M\}\cap U_x}$ small using the smallness of $U_x$ and make $\int_{\mathcal{V}(x)\cap \mathcal{V}(\bar{x})\cap\{|v^{\prime}|\leq M\}\backslash U_x}$ small using the uniformly continuity of $h^m$ on $\{|v^{\prime}|\leq M\}\backslash U_x$. Hence (\ref{136}) is valid.
\newline \textbf{Step 2} : We claim
\begin{equation}
h^i \text{ \ is a continuous function in \ } \mathfrak{C}_T\label{mconti}
\end{equation}
for all $i\in\mathbb{N}$ where $\mathfrak{C}_T$ is defined in (\ref{CT}). By induction choose $h^0 =0$ and (\ref{mconti}) is satisfied for $i=0$. Assume (\ref{mconti}) for all $i=0,1,2,...,m$. Let $w\Gamma_-\left(\frac{h^m}{w},\frac{h^{m+1}}{w}\right) = \nu\left(\frac{h^m}{w}\right)h^{m+1}$. Then the equation of $h^{m+1}$ is
\begin{eqnarray*}
\{\partial_t + v\cdot\nabla_x + \nu(v) +\nu\left(\frac{h^m}{w}\right) \}h^{m+1} = K_w h^m + w \Gamma_+\left(\frac{h^m}{w},\frac{h^m}{w}\right).
\end{eqnarray*}
From Theorem \ref{continuity} and Corollary \ref{continuity1} we know that $\nu\left(\frac{h^m}{w}\right)$ and $w\Gamma_+\left(\frac{h^m}{w},\frac{h^m}{w}\right)$ are both continuous in $[0,T]\times\Omega\times\mathbb{R}^3$. Because of Step 1 we know that $\frac{1}{\tilde{w}(v)}\int_{\mathcal{V}(x)} h^m(t,x,v^{\prime})\tilde{w}(v^{\prime})d\sigma(v^{\prime})$ is also continuous function on $[0,T]\times\gamma$. So we can apply Lemma \ref{transportequation} to conclude (\ref{mconti}) is valid for $i=m+1$.
\newline \textbf{Step 3} : We claim $h^m$ is a Cauchy sequence in $\mathfrak{C}_T$ for some small $T>0$. First we will compute some constants explicitly. From (\ref{normalizedconstant}) the normalized constant $c_{\mu}$ is $\left[\int_{n(x)\cdot v^{\prime}>0} e^{-\frac{|v^{\prime}|^2}{2}}\{n(x)\cdot v^{\prime}\}dv^{\prime}\right]^{-1}$.
Choose $n(x)=(1,0,0)$ and then we can compute the right hand side of above term :
\begin{eqnarray*}
\int_0^{\infty} dv_1 \ v_1 e^{-\frac{v_1^2}{2}} \int_{-\infty}^{\infty} dv_2 \ e^{-\frac{v_2^2}{2}} \int_{-\infty}^{\infty} dv_3 \ e^{-\frac{v_3^2}{2}}
= \int_0^{\infty} \ \frac{d}{dv_1}\left(-e^{-\frac{v_1^2}{2}}\right) dv_1 \times (\sqrt{2\pi})^2 = 2\pi \left[-e^{-\frac{v_1^2}{2}}\right]_{0}^{\infty} = 2\pi.
\end{eqnarray*}
Therefore we have $c_{\mu} =\frac{1}{2\pi}$. Next we will show
\begin{eqnarray}
\frac{1}{\tilde{w}(v)}\underbrace{\int_{v^{\prime}\cdot n(x)>0}\tilde{w}(v^{\prime}) d\sigma(v^{\prime})}_{\spadesuit}
\leq \tilde{C}_{\beta} \rho^{2\beta-4},
\label{coefficient}
\end{eqnarray}
where $\tilde{w}(v)^{-1}=(1+\rho^2 |v|^2)^{\beta} e^{-\frac{|v|^2}{4}}$.
We follow the computation of Lemma 25 in \cite{Guo08}. For $\frac{1}{\tilde{w}(v)}$, in the case of $\beta\rho^2 > \frac{1}{4}$ we can see that $\tilde{w}(v)^{-1}$ has a maximum value at $|v|=\sqrt{\frac{4\beta\rho^2-1}{\rho^2}}$ which is
\begin{equation}
(1+\rho^2 |v|^2)^{\beta} e^{-\frac{|v|^2}{4}}\big|_{|v|=\sqrt{\frac{4\beta\rho^2-1}{\rho^2}}} = 4^{\beta}\beta^{\beta} e^{-\beta} e^{\frac{1}{4\rho^2}} \rho^{2\beta},\label{1111}
\end{equation}
and
\begin{eqnarray*}
\spadesuit &\leq& \int_{v^{\prime}\cdot n(x)>0} \tilde{w}(v^{\prime}) d\sigma(v^{\prime}) = \frac{1}{2\pi} \int_{v_1^{\prime} >0} (1+\rho^2 |v^{\prime}|^2)^{-\beta} e^{\frac{|v^{\prime}|^2}{4}} e^{-\frac{|v^{\prime}|^2}{2}} v_1^{\prime} dv^{\prime}\\
&=&\frac{1}{2\pi} \int_{u_1>0} (1+|u|^2)^{-\beta} e^{\frac{-2|u|^2}{4\rho}} \rho^{-4} u_1 du
\leq \rho^{-4} \times \frac{1}{2\pi} \int_{u_1 >0} \frac{1}{(1+|u|^2)^{\beta -\frac{1}{2}}} du\\
&=& C_{\beta} \rho^{-4},
\end{eqnarray*}
where $\beta \geq 2$ and combining with (\ref{1111}) we conclude (\ref{coefficient}).

First we will show a boundedness (\ref{boundedh-m}).
\begin{lemma}
Let $h^m$ be a solution of (\ref{hm}) with $h^{m+1}_{t=0}=h_0$ and the boundary condition (\ref{iterationdiffuse}). Then there exist $T_*, C, \delta >0$ such that if $||h_0||_{\infty}< \delta$ then
\begin{equation*}
\sup_{0\leq s\leq T_*} ||h^m(s)||_{\infty} < C||h_0||_{\infty} \ \ \text{for all } \  m\in\mathbb{N}.
\end{equation*}
\end{lemma}
\begin{proof}
We will use mathematical induction. Choose $h^0=h_0$ and assume $||h_0||_{\infty} < \delta$ and
\begin{equation}
\sup_{0\leq s\leq T_*} ||h^i(s)||_{\infty}\leq C||h_0||_{\infty},\label{inductionhypothesis}
\end{equation}
for $i=0,1,2,...,m$, where $\delta, C, T_* >0$ will be determined later. From Lemma 24 of \cite{Guo08} the representation of $h^{m+1}$ which is a solution of (\ref{hm}) with the boundary condition (\ref{iterationdiffuse}) is given by
\begin{eqnarray}
h^{m+1}(t,x,v) &=&\mathbf{1}_{t_1 \leq 0}(t,x,v) \Big\{ \underbrace{h_0 (x-tv,v)e^{-\nu(v)t}}_{[\mathrm{initial \ data}]} + \underbrace{\int_0^t e^{-\nu(v)(t-s)}q^m (s,x-(t-s)v,v) ds}_{\mathbf{I}} \Big\}\label{hm11}\\
&+&\mathbf{1}_{0<t_1}(t,x,v)\Big\{ \underbrace{\int_{t_1}^{t} e^{-\nu(v)(t-s)}q^m (s,x-(t-s)v,v) ds}_{\mathbf{II}} + \frac{e^{-\nu(v)(t-t_1)}}{\tilde{w}(v)} \int_{\prod_{j=1}^{k}\mathcal{V}_j}H
\Big\},\label{hm12}
\end{eqnarray}
where $q^m$ was defined (\ref{qmm}) and
\begin{eqnarray}
H= \sum_{l=1}^{k} \underbrace{\mathbf{1}_{t_{l+1}\leq 0< t_l} h_0(x_l -t_l v_l ,v_l)}_{[\mathrm{initial \ data}]} d\Sigma_l(0)
+\sum_{l=1}^{k} \underbrace{\int_0^{t_l} \mathbf{1}_{t_{l+1}\leq 0 < t_l} q^{m-l}(s,x_l -(t_l -s)v_l,v_l) d\Sigma_l(s) ds}_{\mathbf{III}}\label{hm13}\\
+\sum_{l=1}^{k} \underbrace{\int_{t_{l+1}}^{t_l} \mathbf{1}_{0<t_{l+1}} q^{m-l}(s,x_l-(t_l-s)v_l,v_l) d\Sigma_l(s) ds}_{\mathbf{IV}}
+\underbrace{\mathbf{1}_{0< t_{k+1}} h^{m-k+1}(t_{k+1},x_{k+1},v_{k}) d\Sigma_{k}(t_{k+1})}_{[\mathrm{many \ bounces}]}.\label{hm14}
\end{eqnarray}
Here $d\Sigma_{k}(t_{k+1})$ is evaluated at $s=t_{k+1}$ of
\begin{equation*}
d\Sigma_l(s) = \{ \Pi_{j=l+1}^{k} d\sigma_j\} \{ e^{-\nu(v_l)(t_l -s)}\tilde{w}(v_l) d\sigma_l\} \Pi_{j=1}^{l-1}\{ e^{-\nu(v_j)(t_j -t_{j+1})} d\sigma_j\}.
\end{equation*}
\\
First we can estimate $[\mathrm{initial \ data}]$ in (\ref{hm11}) and (\ref{hm13}),
\begin{eqnarray*}
&&\int_{_{\prod_{j=1}^{k}\mathcal{V}_j}} \left\{\mathbf{1}_{t_1\leq 0}|h_0(x-tv,v)| +\frac{1}{\tilde{w}(v)}\sum_{l=1}^{k}\mathbf{1}_{t_{l+1}\leq 0< t_l} |h_0(x_l-t_l v_l,v_l)| \tilde{w}(v_l) \right\}
d\sigma_1...d\sigma_{k}\\
&\leq&\max\Big\{ 1, \frac{1}{\tilde{w}(v)} \max_{1\leq l\leq k}\int_{\prod_{j=1}^{k}\mathcal{V}_j} \tilde{w}(v_l)d\sigma_1...d\sigma_{k}
\Big\}|| h_0||_{\infty}\\
&\leq& \left\{ 1+\tilde{C}_{\beta}\rho^{2\beta-4}
\right\}||h_0||_{\infty},
\end{eqnarray*}
where we used (\ref{coefficient}).\\
Next we estimate $[\mathrm{many \ bounces}]$ term in (\ref{hm14}) which is crucial estimate in this proof. We use Lemma 23 in \cite{Guo08} to bound a contribution of $[\mathrm{many \ bounces}]$ term in (\ref{hm14}) in the last term of (\ref{hm12}) by
\begin{eqnarray*}
&&\frac{1}{\tilde{w}(v)}\int_{\prod_{j=1}^{k}\mathcal{V}_{j}} \mathbf{1}_{\{ t_{k+1}(t,x,v,v_1,v_2,...,v_{k})>0\}} \tilde{w}(v_{k}) d\sigma_{k} d\sigma_{k-1}...d\sigma_1 \times \sup_{0\leq s\leq t}||h^{m-k+1}(s)||_{\infty}\\
&\leq&
\frac{1}{\tilde{w}(v)} \int_{\mathcal{V}_k} \tilde{w}(v_k) d\sigma_k \int_{\prod_{j=1}^{k-1}\mathcal{V}_j} \mathbf{1}_{\{ t_{k}(t,x,v,v_1,...,v_{k-1})>0\}} d\sigma_{k-1}...d\sigma_1 \times \sup_{0\leq s\leq t}||h^{m-k+1}(s)||_{\infty}\\
&\leq& \tilde{C}_{\beta} \rho^{2\beta-4}\left\{\frac{1}{2}\right\}^{C_2 \rho^{5/4}} \sup_{0\leq s\leq t}||h^{m-k+1}(s)||_{\infty} \leq \tilde{C}_{\beta} \rho^{2\beta-4}\left\{\frac{1}{2}\right\}^{C_2 \rho^{5/4}} C||h_0||_{\infty},
\end{eqnarray*}
where we used (\ref{coefficient}) at the last step.
The remainders \textbf{I}, \textbf{II}, \textbf{III} and \textbf{IV} are contributions of $q^m ,..., q^{m-k}$. We introduce a notation
\begin{eqnarray}
\mathcal{H}_i &\equiv&
tC_{\mathbf{k}}\sup_{0\leq s\leq t}||h^i(s)||_{\infty} + C_{\Gamma} \sup_{0\leq s\leq t}||h^i(s)||_{\infty}\left( \sup_{0\leq s\leq t}||h^i(s)||_{\infty} + \sup_{0\leq s\leq t}||h^{i+1}(s)||_{\infty}
\right)\\
&\leq& C||h_0||_{\infty}(C_{\mathbf{k}}T_* + 2CC_{\Gamma}||h_0||_{\infty}), \label{inequalityH}
\end{eqnarray}
where the above inequality holds for $0\leq t\leq T_*$ and $i=0,1,2,...,m-1$ and
\begin{equation}
\mathcal{H}_{m} \leq (T_* C_{\mathbf{k}} + C_{\Gamma}C||h_0||_{\infty}) C ||h_0||_{\infty} + C_{\Gamma}C||h_0||_{\infty} \sup_{0\leq s\leq T_*} ||h^{m+1}(s)||_{\infty},\label{inequalityH2}
\end{equation}
where we used the induction hypothesis (\ref{inductionhypothesis}) for (\ref{inequalityH}) and (\ref{inequalityH2}).
Easily we have
\begin{eqnarray*}
\mathbf{I},\mathbf{II}&\leq&\mathcal{{H}}_m,\\
\mathbf{III},\mathbf{IV}&\leq& \sum_{l=1}^{k} \frac{1}{\tilde{w}(v)} \int_{\mathcal{V}_1} d\sigma_1..\int_{\mathcal{V}_{l-1}}d\sigma_{l-1}\int_{\mathcal{V}_{l+1}} d\sigma_{l+1}..\int_{\mathcal{V}_k}d\sigma_k
\int_{\mathcal{V}_{l}}\int_0^{t_l} \mathcal{{H}}_{m-l}
e^{-\nu(v_l)(t_l -s)}\tilde{w}(v_l)ds d\sigma_l \\
&\leq&
\sum_{l=1}^k \mathcal{H}_{m-l}\frac{1}{\tilde{w}(v)} \int_{\mathcal{V}_l} \tilde{w}(v_l) d\sigma_l
\leq \tilde{C}_{\beta}\rho^{2\beta-4}\sum_{l=1}^k \mathcal{H}_{m-l}.
\end{eqnarray*}
To summarize, we can estimate all terms of representation of $h^{m+1}(t,x,v)$ in (\ref{hm11}) to obtain
\begin{eqnarray*}
|h^{m+1} (t,x,v)|&\leq&
||h_0||_{\infty} \Big\{C \Big[ 2T_* C_{\mathbf{k}} + 2C_{\Gamma}C||h_0||_{\infty} + k\tilde{C}_{\beta} \rho^{2\beta-4} (C_{\mathbf{k}}T_* + 2CC_{\Gamma}||h_0||_{\infty}) + \tilde{C}_{\beta}\rho^{2\beta -4} \left\{\frac{1}{2}\right\}^{C_2 \rho^{5/4}}
\Big]\\
&& \ \ \ \ \ \ \ \ \ \ \  +1 + \tilde{C}_{\beta} \rho^{2\beta-4}\Big\} + C_{\Gamma}C||h_0||_{\infty} \sup_{0\leq s\leq T_*} ||h^{m+1}(s)||_{\infty}.
\end{eqnarray*}
Choose $k=\rho^{5/4}$. Choose $\rho>0$ sufficiently large so that $\tilde{C}_{\beta}\rho^{2\beta-4}\left\{\frac{1}{2}\right\}^{C_2 \rho^{5/4}}\leq \frac{1}{30}$ and then choose $T_*>0$ sufficiently small so that $T_* \times C_{\Gamma}(1+\tilde{C}_{\beta}\rho^{5/4}\rho^{2\beta-4}) \leq \frac{1}{30}$ and then choose $C>0$ sufficiently large $C> 10(1+\tilde{C}_{\beta}\rho^{2\beta-4})$ and choose $\delta = \min\left\{\frac{1}{20C_{\Gamma}C} \ , \  \frac{1}{30C_{\Gamma}}\big( C\tilde{C}_{\beta} \rho^{5/4}\rho^{2\beta-4}
\big)^{-1}
\right\}$. Finally assume $||h_0||_{\infty}\leq \delta$.
Then we have
\begin{eqnarray*}
\sup_{0\leq s\leq T_*}||h^{m+1}(s)||_{\infty}&\leq&\frac{1}{1-C_{\Gamma}C||h_0||_{\infty}} ||h_0||_{\infty} \Big\{ 1+\tilde{C}_{\beta}\rho^{2\beta-4} +C\Big[ \tilde{C}_{\beta} \rho^{2\beta-4}\left\{\frac{1}{2}\right\}^{C_2 \rho^{5/4}} + tC_{\Gamma}(1+ \tilde{C}_{\beta}\rho^{5/4}\rho^{2\beta-4}) \\
&& \ \ \ \ \ \ \ \ \ \ \ \ \ \ \ \ \ \ \ \ \ \ \ \ \ \ \ \ \ \ \ \ \ +C_{\Gamma}C||h_0||_{\infty} + 2C_{\Gamma}\tilde{C}_{\beta}\rho^{5/4}\rho^{2\beta-4}C||h_0||_{\infty}
\Big]
\Big\}\\
&\leq& \frac{20}{19}||h_0||_{\infty} \left\{ \frac{C}{10} + C\Big[ \frac{1}{30} + \frac{1}{30} + \frac{1}{20} +\frac{1}{15}
\Big]
\right\}\leq C||h_0||_{\infty}.
\end{eqnarray*}
\end{proof}
\\
Next we will show that $h^m$ is a Cauchy sequence in $L^{\infty}$.
\begin{lemma}
Let $h^m$ be a solution of (\ref{hm}) with $h^{m+1}|_{t=0} = h_0$ and the boundary condition (\ref{iterationdiffuse}). Then there exist $T_*, C, \delta >0$ so that if $||h_0||_{\infty}< \delta$ then $h^m$ is Cauchy in $L^{\infty}([0,T_*]\times\bar{\Omega}\times\mathbb{R}^3)$.
\end{lemma}
\begin{proof}
The equation of $h^{m+1}-h^m$ is
\begin{eqnarray*}
\{\partial_t + v\cdot\nabla_x +\nu\}(h^{m+1}-h^m) &=& \tilde{q}^m\\
\text{with} \ \ \ \{h^{m+1}-h^m\}|_{t=0}=0  &,&  \{h^{m+1}-h^m\}|_{\gamma_-} = \frac{1}{\tilde{w}(v)}\int_{\mathcal(x)}\{h^m (t,x,v^{\prime})-h^{m-1}(t,x,v^{\prime})\} \tilde{w}(v^{\prime})d\sigma(v^{\prime}),
\end{eqnarray*}
where $\tilde{q}^m$ is defined at (\ref{tildeqm}).
From Lemma 24 of \cite{Guo08} we have the representation
\begin{eqnarray}
\{h^{m+1}-h^m\}(t,x,v)&=&\mathbf{1}_{t_1 \leq 0}(t,x,v) \underbrace{\int_0^t e^{-\nu(v)(t-s)}\tilde{q}^m (s,x-(t-s)v,v) ds}_{\mathbf{\tilde{I}}} \label{hm+1}\\
&&+ \mathbf{1}_{0<t_1}(t,x,v)\Big\{ \underbrace{\int_{t_1}^{t} e^{-\nu(v)(t-s)}\tilde{q}^m (s,x-(t-s)v,v) ds}_{\mathbf{\tilde{II}}} + \frac{e^{-\nu(v)(t-t_1)}}{\tilde{w}(v)} \int_{\prod_{j=1}^{k}\mathcal{V}_j}\tilde{H}
\Big\},\nonumber
\end{eqnarray}
where
\begin{eqnarray*}
\tilde{H}&=& \sum_{l=1}^{k} \underbrace{\int_0^{t_l} \mathbf{1}_{t_{l+1}\leq 0 < t_l} \tilde{q}^{m-l}(s,x_l -(t_l -s)v_l,v_l) d\Sigma_l(s) ds}_{\mathbf{\tilde{III}}}\\
&&+\sum_{l=1}^{k} \underbrace{\int_{t_{l+1}}^{t_l} \mathbf{1}_{0<t_{l+1}} \tilde{q}^{m-l}(s,x_l-(t_l-s)v_l,v_l) d\Sigma_l(s) ds}_{\mathbf{\tilde{IV}}}
+\underbrace{\mathbf{1}_{0< t_{k+1}} \{h^{m-k+1}-h^{m-k}\}(t_{k+1},x_{k+1},v_{k}) d\Sigma_{k}(t_{k+1})}_{\mathrm{[[many \ bounces]]}}.
\end{eqnarray*}
First using Lemma 24 of \cite{Guo08}, we estimate $\mathrm{[[many \ bounces]]}$ term for sufficiently large $k>0$ by
\begin{eqnarray*}
&&\frac{1}{\tilde{w}(v)}\int_{\prod_{j=1}^{k}\mathcal{V}_{j}} \mathbf{1}_{\{ t_{k+1}(t,x,v,v_1,v_2,...,v_{k})>0\}} \tilde{w}(v_{k}) d\sigma_{k} d\sigma_{k-1}...d\sigma_1 \times \sup_{0\leq s\leq t}||\{h^{m-k+1}-h^{m-k}\}(s)||_{\infty}\\
&\leq&
\frac{1}{\tilde{w}(v)} \int_{\mathcal{V}_k} \tilde{w}(v_k) d\sigma_k \int_{\prod_{j=1}^{k-1}\mathcal{V}_j} \mathbf{1}_{\{ t_{k}(t,x,v,v_1,...,v_{k-1})>0\}} d\sigma_{k-1}...d\sigma_1 \times \sup_{0\leq s\leq t}||\{h^{m-k+1}-h^{m-k}\}(s)||_{\infty}\\
&\leq& \tilde{C}_{\beta} \rho^{2\beta-4}\left\{\frac{1}{2}\right\}^{C_2 \rho^{5/4}} \sup_{0\leq s\leq t}||\{h^{m-k+1}-h^{m-k}\}(s)||_{\infty}.
\end{eqnarray*}
Easily we have $ \ \mathbf{\tilde{I}}, \ \mathbf{\tilde{II}} \ \leq \ \mathcal{\delta{H}}_m  , \ \  \ \mathbf{\tilde{III}}, \ \mathbf{\tilde{IV}} \ \leq \ \tilde{C}_{\beta}\rho^{2\beta-4} \mathcal{\delta{H}}_{m-l}$
where
\begin{eqnarray*}
\delta\mathcal{H}_i &\equiv&
tC_{\mathbf{k}}\sup_{0\leq s\leq t}||\{h^i-h^{i-1}\}(s)||_{\infty} + C||h_0||_{\infty}C_{\Gamma}\big(\sup_{0\leq s\leq t}||\{h^i-h^{i-1}\}(s)||_{\infty} + \sup_{0\leq s\leq t}||\{h^{i+1}-h^i\}(s)||_{\infty}
\big)\\
&\leq& \frac{\tau}{4} \left\{ \sup_{0\leq s\leq t}||\{h^i -h^{i-1}\}(s)||_{\infty} +\sup_{0\leq s\leq t}||\{h^{i+1}-h^i\}(s)||_{\infty} \right\},
\end{eqnarray*}
with $\tau=4 \max\{ tC_{\mathbf{k}}, C||h_0||_{\infty}C_{\Gamma}
\}$.
\newline To summarize, we can estimate all terms of representation of $h^{m+1}(t,x,v)-h^m (t,x,v)$ in (\ref{hm+1}) for any $m>k$ to obtain
\begin{eqnarray*}
\sup_{0\leq s\leq t}||\{h^{m+1}-h^m\}(s)||_{\infty}
\leq
\frac{1}{1-2\tau} \bigg\{
\frac{\tau}{2}\tilde{C}_{\beta}\rho^{2\beta-4} \sum_{l=1}^{k}\Big(
\sup_{0\leq s\leq t} ||\{h^{m-l}-h^{m-l-1}\}(s)||_{\infty} + \sup_{0\leq s\leq t} ||\{h^{m-l+1}-h^{m-l}\}(s)||_{\infty}\Big) \nonumber\\
 +\frac{\tau}{2} \sup_{0\leq s\leq t} ||\{h^m -h^{m-1}\}(s)||_{\infty} +
\tilde{C}_{\beta}\rho^{2\beta-4}\left\{\frac{1}{2}\right\}^{C_2 \rho^{5/4}} \sup_{0\leq s\leq t}
||\{h^{m-k+1}-h^{m-k}\}(s)||_{\infty}\bigg\},
\end{eqnarray*}
which is our starting point. Fix a small number $\tilde{\tau}>0$ chosen later. Choose $\rho>0$ sufficiently large so that $2\tilde{C}_{\beta}\rho^{2\beta-4}\left\{\frac{1}{2}\right\}^{C_2 \rho^{5/4}} < \frac{\tilde{\tau}}{4}$ and then choose $\tau>0$ so small that $\frac{\tau/2}{1-2\tau}\tilde{C}_{\beta} \rho^{2\beta-4} < \frac{\tilde{\tau}}{4}$ and $\frac{\tau/2}{1-2\tau}< \frac{\tilde{\tau}}{4}$. Then we have
\begin{equation}
\sup_{0\leq s\leq t}||\{h^{m+1}-h^m\}(s)||_{\infty}
\leq \tilde{\tau} \left\{ \sup_{0\leq s\leq t} ||\{h^m-h^{m-1}\}(s)||_{\infty} + ...+ \sup_{0\leq s\leq t} ||\{h^{m-k+1}-h^{m-k}\}(s)||_{\infty}\right\}.\label{termbyterm}
\end{equation}
Using (\ref{termbyterm}) for $m,j\in\mathbb{N}$ so that $m-(i+1)k>0$ and $j=0,1,...,m-1 \ $it is easy to show
\begin{eqnarray*}
\sup_{0\leq s\leq t} ||\{h^{m-ik+1+j}-h^{m-ik+j}\}(s)||_{\infty} \leq \ \ \ \ \ \ \ \ \ \ \ \ \ \ \ \ \ \ \ \ \ \ \ \ \ \ \ \ \ \ \ \ \ \ \ \ \ \ \ \ \ \ \ \ \ \ \ \ \ \ \ \ \ \ \ \ \ \ \ \ \ \ \ \ \ \ \ \ \ \ \ \ \ \ \ \ \ \ \ \ \ \ \ \ \\
\tilde{\tau}(1+\tilde{\tau})^{j} \left\{ \sup_{0\leq s\leq t}||\{h^{m-ik}-h^{m-ik-1}\}(s)||_{\infty}+...+\sup_{0\leq s\leq t}||\{
h^{m-(i+1)k+1}-h^{m-(i+1)k}
\}(s)||_{\infty}
\right\}.
\end{eqnarray*}
We apply the above inequality term by term in (\ref{termbyterm}) to have
\begin{eqnarray*}
\sup_{0\leq s\leq t}||\{h^{m+1}-h^m\}(s)||_{\infty}\leq\tilde{\tau}\{(1+\tilde{\tau})^k -1\} \{ \sup_{0\leq s\leq t}||\{h^{m-k}-h^{m-k-1}\}(s)||_{\infty}+...+\sup_{0\leq s\leq t}||\{
h^{m-2k+1}-h^{m-2k}
\}(s)||_{\infty}
\}\\
\leq\tilde{\tau}\{(1+\tilde{\tau})^k -1\}^i \{ \sup_{0\leq s\leq t}||\{h^{m-ik}-h^{m-ik-1}\}(s)||_{\infty}+...+\sup_{0\leq s\leq t}||\{
h^{m-(i+1)k+1}-h^{m-(i+1)k}
\}(s)||_{\infty}
\}.
\end{eqnarray*}
Now we estimate
\begin{eqnarray*}
&&\sup_{0\leq s\leq t} ||\{h^{m}-h^n\}(s)||_{\infty}\leq \sum_{l=0}^{m-n-1} \sup_{0\leq s\leq t}||\{h^{m-l}-h^{m-l-1}\}(s)||_{\infty}\\
&\leq& \sum_{l=0}^{m-n-1} \tilde{\tau}\{(1+\tilde{\tau})^k -1\}^i \{ \sup_{0\leq s\leq t}||h^{m-ik-l-1}-h^{m-ik-l-2}||_{\infty} +...+ \sup_{0\leq s\leq t}||h^{m-(i+1)k-l} -h^{m-(i+1)k-l-1}||_{\infty}\}\\
&\leq& \sum_{l=0}^{m-n-1} \tilde{\tau} \{(1+\tilde{\tau})^k-1\}^{\left[\frac{m-l-1}{k}\right]-1}\{ \sup_{0\leq s\leq t}||h^{2k}-h^{2k-1}||_{\infty} + ...+ \sup_{0\leq s\leq t}||h^{1}-h^{0}||_{\infty}
\}
\end{eqnarray*}
\begin{eqnarray*}
&\leq& \tilde{\tau} \{(1+\tilde{\tau})^k -1\}^{\left[\frac{n}{k}\right]-1} \sum_{l=0}^{m-n-1} \{(1+\tilde{\tau})^k -1\}^{\left[\frac{m-l-1}{k}\right]-\left[\frac{n}{k}\right]}\{ \sup_{0\leq s\leq t}||h^{2k}-h^{2k-1}||_{\infty} + ...+ \sup_{0\leq s\leq t}||h^{1}-h^{0}||_{\infty}\}\\
&\leq& \tilde{\tau} \{(1+\tilde{\tau})^k -1\}^{\left[\frac{n}{k}\right]-1} \frac{1}{2-(1+\tilde{\tau})^k} \{ \sup_{0\leq s\leq t}||h^{2k}-h^{2k-1}||_{\infty} + ...+ \sup_{0\leq s\leq t}||h^{1}-h^{0}||_{\infty}\},
\end{eqnarray*}
where we choose $i=\left[\frac{m-l-1}{k}\right]-1$ so that $m-(i+1)k-l-1 \in [0,k)$. If $\tilde{\tau}>0$ is chosen sufficiently small so that ${(1+\tilde{\tau})^k-1}\leq \frac{1}{2}$ then $\{(1+\tilde{\tau})^k -1\}^{\left[\frac{n}{k}\right]-1} \rightarrow 0$ as $n\rightarrow\infty$ which implies that
\begin{equation}
\sup_{0\leq s\leq t} ||\{h^{m}-h^n\}(s)||_{\infty} \rightarrow 0,
\end{equation}
as $m,n \rightarrow \infty$. Thus $h^m$ is Cauchy in $L^{\infty}$.
\end{proof}
\newline \textbf{Step 4} : We claim that $h$ is continuous in $\mathfrak{C}$. Notice that $T$ only depends on $||h_0||_{\infty}$ and $\sup_{0\leq s\leq T}||wg(s)||_{\infty}$ (Theorem 1 of \cite{Guo08}). Using a unform bound of $\sup_{0\leq s < \infty}||h(s)||_{\infty}$, we can obtain the continuity of $h$ for all time by repeating $[0,T],[T,2T],...$.
If the boundary $\partial\Omega$ does not include a line segment (\ref{linesegment}) then every step is valid with $[0,\infty)\times\{\bar{\Omega}\times\mathbb{R}^3\}\backslash\mathfrak{D}$ instead of $\mathfrak{C}$ and $[0,T]\times\{\bar{\Omega}\times\mathbb{R}^3\}\backslash\mathfrak{D}$ instead of $\mathfrak{C}_T$.
\subsection{Propagation of Discontinuity}
\textbf{Proof of 2 of Theorem \ref{propagation}}\\
\textbf{Proof of (\ref{1sidedinequality})} : The proof is exactly same as in-flow case in Section 4.3.
\\
\textbf{Proof of (\ref{decaydiscontinuity})}
The proof is exactly same as the proof of in-flow case in Section 4.3 except \textbf{Step 2}. As we mentioned in Remark of \textbf{Step 2}, we need to show a continuity of a boundary datum on $\gamma_- \cup \gamma_0^{\mathbf{S}}$. In diffuse reflection boundary condition case, we need
\begin{eqnarray*}
0&=&[ \ h|_{[0,\infty)\times\gamma_-}]_{t,y,v} = \lim_{\delta\downarrow 0} \sup_{\begin{scriptsize}\begin{array}{ccc} t^{\prime},t^{\prime\prime}\in B(t;\delta)\\(y^{\prime},v^{\prime}), (y^{\prime\prime},v^{\prime\prime})\in \gamma_- \cap B((y,v);\delta) \backslash (y,v)  \end{array}\end{scriptsize}}
|h(t^{\prime},y^{\prime},v^{\prime})- h(t^{\prime\prime},y^{\prime\prime},v^{\prime\prime})|\\
&=& \lim_{\delta\downarrow 0} \sup_{\begin{scriptsize}\begin{array}{ccc} t^{\prime},t^{\prime\prime}\in B(t;\delta)\\(y^{\prime},v^{\prime}), (y^{\prime\prime},v^{\prime\prime})\in \gamma_- \cap B((y,v);\delta) \backslash (y,v)  \end{array}\end{scriptsize}}
\left|\frac{1}{\tilde{w}(v^{\prime})}\int_{\mathcal{V}(y^{\prime})} h(t^{\prime},y^{\prime},\mathfrak{v}) \tilde{w}(\mathfrak{v}) d\sigma(\mathfrak{v})-
\frac{1}{\tilde{w}(v^{\prime\prime})}\int_{\mathcal{V}(y^{\prime\prime})} h(t^{\prime\prime},y^{\prime\prime},\mathfrak{v}) \tilde{w}(\mathfrak{v}) d\sigma(\mathfrak{v})\right|
\end{eqnarray*}
for $(y,v)\in\gamma_- \cup \gamma_0^{\mathbf{S}}$. This is already proven in section 5.2  Continuity away from $\mathfrak{D}$.
\section{Bounce-Back Boundary Condition}
In this section, we consider the linear Boltzmann equation (\ref{hBol}) with the bounce-back boundary condition (\ref{bounceback}).
\subsection{Formation of Discontinuity}
We prove part 3 of Theorem \ref{formationofsingularity}. Without loss of generality we may assume $x_0=(0,0,0)$ and $v_0=(1,0,0)$ and $(x_0,v_0)\in\gamma_0^{\mathbf{S}}$. Locally the boundary is a graph, i.e. $\Omega\cap B(\mathbf{0};\delta) = \{(x_1,x_2,x_3)\in B(\mathbf{0};\delta) : x_3 > \Phi(x_1,x_2)\}$. The condition $(x_0,v_0)\in\gamma_0^{\mathbf{S}}$ implies $t_{\mathbf{b}}(x_0,v_0)\neq 0$ and $t_{\mathbf{b}}(x_0,-v_0)\neq 0$ which means $\Phi(\xi,0)<0$ for $\xi\in(-\delta,\delta)\backslash\{0\}$. (See Figure 3)

Assume that $||h_0||_{\infty}<\delta$ is sufficiently small so that the global solution $h$ of (\ref{hBol}) with bounce-back boundary (\ref{bounceback}) has a uniform bound (\ref{boundedness}), from Theorem 2 of \cite{Guo08}.

Recall the constants $C_{\mathbf{k}}$ and $C_{\Gamma}$ from (\ref{wk}) and (\ref{Gamma1}).
Choose $t_0 \in (0,\min\{\frac{\delta}{2},\frac{t_{\mathbf{b}}(x_0,-v_0)}{2},\frac{t_{\mathbf{b}}(x_0,v_0)}{2}\})$ sufficiently small so that
\begin{equation}
\frac{1}{2}\leq\left( e^{-\nu(1)t_0} -t_0 C_{\mathbf{k}}C^{\prime} -(1-e^{-\nu(1)t_0})C_{\Gamma}(C^{\prime})^2
\right).\label{1/2B}
\end{equation}
Assume a condition for the initial datum $h_0$ : there is sufficiently small $\delta^{\prime}=\delta^{\prime}(\Omega,t_0)>0$ such that\\ $B((-t_0 ,0,0),\delta^{\prime}), \ B((t_0 ,0,0),\delta^{\prime}) \subset \Omega$ and
\begin{eqnarray*}
&&h_0(x,v) \equiv ||h_0||_{\infty}>0 \ \ \ \text{for} \ \ (x,v)\in B((-t_0,0,0);\delta^{\prime}) \times B((1,0,0);\delta^{\prime}),\\
&&h_0(x,v) \equiv -||h_0||_{\infty}>0 \ \ \ \text{for} \ \ (x,v)\in B((t_0,0,0);\delta^{\prime}) \times B((-1,0,0);\delta^{\prime}).
\label{supp}
\end{eqnarray*}
We will use a contradiction argument : Assume the Boltzmann solution $h$ is continuous at $(t_0,x_0,v_0)$, i.e. (\ref{contradiction}) is valid.
Choose sequences of points $(x_n^{\prime},v_n^{\prime})=((0,0,\frac{1}{n}), (1,0,0))$ and $(x_n ,v_n)=((\frac{1}{n},0,\Phi(\frac{1}{n},0)),\frac{1}{\sqrt{1+\frac{1}{n^2}}}(1,0,\frac{1}{n}))$. Because of our choice, for sufficiently large $n\in\mathbb{N}$, we have
\begin{eqnarray*}
(x_n^{\prime}-t_0 v_n^{\prime},v_n^{\prime} )&=&((-t_0 ,0, \frac{1}{n}),(1,0,0)) \in B((-t_0,0,0);\delta^{\prime}) \times B((1,0,0);\delta^{\prime}),\\
(x_n - t_0 (-v_n),-v_n) &=&((\frac{1}{n}+\frac{t_0}{\sqrt{1+{1}/{n^2}}},0,\Phi(\frac{1}{n},0)+\frac{t_0}{n\sqrt{1+{1}/{n^2}}}),\frac{1}{\sqrt{1+{1}/{n^2}}}(-1,0,-\frac{1}{n}))\\
&\in& B((t_0,0,0);\delta^{\prime}) \times B((-1,0,0);\delta^{\prime}).
\end{eqnarray*}
Hence the Boltzmann solution at $(t_0,x_n^{\prime},v_n^{\prime})$ and $(t_0,x_n,v_n)$ is
\begin{eqnarray*}
&&h(t_0,x_n^{\prime},v_n^{\prime}) = ||h_0||_{\infty} e^{-\nu(v_n^{\prime})t_0} + \int_{0}^{t_0} e^{-\nu(-v_n^{\prime})(t_0-\tau)}\{ K_w h + w\Gamma\left( \frac{h}{w},\frac{h}{w}
\right)
\}(\tau, x_n^{\prime} -(-v_n^{\prime})(t_0 -\tau),-v_n^{\prime})
d\tau,\\
&&h(t_0,x_n,v_n)=h(t_0,x_n,-v_n)\\
&& \ \ \ \ \ \ \ \ \ \ \ \ \ \ \ \  =-||h_0||_{\infty} e^{-\nu(-v_n)t_0} + \int_0^{t_0} e^{-\nu(-v_n)(t_0-\tau)}\{ K_w h + w\Gamma\left( \frac{h}{w},\frac{h}{w}
\right)
\}(\tau, x_n -(-v_n)(t_0 -\tau),-v_n)
d\tau.
\end{eqnarray*}
Using a pointwise boundedness (\ref{boundedness}) of $h$ with (\ref{wk}) and (\ref{Gamma1}), we have
\begin{eqnarray*}
h(t_0,x_n^{\prime},v_n^{\prime}) &\geq&
||h_0||_{\infty}e^{-\nu(1)t_0} -t_0 C_{\mathbf{k}}C^{\prime} ||h_0||_{\infty}
-(1-e^{-\nu(1)t_0})C_{\Gamma}(C^{\prime})^2 ||h_0||_{\infty}^2
, \\
h(t_0,x_n,v_n) &\leq& -||h_0||_{\infty}e^{-\nu(1)t_0} + t_0 C_{\mathbf{k}}C^{\prime}||h_0||_{\infty} + (1-e^{-\nu(1)t_0})C_{\Gamma}(C^{\prime})^2 ||h_0||_{\infty}^2
\end{eqnarray*}
Therefore using (\ref{1/2B}),
\begin{eqnarray*}
h(t_0 ,x_n^{\prime},v_n^{\prime})-h(t_0 ,x_n ,v_n)\geq 2||h_0||_{\infty} \left( e^{-\nu(1)t_0} -t_0 C_{\mathbf{k}}C^{\prime} -(1-e^{-\nu(1)t_0})C_{\Gamma}(C^{\prime})^2
\right) \geq ||h_0||_{\infty}\neq 0,
\end{eqnarray*}
which is contradiction to (\ref{contradiction}).
\subsection{Continuity away from $\mathfrak{D}_{bb}$}
We recall some basic facts to study the bounce-back boundary condition from \cite{Guo08}.
\begin{definition}\cite{Guo08}
\label{bouncebackcycles}(\textbf{Bounce-Back Cycles}) Let $(t,x,v)\notin
\gamma _{0}\cup \gamma _{-}.$ Let $(t_{0},x_{0},v_{0})=(t,x,v)$ and
inductively define for $k\geq 1:$%
\begin{equation*}
(t_{k+1},x_{k+1},v_{k+1})=(t_{k}-t_{\mathbf{b}}(x_{k},v_{k}),x_{\mathbf{b}%
}(x_{k},v_{k}),-v_{k}).
\end{equation*}%
We define the back-time cycles as:
\begin{equation}
X_{\mathbf{cl}}(s;t,x,v)=\sum_{k}\mathbf{1}_{[t_{k+1},t_{k})}(s)%
\{x_{k}+(s-t_{k})v_{k}\},\text{ \ }V_{\mathbf{cl}}(s;t,x,v)=\sum_{k}\mathbf{1%
}_{[t_{k+1},t_{k})}(s)v_{k}.\label{bouncebackcycle}
\end{equation}
\end{definition}
Clearly, we have $v_{k+1}\equiv (-1)^{k+1}v,$ for $k\geq 1,$
\begin{equation}
x_{k}=\frac{1-(-1)^{k}}{2}x_{1}+\frac{1+(-1)^{k}}{2}x_{2},
\label{bouncebackx}
\end{equation}%
where $x_1 = x-t_{\mathbf{b}}(x,v) v$ and $x_2 = x-[2t_{\mathbf{b}}(x,v)+t_{\mathbf{b}}(x,-v)](-v)$
and let $d=t_{1}-t_{2},$ then $t_{k}-t_{k+1}=d\geq t_{\mathbf{b}}(t,x,v)>0$ for $k\geq 1,$ and
\begin{eqnarray}
t_1(t,x,v) &=& t-t_{\mathbf{b}}(x,v) \ , \nonumber\\
t_2(t,x,v) &=& t_1 -t_{\mathbf{b}}(x_1,v_1) = t_1 -(t_{\mathbf{b}}(x,v)+ t_{\mathbf{b}}(x_1,v_1)) = t_1 -(2t_{\mathbf{b}}(x,v)+t_{\mathbf{b}}(x,-v)) \ , \nonumber\\
&\vdots&\nonumber\\
t_{k+1}(t,x,v) &=& t_1 -k (2t_{\mathbf{b}}(x,v)+t_{\mathbf{b}}(x,-v)).\label{tk1}
\end{eqnarray}
\begin{lemma}\cite{Guo08}
\label{bouncebackformula}Let $h_{0}\in L^{\infty }(\Omega \times \mathbf{R}%
^{3})$ and $\phi(t,x,v)$ with $\sup_{[0,T]\times\Omega}|\phi(\cdot,\cdot,v)|< \infty$. There exists a unique solution $G(t)h_{0}$ of
\begin{equation*}
\{\partial _{t}+v\cdot \nabla _{x} +\phi\}\{G(t)h_{0}\}=0,\text{ \ \ \ \ }%
\{G(0)h_{0}\}=h_{0},  \label{gh0}
\end{equation*}%
with the bounce-back reflection $\{G(t)h_{0}\}(t,x,v)=\{G(t)h_{0}\}(t,x,-v)$
for $x\in \partial \Omega .$ For almost any $(x,v)\in \bar{\Omega}\times
\mathbf{R}^{3}\setminus \gamma _{0},$
\begin{equation}
\{G(t)h_{0}\}(t,x,v)=\sum_{k}\mathbf{1}_{[t_{k+1},t_{k})}(0) h_{0}\left( X_{\mathbf{cl}}(0),V_{\mathbf{cl}}(0)\right) e^{-\int_0^t \phi(\tau,X_{\mathbf{cl}}(\tau),V_{\mathbf{cl}}(\tau))d\tau},
\label{bouncebackformular}
\end{equation}%
where $X_{\mathbf{cl}}(\tau)= X_{\mathbf{cl}}(\tau;t,x,v)$ and $V_{\mathbf{cl}}(\tau)=V_{\mathbf{cl}}(\tau;t,x,v)$ in (\ref{bouncebackcycle}).
\end{lemma}
Next we prove a generalized version of Lemma 16 in \cite{Guo08}.
\begin{lemma}[{Continuity away from $\mathfrak{D}_{bb}$ : Transport Equation}] \label{bbtransportequation}
Let $\Omega$ be an open subset of $\mathbb{R}^3$ with a smooth boundary $\partial\Omega$ and an initial datum $h_0(x,v)$ be continuous in $\Omega\times\mathbb{R}^3 \cup \{\gamma_- \cup \gamma_+ \cup \gamma_0^{{I}}\}$. Also assume $q(t,x,v)$ and $\phi(t,x,v)$ be continuous in the interior of $[0,T]\times\Omega\times\mathbb{R}^3$ and $\sup_{[0,T]\times\Omega\times\mathbb{R}^3} |q(t,x,v)|< \infty$ and $\sup_{[0,T]\times\Omega}|\phi(\cdot,\cdot,v)|< \infty$ for all $v\in\mathbb{R}^3$. Let $h(t,x,v)$ be the solution of
\begin{equation*}
\{\partial_t + v\cdot\nabla_x + \phi \}h = q \ , \ \ \ h(0,x,v)=h_0 \ , \ \ \ h|_{\gamma_-}(t,x,v) =h(t,x,-v).
\end{equation*}
Assume the compatibility condition on $\gamma_- \cup \gamma_0^{I-}$
\begin{equation*}
h_0(x,v) = h_0(x,-v).
\end{equation*}
Then the Boltzmann solution $h(t,x,v)$ is continuous on $\mathfrak{C}_{bb}$. Further, if the boundary $\partial\Omega$ does not include a line segment (\ref{linesegment}) then $h(t,x,v)$ is continuous on a complementary set of the discontinuity set, i.e. $[0,T]\times \{\bar{\Omega}\times\mathbb{R}^3\}\backslash \mathfrak{D}_{bb}$.
\end{lemma}
\begin{proof}
The proof is similar to the proof of Lemma 16 of \cite{Guo08}. Take any point $(t,x,v)\in[0,T]\times\bar{\Omega}\times\mathbb{R}^3$ and recall its back-time cycle and (\ref{bouncebackformular}). Assume $t_{m+1}\leq 0 < t_m$. Using (\ref{bouncebackformular}), $h(t,x,v)$ takes the form
\begin{eqnarray}
&&h_0(x_m -t_m v_m ,v_m) e^{ -\sum_{k=0}^{m-1} \int_{t_{k+1}}^{t_k} \phi(\tau,x_k -(t_k -\tau)v_k,v_k)d\tau -\int_0^{t_m} \phi(\tau,x_m -(t_m -\tau)v_m ,v_m)d\tau}\nonumber\\
&& \ \ \ \  + \sum_{k=0}^{m-1}\int_{t_{k+1}}^{t_k} q(s,x_k-(t_k-s)v_k,v_k) e^{ -\sum_{i=0}^{k-1} \int_{t_{i+1}}^{t_i} \phi(\tau,x_i -(t_i -\tau)v_i,v_i)d\tau -\int_s^{t_k} \phi(\tau,x_k -(t_k -\tau)v_k ,v_k)d\tau}\nonumber\\
&& \ \ \ \ + \int_{0}^{t_m} q(s,x_m-(t_m-s)v_m,v_m) e^{ -\sum_{i=0}^{m-1} \int_{t_{i+1}}^{t_i} \phi(\tau,x_i -(t_i -\tau)v_i,v_i)d\tau -\int_s^{t_m} \phi(\tau,x_m -(t_m -\tau)v_m ,v_m)d\tau}.\label{hmmm}
\end{eqnarray}
Take any point $(t,x,v)\in \mathfrak{C}_{bb}$. By the definition of $\mathfrak{C}_{bb}$ we assume that $(x,v)\in\Omega\times\mathbb{R}^3$ or $(x,v)\in \gamma_- \cup \gamma_0^{I-}$ and we can separate three cases : $t-t_{\mathbf{b}}(x,v)<0 \ , \ (x_{\mathbf{b}}(x,v),v)\in\gamma_- \cup \gamma_0^{I-}$ with $t < 2t_{\mathbf{b}}(x,v)+t_{\mathbf{b}}(x,-v),$ and $(x_{\mathbf{b}}(x,-v),-v)\in\gamma_- \cup\gamma_0^{I-}$ with $(x_{\mathbf{b}}(x,v),v)\in\gamma_-\cup\gamma_0^{I-}$.\\
\newline \underline{Case of $t<t_{\mathbf{b}}(x,v)$} \ Simply we have $h(t,x,v)=h_0(x-tv,v)e^{-\int_0^t \phi(\tau,x-(t-\tau)v,v)d\tau} + \int_0^t q(s,x-(t-s)v,v)e^{\int_s^t\phi(\tau,x-(t-\tau)v,v)d\tau}ds$ and use the continuity of $q(t,x,v)$ and $\phi(t,x,v)$ to conclude the continuity of $h(t,x,v)$.\\
\newline \underline{Case of $(x_{\mathbf{b}}(x,v),v)\in\gamma_- \cup \gamma_0^{I-}$ with $t < 2t_{\mathbf{b}}(x,v)+t_{\mathbf{b}}(x,-v)$} \ A representation of $h(t,x,v)$ takes the form
\begin{eqnarray*}
&&h_0(x_1 -t_1 v_1, v_1) e^{-\int_{t_1}^{t} \phi(\tau,x-(t-\tau)v,v)d\tau - \int_0^{t_1}\phi(\tau,x_1 -(t_1-\tau)v_1,v_1)d\tau} + \int_{t_1}^t q(s,x-(t-s)v,v) e^{-\int_s^t \phi(\tau,x-(t-\tau)v,v)d\tau}ds\\
&& \ \ \ \ + \int_0^{t_1} q(s,x_1 -(t_1-s)v_1,v_1) e^{-\int_{t_1}^t \phi(\tau,x-(t-\tau)v,v) d\tau -\int_s^{t_1}\phi(\tau,x_1-(t_1-\tau)v_1,v_1)d\tau}ds.
\end{eqnarray*}
Thanks to Lemma \ref{huang} and Lemma \ref{tbcon}, the condition $(x_{\mathbf{b}}(x,v),v)\in\gamma_- \cup \gamma_0^{I-}$ implies continuity of $x_1(x,v) = x-x_{\mathbf{b}}(x,v) \ , \ t_1(t,x,v) = t-t_{\mathbf{b}}(x,v)$. Therefore we can show the continuity of $h(t,x,v)$.\\
\newline \underline{Case of $(x_{\mathbf{b}}(x,-v),-v)\in\gamma_- \cup\gamma_0^{I-}$ with $(x_{\mathbf{b}}(x,v),v)\in\gamma_-\cup\gamma_0^{I-}$} \ We have (\ref{hmmm}) for $h(t,x,v)$. Thanks to (\ref{bouncebackx}) and (\ref{tk1}) and Lemma \ref{huang} and Lemma \ref{tbcon}, the conditions $(x_{\mathbf{b}}(x,-v),-v)\in\gamma_- \cup\gamma_0^{I-}$ and $(x_{\mathbf{b}}(x,v),v)\in\gamma_-\cup\gamma_0^{I-}$ imply continuity of $x_k(x,v), v_k(x,v), t_k(t,x,v)$. Therefore we can show the continuity of $h(t,x,v)$.
\end{proof}
\\
\newline \textbf{Proof of Part 1 of Theorem 3}\\
Following the in-flow and diffuse cases, we use the iteration scheme (\ref{hm}) which is equivalent to (\ref{hmequation}) with bounce-back boundary condition $h^{m+1}|_{\gamma_-}(t,x,v) = h^{m+1}(t,x,-v)$ and an initial condition $h^{m+1}|_{t=0}=h_0$.
\newline \textbf{Step 1} : We claim that $h^i$ is a continuous function in $\mathfrak{C}_{bb,T}$ for all $i\in\mathbb{N}$ and for any $T>0$ where $\mathfrak{C}_{bb,T}= \mathfrak{C}_{bb} \cap \{[0,T]\times\bar{\Omega}\times\mathbb{R}^3\}$. Choose $h^0 \equiv 0$ and use mathematical induction. Assume $h^i$ is continuous $\mathfrak{C}_{bb,T}$ for $i=0,1,2,...,m$. Apply Lemma \ref{bbtransportequation} to concluse that $h^{m+1}$ is continuous in $\mathfrak{C}_{bb,T}$.
\newline \textbf{Step 2} : We claim that there exist $C>0$ and $\delta>0$ such that if $C||h_0||_{\infty}<\delta$ then there exists $T=T(C,\delta)>0$ so that $\sup_{0\leq s\leq T}||h^m(s)||_{\infty}\leq C||h_0||_{\infty}$ and $\{h^m\}_{m=0}^{\infty}$ is Cauchy in $L^{\infty}([0,T]\times\bar{\Omega}\times\mathbb{R}^3)$. \\
Fisrt we will show the boundedness using mathematical induction. Assume $\sup_{0\leq s\leq T}||h^{m}(s)||_{\infty}\leq C||h_0||_{\infty}$ where $T>0$ will be chosen later. Applying Lemma \ref{bouncebackformula}, $\phi$ and $q$ correspond with $\nu$ and the right hand side of (\ref{hm}) respectively to have a representation of $h^{m+1}(t,x,v)$
\begin{equation*}
h_0(X_{\mathbf{cl}}(0),V_{\mathbf{cl}}(0)) e^{-\nu(v)t} + \int_0^t e^{-\nu(v)(t-s)}\{K_w h^m + w\Gamma_+\left(\frac{h^m}{w},\frac{h^m}{w}\right)-w\Gamma_-\left(\frac{h^m}{w},\frac{h^{m+1}}{w}\right)\}(s,X_{\mathbf{cl}}(s),V_{\mathbf{cl}}(s))
ds,
\end{equation*}
where $\left[X_{\mathbf{cl}}(s),V_{\mathbf{cl}}(s)\right]=\left[X_{\mathbf{cl}}(s;t,x,v),V_{\mathbf{cl}}(s;t,x,v)\right]$ is in (\ref{bouncebackcycle}). The above term is bounded by
\begin{equation*}
||h_0||_{\infty} + t C_{\mathbf{k}}\sup_{0\leq s\leq t}||h^m(s)||_{\infty} + C_{\Gamma} \sup_{0\leq s\leq t}||h^m(s)||_{\infty} \sup_{0\leq s\leq t}(||h^m(s)||_{\infty} + ||h^{m+1}(s)||_{\infty}),
\end{equation*}
where the constants are coming from basic estimates, (\ref{wk}) and (\ref{Gamma1}). Choose $C>4$ and $\delta< \frac{1}{2C_{\Gamma}}$ and $T=\frac{C-3}{2C_{\mathbf{k}}C}$. Then we have $\sup_{0\leq s\leq T}||h^{m+1}(s)||_{\infty}\leq C||h_0||_{\infty}$.\\
Next we will show $\{h^m\}_{m=0}^{\infty}$ is Cauchy in $L^{\infty}([0,T]\times\bar{\Omega}\times\mathbb{R}^3)$. Recall $\tilde{q}^m(t,x,v)$ from (\ref{tildeqm}). The equation of $h^{m+1}-h^m$ is (\ref{hm1}) with a zero initial condition $(h^{m+1}-h^m)|_{t=0}=0$ and a bounce-back boundary condition $(h^{m+1}-h^m)|_{\gamma_-}(t,x,v) = (h^{m+1}-h^m)(t,x,-v)$. Applying Lemma \ref{bouncebackformula} to (\ref{hm1}) we have
\begin{equation*}
(h^{m+1}-h^m)(t,x,v) = \int_0^t e^{-\nu(v)(t-s)} \tilde{q}^m (s,X_{\mathbf{cl}}(s),V_{\mathbf{cl}}(s)) ds,
\end{equation*}
where $\left[X_{\mathbf{cl}}(s),V_{\mathbf{cl}}(s)\right]=\left[X_{\mathbf{cl}}(s;t,x,v),V_{\mathbf{cl}}(s;t,x,v)\right]$ is in (\ref{bouncebackcycle}). Then we have exactly same estimates of in-flow case to conclude $\{h^m\}$ is Cauchy.
\newline \textbf{Step 3} : Same argument as in-flow case but substitute $\mathfrak{C}_{bb,T} \ , \ \mathfrak{C}_{bb} \ , \ \mathfrak{D}_{bb,T} \ , \ \mathfrak{D}_{bb}$ for $\mathfrak{C}_{T} \ , \ \mathfrak{C} \ , \ \mathfrak{D}_{T} \ , \ \mathfrak{D}$ respectively.
\subsection{Propagation of Discontinuity}
\textbf{Proof of 2 of Theorem \ref{propagation}}\\
\textbf{Proof of (\ref{1sidedinequality})} : The proof is exactly same as in-flow case in Section 4.3.
\\
\textbf{Proof of (\ref{decaydiscontinuity}) }
Recall that we have $[h(t_0)]_{x_0,v_0}\neq 0$ for $(x_0,v_0)\in\gamma_0^{\mathbf{S}}$ and $t_0 \in (0,\min\{t_{\mathbf{b}}(x_0,-v_0),t_{\mathbf{b}}(x_0,v_0)\})$. The proof is exactly same as the proof of in-flow case in Section 4.3 except \textbf{Step 2}. We need to show a continuity of a boundary datum on $\gamma_- \cup \gamma_0^{\mathbf{S}}$. In bounce-back reflection boundary condition case, we need to show
\begin{eqnarray*}
&&0=[ \ h|_{[0,\infty)\times\gamma_-}]_{t_0,x_0,v_0} = \lim_{\delta\downarrow 0} \sup_{\begin{scriptsize}\begin{array}{ccc} t^{\prime},t^{\prime\prime}\in B(t;\delta)\\(y^{\prime},v^{\prime}), (y^{\prime\prime},v^{\prime\prime})\in \gamma_- \cap B((x_0,v_0);\delta) \backslash (x_0,v_0)  \end{array}\end{scriptsize}}
|h(t^{\prime},y^{\prime},v^{\prime})- h(t^{\prime\prime},y^{\prime\prime},v^{\prime\prime})|.
\end{eqnarray*}
Because $(y^{\prime},v^{\prime})$ is in the incoming boundary $\gamma_-$, using the bounce-back boundary condition, we have $h(t^{\prime},y^{\prime},v^{\prime}) = h(t^{\prime},y^{\prime},-v^{\prime})$. Further due to the condition $0< t_0 < t_{\mathbf{b}}(x_0,-v_0)$ we have $0<t^{\prime}< t_{\mathbf{b}}(y^{\prime},-v^{\prime})$ and
\begin{eqnarray*}
h(t^{\prime},y^{\prime},v^{\prime}) &=& h(t^{\prime},y^{\prime},-v^{\prime}) = h_0 (y^{\prime}+t^{\prime}v^{\prime},v^{\prime})e^{-\nu(v^{\prime})t^{\prime}-\int_0^{t^{\prime}}\nu(\sqrt{\mu}\frac{h}{w})(\tau,y^{\prime}+(t^{\prime}-\tau)v^{\prime},v^{\prime})d\tau}\\
&&+\int_0^{t^{\prime}} \{K_w h + w\Gamma_+ (\frac{h}{w},\frac{h}{w})\}(s,y^{\prime}+(t^{\prime}-s)v^{\prime},v^{\prime}) e^{-\nu(v^{\prime})(t^{\prime}-s)-\int_0^{t^{\prime}}\nu(\sqrt{\mu}\frac{h}{w})(\tau,y^{\prime}+(t^{\prime}-\tau)v^{\prime},v^{\prime})d\tau}
ds,
\end{eqnarray*}
and similar representation for $h(t^{\prime},y^{\prime},v^{\prime})$. Using the continuity of $\nu(\sqrt{\mu}\frac{h}{w}), K_w h$ and $w\Gamma_+ (\frac{h}{w},\frac{h}{w})$ we have
\begin{eqnarray*}
0=[ \ h|_{[0,\infty)\times \gamma_-}]_{t_0,x_0,v_0} &=& \lim_{\delta\downarrow 0} \sup_{\begin{scriptsize}\begin{array}{ccc} t^{\prime},t^{\prime\prime}\in B(t;\delta)\\(y^{\prime},v^{\prime}), (y^{\prime\prime},v^{\prime\prime})\in \gamma_- \cap B((x_0,v_0);\delta) \backslash (x_0,v_0)  \end{array}\end{scriptsize}}
|h_0(y^{\prime}+t^{\prime}v^{\prime},v^{\prime})-h_0(y^{\prime\prime}+t^{\prime\prime}v^{\prime\prime},v^{\prime\prime})|\\
&&\times e^{-\nu(v_0)t_0 -\int_0^{t_0}\nu(\sqrt{\mu}\frac{h}{w})(\tau,x_0+(t_0-\tau)v_0,v_0)d\tau},
\end{eqnarray*}
where we used the continuity of the initial datum $h_0$ in the last equality.
\\
\\
\newline \textbf{Acknowledgements. } The author is indebted to his advisor Yan Guo for helpful discussions. The research is supported in part by FRG No. 524230.


\begin{thebibliography}{99}
\bibitem{Aoki} Aoki, K. : private communications.
\bibitem{ABL}  Arlotti, L.; Banasiak, J.; Lods, B. : On general transport equations with abstract boundary conditions. The case of divergence free force field, \textit{preprint} 2009.
\bibitem{ABDG} Aoki, K.; Bardos, C.; Dogbe, C.; Golse, F. : A note on the propagation of boundary induced discontinuities in kinetic theory, \textit{Math. Models Methods Appl. Sci.} 11 (2001), no. 9, 1581--1595.
\bibitem{ATAG} Aoki, K.; Takata, S.; Aikawa, H.; Golse, F. : A rarefied gas flow caused by a discontinuous wall temperature, \textit{Phys. Fluids} 13 (2001), no. 9, 2645--2661
\bibitem{B-D} Boudin, L.; Desvillettes ,L.: On the singularities of the global small solutions of the full Boltzmann equation, \textit{Monatshefte Math.}
    131 (2000), 91--108.
\bibitem{B-D1} Bernis, L.; Desvillettes, L. : Propagation of singularities for classical solutions of the Vlasov-Poisson-Boltzmann equation, \textit{Discrete Contin. Dyn. Syst.} 24 (2009), no. 1, 13--33.
\bibitem{Cercignani} Cercignani, C. : Propagation phenomena in classical and relativistic rarefied gases, \textit{Transport Theory Statist. Phys.} 29 (2000), no. 3-5, 607--614.
\bibitem{DLY}  Duan, R.; Li, M.-R.; Yang, T. : Propagation of singularities in the solutions to the Boltzmann equation near equilibrium, \textit{Math. Models Methods Appl. Sci.} 18 (2008), no. 7, 1093--1114.
\bibitem{G-M-P} Greenberg, W.; van der Mee, C.; Protopopescu, V. : Boundary value problems in abstract kinetic theory. Operator Theory: Advances and Applications, 23. Birkhauser Verlag, Basel, 1987.
\bibitem{Guiraud} Guiraud, J.-P. : An $H$ theorem for a gas of rigid spheres in a bounded domain, Theories cinetiques classiques et relativistes, pp. 29--58. Centre Nat. Recherche Sci., Paris, 1975
\bibitem{Guo95} Guo, Y.: Singular solutions of the Vlasov-Maxwell system on a half line, \textit{Arch. Rational Mech. Anal.} 131 (1995), no. 3, 241-304.
\bibitem{Guo03} Guo, Y.: Classical solutions to the Boltzmann equation for molecules with an angular cutoff, \textit{Arch. Ration. Mech. Anal.} 169 (2003),
    no. 4, 305-353.
\bibitem{Guo08} Guo, Y.: Decay and Continuity of Boltzmann Equation in Bounded Domains, To appear in \textit{Arch. Rat. Mech. Anal.}
\bibitem{Grad} Grad, H.: Asymptotic theory of the Boltzmann equation. II. Rarefied gas dynamics.In: Proceedings of the 3rd international Symposium, pp.26-59, Paris, 1962
\bibitem{Hwang04} Hwang, H-J : Regularity for the Vlasov-Poisson system in a convex domain, \textit{SIAM J. Math. Anal.}, 36 (2004), no.1, 121--171
\bibitem{H-V} Hwang, H-J; Velazquez J.: Global existence for the Vlasov-Poisson system in bounded domains, To appear in \textit{Arch. Rat. Mech. Anal.}
\bibitem{kim} Kim, C.: Boltzmann Equation in a Bounded Domain : Specular Reflection with 2-D General Geometry, \textit{in preparation}.
\bibitem{Lions} Lions, P.-L.: Compactness in Boltzmann's equation via Fourier integral operators and applications. I, II, III. \textit{J. Math. Kyoto Univ.} 34 (1994), no. 2, 391--427, 429--461, 539--584.
\bibitem{Mischler} Mischler, S. : On the initial boundary value problem for the Vlasov-Poisson-Boltzmann system. \textit{Comm. Math. Phys.} 210 (2000), no. 2, 447--466.
\bibitem{Sone} Sone, Y. : Molecular gas dynamics. Theory, techniques, and applications. Modeling and Simulation in Science, Engineering and Technology. Birkhauser Boston, Inc., Boston, MA, 2007.
\bibitem{S-T} Sone, Y. ; Takata, S. : Discontinuity of the velocity distribution function in a rarefied gas around a convex body and the S layer at the bottom of the Knudsen layer, \textit{Transport Theor. Stat. Phys.} 21 (1992), 501--530.
\bibitem{T-S-A} Takata, S. ; Sone, Y.; Aoki, K. : Numerical analysis of a uniform flow of a rarefied gas past a sphere on the basis of the Boltzmann equation for hard-sphere molecules, \textit{Phys. Fluids} \textbf{A 5} (1993), 716--737.
\bibitem{ukai} Ukai, S. ; Solutions of the Boltzmann equation. Patterns and waves, 37--96, Stud. Math. Appl., 18, North-Holland, Amsterdam, 1986
\bibitem{V} Villani, C. ; A review of mathematical topics in collisional kinetic theory. Handbook of mathematical fluid dynamics, Vol. I, 71--305, North-Holland, Amsterdam, 2002.
\bibitem{voigt} Voigt, J.: Functional analytic treatment of the initial boundary value problem for collisionless gases, Habilitationsschrift, Munchen, 1981 (http://www.math.tu-dresden.de/~voigt/vopubl/habilschr/habil80.pdf)
\bibitem{Wennberg94} Wennberg, B.: Regularity in the Boltzmann Equation and the Radon Transform, \textit{Commun. in P.D.E.} 19(1994), 2057-2074.
\bibitem{Wennberg97} Wennberg, B.: The geometry of binary collisions and generalized Radon transforms, \textit{Arch. Rational Mech. Anal.} 139 (1997), no. 3, 291--302.
\end{thebibliography}
\end{document}